\documentclass[10pt,reqno]{amsart}

\usepackage[utf8]{inputenc}
\usepackage[english]{babel}
\usepackage[dvipsnames]{xcolor}
\usepackage{amsmath, amssymb, amsthm}
\usepackage{hyperref}
\usepackage{mathrsfs}
\usepackage{tikz-cd}
\usepackage{enumerate}
\usepackage{esint}
\usepackage{xcolor}
\usepackage{mathtools}
\usepackage{bm}
\usepackage{comment}

\usepackage{xparse}% 

\usepackage{color}
\usepackage{enumerate}

\newtheorem{theorem}{Theorem}[section]
\newtheorem{lemma}[theorem]{Lemma}
\newtheorem{corollary}[theorem]{Corollary}
\newtheorem{proposition}[theorem]{Proposition}

\theoremstyle{remark}
\newtheorem{remark}[theorem]{\bf{Remark}}

\theoremstyle{definition}
\newtheorem{assumption}[theorem]{Assumption}

\newtheorem{definition}[theorem]{Definition}

\newcommand\cbrk{\text{$]$\kern-.15em$]$}}
\newcommand\opar{\text{\,\raise.2ex\hbox{${\scriptstyle
|}$}\kern-.34em$($}}
\newcommand\cpar{\text{$)$\kern-.34em\raise.2ex\hbox{${\scriptstyle |}$}}\,}

\newcommand{\aint}{-\hspace{-0.38cm}\int}

\newcommand\bE{\mathbb{E}}

\newcommand\bH{\mathbb{H}}

\newcommand\bL{\mathbb{L}}

\newcommand\bN{\mathbb{N}}

\newcommand\bP{\mathbb{P}}

\newcommand\bR{\mathbb{R}}

\newcommand\bZ{\mathbb{Z}}

\newcommand\cF{\mathcal{F}}
\newcommand\cG{\mathcal{G}}

\newcommand\cI{\mathcal{I}}

\newcommand\cK{\mathcal{K}}

\newcommand\cM{\mathcal{M}}

\newcommand\cO{\mathcal{O}}

\newcommand\cS{\mathcal{S}}

\newcommand{\mysection}[1]{\section{#1}
\setcounter{equation}{0}}

\newcommand{\Ccinf}{C_{c}^{\infty}}
\newcommand{\R}{\mathbb{R}}

\begin{document}

\title[Evoution equations with space-time anisotropic non-local operators]
{A regularity theory for evolution equations with space-time anisotropic non-local operators in mixed-norm Sobolev spaces}

\author{Jae-Hwan Choi}
\address{School of Mathematics, Korea Institute for Advanced Study, 85 Hoegi-ro, Dongdaemun-gu, Seoul 02455, Republic of Korea} 
\email{jhchoi@kias.re.kr}

\author{Jaehoon Kang}
\address{Department of Applied Mathematics and Institute for Integrated Mathematical Sciences, Hankyong National University, 327 Jungang-ro, Anseong-si, Gyeonggi-do 17579, Republic of Korea} 
\email{jaehoon.kang@hknu.ac.kr}

\author{Daehan Park}
\address{Department of Mathematics, Kangwon National University, 1 Kangwondaehakgil, Chucheon-si, Gangwon-do, 24341, Republic of Korea} 
\email{daehanpark@kangwon.ac.kr}

\author{Jinsol Seo}
\address{School of Mathematics, Korea Institute for Advanced Study, 85 Hoegi-ro, Dongdaemun-gu, Seoul 02455, Republic of Korea} 
\email{seo9401@kias.re.kr}

\subjclass[2020]{Primary: 26A33, 35S10, 47G20,
Secondary: 30H25,  46B70,  46E35 }

\keywords{Space-time non-local equation, Anisotropic non-local operator, Caputo fractional derivative, Subordinate Brownian motion, Initial value problem, Generalized real interpolation, Trace theorem}

\begin{abstract}
In this article, we study the regularity of solutions to inhomogeneous time-fractional evolution equations involving anisotropic non-local operators in mixed-norm Sobolev spaces of variable order, with non-trivial initial conditions.
The primary focus is on space-time non-local equations where the spatial operator is the infinitesimal generator of a vector of independent subordinate Brownian motions, making it the sum of subdimensional non-local operators.
A representative example of such an operator is $(\Delta_{x})^{\beta_{1}/2}+(\Delta_{y})^{\beta_{2}/2}$.
We establish existence, uniqueness, and precise estimates for solutions in corresponding Sobolev spaces.
Due to singularities arising in the Fourier transforms of our operators, traditional methods involving Fourier analysis are not directly applicable. 
Instead, we employ a probabilistic approach to derive solution estimates. 
Additionally, we identify the optimal initial data space using generalized real interpolation theory.
\end{abstract}

\maketitle

\tableofcontents 

\mysection{Introduction}
\subsection{Motivations and goals}
Anisotropic non-local operators such as $\Delta^{\beta_{1}/2}_{x} + \Delta^{\beta_{2}/2}_{y}$ are important in describing phenomena that exhibit distinct behaviors in different coordinate directions.
Applications of anisotropic non-local operators appear frequently in various scientific fields; see, for instance, \cite{DWZ20anisotropic,de2011anisotropic,H08MRI}.
Additionally, there has been significant theoretical development and practical applications of space-time non-local operators.
Examples include the derivation of space-time fractional Fokker–Planck–Kolmogorov equations within fractional kinetics frameworks \cite{Z02,ZEN97} and the study of space-time non-local diffusion-advection equations \cite{GK25,P15diffadv}.

Motivated by these applications and developments, we study the following fractional evolution equation involving anisotropic spatial non-local operators:
\begin{align*}
\frac{1}{\Gamma(1-\alpha)} \int_{0}^{t} (t-s)^{-\alpha} \left( u(s,\vec{x}) - u_{0}(\vec{x}) \right) \mathrm{d}s = \int_{0}^{t} \left( \sum_{i=1}^{\ell} \phi_{i}(\Delta_{x_{i}}) u(s,\vec{x}) + f(s,\vec{x}) \right) \mathrm{d}s, \quad (t,\vec{x}) \in (0,T)\times \mathbb{R}^{d},
\end{align*}
where $\alpha \in (0,1)$ and the spatial dimension $d$ is composed of $\ell$ sub-dimensions $d_{1},\dots,d_{\ell}$, so that $\bR^{d} = \bR^{d_{1}}\times \dots \times \bR^{d_{\ell}}$. Each point $\vec{x}\in \bR^{d}$ can thus be represented as
\begin{equation}\label{eqn 07.11.19:58}
\vec{x} = (x_{1},\dots, x_{\ell}), \quad x_{i}=(x^{1}_{i},\dots,x^{d_{i}}_{i})\in \bR^{d_{i}}, \quad i=1,\dots,\ell.
\end{equation}
The spatial non-local operators $\phi_i(\Delta_{x_i})$ are defined by
\begin{align*}
\phi_{i}(\Delta_{x_{i}})g(\vec{x}) := b_i\Delta_{x_i}g(\vec{x})+\int_{\mathbb{R}^{d_i}}\left(g(x_1,\dots,x_{i-1},x_i+y_i,x_{i+1},\dots,x_{\ell})-g(\vec{x})-1_{|y_i|\leq 1}y_i\cdot\nabla_{x_i}g(\vec{x})\right)J_{\phi_i}(y_i)\,\mathrm{d}y_i,
\end{align*}
where $b_i\geq0$ and $\Delta_{x_{i}}$ denotes the standard $d_{i}$-dimensional Laplacian.
Differentiating in time, the equation can equivalently be expressed in terms of the Caputo fractional derivative $\partial^{\alpha}_{t}$ as follows:
\begin{equation}\label{eqn 07.05.17:01}
\partial^{\alpha}_{t}u(t,\vec{x}) =  \sum_{i=1}^{\ell} \phi_{i}(\Delta_{x_{i}})  u(t,\vec{x}) + f(t,\vec{x}), \quad (t,\vec{x}) \in (0,T)\times \bR^{d}, \quad u(0,\vec{x}) = u_{0}(\vec{x}).
\end{equation}
The objectives of this article are three-fold:
\begin{itemize}
\item Identify the optimal initial data space $X$ (trace and extension theorem for \eqref{eqn 07.05.17:01}).
\item Prove existence and uniqueness of solutions to \eqref{eqn 07.05.17:01} in $L_q((0,T);L_p)$.
\item Obtain maximal regularity estimates for solutions to \eqref{eqn 07.05.17:01}, specifically
\begin{align}\label{eqn 02.13.17:11}
\|\partial^{\alpha}_{t}u\|_{L_{q}((0,T);L_{p})} + \left\|\sum_{i=1}^{\ell} \phi_{i}(\Delta_{x_{i}})  u\right\|_{L{q}((0,T);L_{p})} \leq C \left( \|u_{0}\|_{X} + \|f\|_{L_{q}((0,T);L_{p})} \right), \quad 1<p,q<\infty.
\end{align}
\end{itemize}

\subsection{Historical Results}

In this subsection, we summarize some known results from the literature concerning the fractional evolution equations, and PDEs involving anisotropic non-local operators. 
For a more comprehensive historical overview beyond the scope of this article, we refer the reader to the introduction of \cite{CKP23}.

\emph{Evolution equations with time fractional derivative.}
The Sobolev regularity theory for fractional evolution equations initially focused on equations involving second-order differential operators.
For instance, I. Kim, K.-H. Kim, and S. Lim \cite{KKL17} studied fractional diffusion-wave-type equations (\textit{i.e.}, $\alpha \in (0,2)$) with second-order differential operators having continuous coefficients in mixed-norm Lebesgue spaces. 
B.-S. Han, K.-H. Kim, and D. Park \cite{HKP20weighted} investigated the weighted counterpart of \cite{KKL17}, which was subsequently extended to higher regularity by D. Park \cite{P23weighted}.
A particularly challenging research direction has involved relaxing the continuity assumptions on coefficients, significantly advanced by H. Dong and D. Kim. 
Detailed unweighted results can be found in \cite{DK19fractional,DK20div,DK2021}, while their weighted analogues are presented in \cite{DK21weighted,DK2023}.
Additionally, H. Dong and Y. Liu \cite{DY22fractional} provided weighted results specifically for $\alpha \in (1,2)$.

The regularity theory for fractional evolution equations involving non-local operators is a natural subsequent research direction. 
K.-H. Kim, D. Park, and J. Ryu \cite{KPR21nonlocal} explored evolution equations with time fractional derivatives and variable-order spatial non-local operators in mixed-norm Lebesgue spaces. 
The assumptions regarding spatial non-local operators were further relaxed by J. Kang and D. Park \cite{KP23}, who studied equations associated with infinitesimal generators of general Lévy processes.  
Additionally, H. Dong and Y. Liu \cite{DY23nonlocal} investigated fractional evolution equations involving space-dependent non-local operators.
We also refer readers to \cite{AL24abstract,P13,Z05abstract,Z06abstract} for alternative approaches to abstract Volterra equations.

One of the important research directions to study fractional evolution equations is the trace theorem.
D. Kim and K. Woo \cite{KW2025} provided trace theorems for fractional evolution equations involving second-order divergence and non-divergence operators. 
Additionally, J.-H. Choi, J. B. Lee, J. Seo, and K. Woo \cite{CLSW2023trace} established trace theorems for generalized time fractional equations within a generalized real interpolation framework. 
More references on this topic can be found in the introductions of \cite{CLSW2023trace,KW2025}.

\emph{PDEs with anisotropic non-local operators.}
We consider anisotropic non-local operators of the following form:
\begin{equation}\label{eqn 07.30.14:45}
\mathcal{L}_{\beta_{1},\beta_{2}}:=\Delta^{\beta_{1}/2}_{x^{1}} + \Delta^{\beta_{2}/2}_{x^{2}}, \quad (x^{1},x^{2})\in \bR^{2}, \quad \beta_{1},\beta_{2} \in (0,2).
\end{equation}
R. Mikulevi{\v{c}}ius and C. Phonsom \cite{mikulevivcius2017p, mikulevivcius2019cauchy} investigated the Sobolev regularity theory for parabolic PDEs involving scalable non-local operators. 
J.-H. Choi and I. Kim \cite{CK2020} extended these results specifically to the case of homogeneous parabolic PDEs.
A. de Pablo, F. Quir\'os, and A. Rodr\'iguez  studied the well-posedness and regularity of very weak solutions to anisotropic non-local parabolic PDEs defined by operators of the form
$$
\mathcal{L}u(x)=\frac{1}{2}\int_{\mathbb{R}^d}(u(x+y)+u(x-y)-2u(x))\nu(\mathrm{d}y),
$$
where $\nu$ is the L\'evy measure given by
$$
\nu(A):=\int_{\mathbb{S}^{d-1}}\int_0^{\infty}1_A(r\theta)\frac{\mathrm{d}r}{r^{1+\alpha}}\mu(\mathrm{d}\theta),
$$
and $\mu$ is a nondegenerate finite surface measure defined on $\mathbb{S}^{d-1}$.
In particular, if $d=2$ and
$$
\mu(\mathrm{d}\theta):=\epsilon_{(1,0)}(\mathrm{d}\theta)+\epsilon_{(0,1)}(\mathrm{d}\theta),
$$
where $\epsilon_{(1,0)}$ and $\epsilon_{(0,1)}$ denote Dirac measures centered at $(1,0)$ and $(0,1)$ respectively, then we have $\mathcal{L}=\mathcal{L}_{\alpha,\alpha}$.
Recently, H. Dong and J. Ryu \cite{DR2024} developed the weighted Sobolev regularity theory for elliptic and parabolic PDEs in $C^{1,\tau}$-domains associated with the operator $\mathcal{L}$.
However, these earlier works exclusively considered operators \eqref{eqn 07.30.14:45} with $\beta_1=\beta_2$.
J.-H. Choi, J. Kang, and D. Park \cite{CKP23} subsequently developed the Sobolev regularity theory for elliptic and parabolic PDEs with $\mathcal{L}_{\beta_1,\beta_2}$ for arbitrary $\beta_1,\beta_2\in(0,2)$.

Although not covered in this article, an interesting anisotropic nonlocal operator is given by
$$
Lu(x)=\int_{\mathbb{R}^d}\frac{u(x+y)-u(x)-\nabla u(x)\cdot y\mathbf{1}_{|y|\leq1}}{|y_1|^{d+\beta_1}+\cdots+|y_d|^{d+\beta_d}}\mathrm{d}y.
$$
L.A. Caffarelli, R. Leit\~ao, and J.M. Urbano developed the regularity theory for fully nonlinear integro-differential equations involving $L$.
A version of Caffarelli-Silvestre's extension problem \cite{C2007extension} associated with $L$ was explored by  R. Leit\~ao \cite{Leit20}. 
E.B. dos Santos, R. Leit\~ao \cite{SL21} studied the H\"older regularity theory for equations involving $L$-like operators.
R. Leit\~ao \cite{Leit23} also established Sobolev regularity theory for equations involving $L$, following the spirit of \cite{DK12elliptic}.

\subsection{Description of Approaches}

We now describe the approach employed in this article.
For parabolic PDEs involving anisotropic non-local operators of the form
\begin{equation*}
\partial_{t} u = \sum_{i=1}^{\ell} \phi_{i}(\Delta_{x_{i}})u  + f, \quad u(0)=0,
\end{equation*}
the solution $u$ admits the following representation:
\begin{align}\label{eqn 08.26.15:26}
u(t,\vec{x}) = \int_{0}^{t} \int_{\bR^{d}} p(t-s,\vec{x}-\vec{y})f(s,\vec{y}) \mathrm{d}\vec{y}\mathrm{d}s = \int_{0}^{t} \mathbb{E}[f(s,\vec{x}-\vec{X}_{t-s})] \mathrm{d}s,
\end{align}
where $p(t,\vec{x})$ is the transition density of the \emph{independent array of subordinate Brownian motion} $\vec{X}_{t}$.
One difficulty arises in proving the maximal regularity estimates \eqref{eqn 02.13.17:11}. 
A natural approach to obtain \eqref{eqn 02.13.17:11} is the Calder\'on-Zygmund approach based on the Fourier transform. Specifically, the Fourier transform of our spatial operator is given by
$$
\mathcal{F}_d\left[\sum_{i=1}^{\ell} \phi_{i}(\Delta_{x_{i}})u(t,\cdot)\right](\vec{\xi})=-\sum_{i=1}^{\ell}  \phi_{i}(|\xi_{i}|^{2})\mathcal{F}_d[u(t,\cdot)](\vec{\xi}) \quad \vec{\xi} = (\xi_{1},\dots,\xi_{\ell}) \in \mathbb{R}^{d_{1}}\times \cdots \times \mathbb{R}^{d_{\ell}}=\mathbb{R}^d.
$$
However, singularities arise when estimating derivatives of the symbol $m(\vec{\xi}):=-\sum_{i=1}^{\ell} \phi_{i}(|\xi_{i}|^{2})$ due to its \textit{coordinate-wise} symmetry.
Consequently, classical multiplier theorems such as those by Mikhlin and Marcinkiewicz are \textit{not applicable}, even in simpler parabolic PDE cases (see \cite[Remark 2.14]{CKP23}). 
Thus, directly applying existing results on time non-local equations such as \cite{AL24abstract,Z05abstract,Z06abstract} to establish \eqref{eqn 02.13.17:11} is nontrivial.
This motivates us to revisit and adapt the Calder\'on–Zygmund theory and seek a suitable representation analogous to \eqref{eqn 08.26.15:26} for the time non-local setting.

If we replace the time variable $t$ of the process $\vec{X}_{\cdot}$ by the inverse $R_{t}$ (with transition density $\varphi(t,r)$) of an $\alpha$-stable process, then the resulting transition density $q(t,\vec{x})$ of $\vec{X}_{R_{t}}$ serves as the fundamental solution to the fractional PDE
$$
\partial^{\alpha}_{t} q= \sum_{i=1}^{\ell} \phi_{i}(\Delta_{x_{i}})q.
$$
This allows us to represent the solution of the fractional PDE
$$
\partial_t^{\alpha}u=\sum_{i=1}^{\ell} \phi_{i}(\Delta_{x_{i}})u+f,\quad u(0)=0,
$$
as
\begin{equation}\label{eqn 12.06.14:38}
u(t,\vec{x}) = \int_{0}^{t} \int_{\bR^{d}} D^{1-\alpha}_{t}\mathbb{E}[f(s,\vec{x}-\vec{X}_{R_{t-s}})]\,\mathrm{d}s = \int_{0}^{t} \int_{\mathbb{R}^{d}} D^{1-\alpha}_{t}q(t-s,\vec{y})f(s,\vec{x}-\vec{y}) \,\mathrm{d}\vec{y}\,\mathrm{d}s,
\end{equation}
where $D^{1-\alpha}_{t}$ denotes the Riemann–Liouville fractional derivative of order $1-\alpha$, and the transition density $q(t,\vec{x})$ is given by the integral representation
\begin{equation}\label{eqn 07.29.16:40}
q(t,\vec{x}) = \int_{0}^{\infty} p(r,\vec{x}) \varphi(t,r) \,\mathrm{d}r,
\end{equation}
where $p(r,\vec{x})$ is the transition density of $\vec{X}_r$ and $\varphi(t,r)$ is the transition density of $R_t$.
For detailed derivations of \eqref{eqn 12.06.14:38} and \eqref{eqn 07.29.16:40}, we refer to Section \ref{sec 07.31.15:14} and Lemma \ref{lem 06.24.15:35}.

We now briefly outline our approach to establish \eqref{eqn 02.13.17:11}.
The proof of \eqref{eqn 02.13.17:11} consists of three parts:
\begin{itemize}
    \item Upper bound estimates for the heat kernel $q(t,\vec{x})$ defined by \eqref{eqn 07.29.16:40}: Section \ref{sec 07.31.15:14}.
    \item $BMO$-$L_{\infty}$ estimates of the solution $\sum_{i=1}^{\ell}\phi_i(\Delta_{x_i})u(t,\vec{x})$ in \eqref{eqn 12.06.14:38}: Section \ref{25.04.07.15.44}.
    \item Initial trace theorem: Section \ref{sec 01.17.17:00}.
\end{itemize}

The first part involves establishing appropriate upper bound estimates for $q(t,\vec{x})$. When $\alpha=1$ (the classical parabolic case), each component of the process $\vec{X}_t=(X^1_t,\dots,X^{\ell}_t)$ is independent, yielding
\begin{equation}
\label{25.04.07.15.22}
p(t,\vec{x})=p_1(t,x_1)\times\cdots\times p_{\ell}(t,x_{\ell}),
\end{equation}
where each $p_i(t,x_i)$ is the transition density of $X^i_t$.
The product structure \eqref{25.04.07.15.22} directly provides upper bound estimates for $p$ based on the known estimates for $p_i$.
However, since $q(t,\vec{x})$ is the transition density of $(X^{1}_{R_{t}}, \dots, X^{\ell}_{R_{t}})$, whose component processes are no longer independent, we cannot easily expect an estimate of the form
\begin{align*}
\left| q(t,\vec{x}) \right| \leq G_{1}(t,x_{1}) \times \cdots \times G_{\ell}(t,x_{\ell}),
\end{align*}
where $G_{i}(t,x_{i})$ suitably bounds the transition density $q_{i}(t,x_{i})$ of $X^{i}_{R_{t}}$.
Therefore, obtaining proper upper bound estimates for $q$ requires a detailed analysis of the representation \eqref{eqn 07.29.16:40}, combined with existing estimations of $p$ from \cite{CKP23}.
Furthermore, since the given process $\vec{X}_{t}$ lacks global symmetry in $\mathbb{R}^{d}$, there is no straightforward criterion to derive estimates for $q$ from the estimates for $p$.
These complexities necessitate more sophisticated estimations compared to those previously considered in the literature (see, \textit{e.g.}, \cite{KP23,KPR21nonlocal}).

The second part is to establish the $BMO$–$L_{\infty}$ estimate of solutions, specifically
\begin{equation}\label{eqn 07.29.17:13}
\left\| \sum_{i=1}^{\ell}\phi_i(\Delta_{x_i})  u \right\|_{BMO} \lesssim \|f\|_{L_{\infty}}.
\end{equation}
From the representation \eqref{eqn 12.06.14:38}, we have
$$
\sum_{i=1}^{\ell}\phi_i(\Delta_{x_i}) u(t,\vec{x})=\int_{0}^{t} \int_{\mathbb{R}^{d}} D^{1-\alpha}_{t}\sum_{i=1}^{\ell}\phi_i(\Delta_{x_i}) q(t-s,\vec{y})f(s,\vec{x}-\vec{y}) \,\mathrm{d}\vec{y}\,\mathrm{d}s=:\mathcal{G}f(t,\vec{x}).
$$
Thus, the estimate \eqref{eqn 07.29.17:13} is equivalent to
\begin{equation}
\label{25.04.05.09.14}
\left\| \mathcal{G}f\right\|_{BMO} \lesssim \|f\|_{L_{\infty}}.
\end{equation}
If the kernel of the operator $\mathcal{G}$, defined by
\begin{equation}
\label{25.04.08.09.16}
(t,\vec{x})\mapsto D_t^{1-\alpha}\sum_{i=1}^{\ell}\phi_i(\Delta_{x_i})q(t,\vec{x}),
\end{equation}
were integrable on $(0,T)\times\mathbb{R}^d$, then the estimate \eqref{25.04.05.09.14} would follow immediately.
However, integrability of the kernel \eqref{25.04.08.09.16} is generally not expected (see Lemma \ref{lem 06.08.14:52}).
Consequently, to obtain \eqref{25.04.05.09.14}, we utilize detailed upper bound estimates for the heat kernel $q$.
Unlike the parabolic case \cite{CKP23}, in which the heat kernel $p$ admits coordinate-wise separable estimates \eqref{25.04.07.15.22}, our kernel $q$ involves intricate, intertwined estimates across all coordinates. Hence, even when analyzing the mean oscillation of $\mathcal{G}f$ in a single coordinate $x_{i}$, we cannot ignore the influence of other variables. To isolate behavior along the $x_{i}$-direction while still accounting for this dependency, we derive the following natural bound:
$$
\left| \int_{\bR^{d-d_{i}}} q(t,\vec{x}) \, \mathrm{d}\hat{x}_i  \right| \leq G_{i}(t,x_{i}) \quad \left( \hat{x}_i = (x_{1},\dots,x_{i-1},x_{i+1},\dots,x_{\ell}) \right),
$$
reflecting that the integral above represents the transition density of the component $X^{i}_{R_{t}}$.

The third part involves establishing the initial trace theorem, identifying the optimal initial data space. To achieve this, we rely on the trace results presented in \cite{CLSW2023trace,KW2025}. Specifically, if $q\in(1,\infty)$ and $\alpha\in(1/q,1]$, then it is known that
\begin{equation}
\label{25.04.08.11.07}
X=(H_p^{\vec{\phi},2},L_p)_{\frac{1}{\alpha q},q},
\end{equation}
where $X$ denotes the optimal initial data space appearing in \eqref{eqn 02.13.17:11}. A more explicit characterization of \eqref{25.04.08.11.07} is desirable for broader applicability. However, to the authors' best knowledge, even for the specific case $\vec{\phi}(\lambda)=(\lambda_1^{1/2},\lambda_2)$—that is, $H_p^{\vec{\phi},2}=W_p^{1,2}(\mathbb{R}\times\mathbb{R})$—such a detailed characterization remains unresolved. The primary difficulty arises from classical Littlewood–Paley operators, which are optimized for isotropic rather than anisotropic differentiability.
To overcome this, we introduce a modified Littlewood–Paley operator $\Delta_j^{\vec{\phi}}$ tailored to the symbol $\sum_{i=1}^{\ell}\phi_i(|\xi_i|^2)$, thereby capturing the anisotropic differentiability effectively. 
Additionally, following the approach in \cite{KW2025}, we extend the trace theorem to the range $\alpha\in(0,1/q]$.

\subsection{Notations}
We finish the introduction with some notations. We use $``:="$ or $``=:"$ to denote a definition. The symbol $\bN$ denotes the set of positive integers and $\bN_0:=\bN\cup\{0\}$. Also we use $\bZ$ to denote the set of integers. For any $a\in \bR$, we denote $\lfloor a \rfloor$ the greatest integer less than or equal to $a$. As usual $\bR^d$ stands for the Euclidean space of points $x=(x^1,\dots,x^d)$. We set
$$
B_r(x):=\{y\in \bR^{d} : |x-y|<r\}, \quad \bR_+^{d+1} := \{(t,x)\in\bR^{d+1} : t>0 \}.
$$
For $i=1,\ldots,d$,
multi-indices $\sigma=(\sigma_{1},\ldots,\sigma_{d})$,
 and functions $u(t,x)$ we set
$$
\partial_{x^{i}}u=\frac{\partial u}{\partial x^{i}}=D_{i}u.
$$
We also use the notation $D_{x}^{m}$ for arbitrary partial derivatives of
order $m$ with respect to $x$.
For an open set $\cO$ in $\bR^{d}$ or $\bR^{d+1}$, $C_c^\infty(\cO)$ denotes the set of infinitely differentiable functions with compact support in $\cO$. By 
$\cS=\cS(\bR^d)$ we denote the  class of Schwartz functions on $\bR^d$.
$\cS'=\cS'(\bR^d)$ denotes the dual space of $\mathcal{S}$.
For $p\geq1$, by $L_{p}$ we denote the set
of complex-valued Lebesgue measurable functions $u$ on $\R^{d}$ satisfying
\[
\left\Vert u\right\Vert _{L_{p}}:=\left(\int_{\R^{d}}|u(x)|^{p} \mathrm{d}x\right)^{1/p}<\infty.
\]
Generally, for a given measure space $(X,\mathcal{M},\mu)$, $L_{p}(X,\cM,\mu;F)$
denotes the space of all $F$-valued $\mathcal{M}^{\mu}$-measurable functions
$u$ so that
\[
\left\Vert u\right\Vert _{L_{p}(X,\cM,\mu;F)}:=\left(\int_{X}\left\Vert u(x)\right\Vert _{F}^{p}\mu(\mathrm{d}x)\right)^{1/p}<\infty,
\]
where $\mathcal{M}^{\mu}$ denotes the completion of $\cM$ with respect to the measure $\mu$.
We also denote by $L_{\infty}(X,\cM,\mu;F)$ the space of all $\mathcal{M}^{\mu}$-measurable functions $f : X \to F$ with the norm
$$
\|f\|_{L_{\infty}(X,\cM,\mu;F)}:=\inf\left\{r\geq0 : \mu(\{x\in X:\|f(x)\|_F\geq r\})=0\right\}<\infty.
$$
If there is no confusion for the given measure and $\sigma$-algebra, we usually omit the measure and the $\sigma$-algebra. 
For any given function $f:X \to \bR$, we denote its inverse (if it exists) by $f^{-1}$. Also, for $\nu \in \bR\setminus\{-1\}$ and nonnegative function $f$, we denote $f^{\nu}(x)=(f(x))^{\nu}$.
We denote $a\wedge b := \min\{a,b\}$ and $a\vee b:=\max\{a,b\}$. By $\cF$ and $\cF^{-1}$ we denote the $d$-dimensional Fourier transform and the inverse Fourier transform respectively, i.e.
$$
\cF[f](\xi):=\int_{\bR^d} \mathrm{e}^{-i\xi\cdot x} f(x)\mathrm{d}x, \quad \cF^{-1}[f](\xi):=\frac{1}{(2\pi)^d}\int_{\bR^d} \mathrm{e}^{i\xi\cdot x} f(x)\mathrm{d}x. 
$$
For any $a,b>0$, we write $a\simeq b$ if there is a constant $c>1$ independent of $a,b$ such that $c^{-1}a\leq b\leq ca$. We use
$$
\sum_{i=1}^{k}a_{i} , \quad \prod_{i=1}^{k} a_{i}
$$
to denote the summation and the product of indexed numbers. If the given index set is not well-defined, we define the summation as $0$ and the product as $1$.   For any complex number $z$, we denote $\Re[z]$ and $\Im[z]$ as the real and imaginary parts of $z$. If we write $C=C(\dots)$, this means that the constant $C$ depends only on what are in the parentheses. The constant $C$ can differ from line to line.

\mysection{Main Results} \label{Main result sec}

\subsection{Definition of Non-Local Operators}
We begin by introducing the mathematical formulation of the nonlocal operators in our main equation \eqref{eqn 07.05.17:01}. 

$\bullet$ \textbf{Definition of time-nonlocal operators}
\\
For $\alpha>0$ and $\varphi\in L_{1}((0,T))$, the \textit{Riemann-Liouville fractional integral} of the order $\alpha$ is defined as
$$
I_{t}^{\alpha}\varphi:=\frac{1}{\Gamma(\alpha)}\int_{0}^{t}(t-s)^{\alpha-1}\varphi(s)\mathrm{d}s, \quad 0\leq t\leq T.
$$
For convenience, we set $I^0\varphi:=\varphi$.
Let $n\in \bN$ be such that  $\alpha \in [n-1, n)$. 
Suppose that  $\varphi(t)$ is  $(n-1)$-times continuously differentiable and that $\left(\frac{\mathrm{d}}{\mathrm{d}t}\right)^{n-1} I_t^{n-\alpha}  \varphi$ is absolutely continuous on $[0,T]$.
Then the \textit{Riemann-Liouville fractional derivative} $D_{t}^{\alpha}$ and the \textit{Caputo fractional derivative} $\partial_{t}^{\alpha}$ of order $\alpha$ are defined as
\begin{equation}
                          \label{eqn 4.15}
D_{t}^{\alpha}\varphi:=\left(\frac{\mathrm{d}}{\mathrm{d}t}\right)^{n}\left(I_{t}^{n-\alpha}\varphi\right),
\end{equation}
and
\begin{align*}
\partial_{t}^{\alpha}\varphi= D_{t}^{\alpha} \left(\varphi(t)-\sum_{k=0}^{n-1}\frac{t^{k}}{k!}\varphi^{(k)}(0)\right).
\end{align*}
 Using  Fubini's theorem, we obtain the following composition property of fractional integrals: for any $\alpha,\beta\geq 0$,
 \begin{equation}
                                                          \label{eqn 4.15.3}
I^{\alpha}_tI^{\beta}_t \varphi=I^{\alpha+\beta}_t \varphi, \quad
\text{$(a.e.)$} \,\, t\leq T.
\end{equation}
It is important to note that if  $\varphi(0)=\varphi'(0)=\cdots=\varphi^{(n-1)}(0)=0$, then $D^{\alpha}_t\varphi=\partial^{\alpha}_t \varphi$.
Furthermore, from \eqref{eqn 4.15.3} and \eqref{eqn 4.15}, we obtain the following fundamental properties: for any  $\alpha,\beta\geq 0$,
\begin{equation*}
D^{\alpha}_tD^{\beta}_t=D^{\alpha+\beta}_t, \quad D^{\alpha}_t I_{t}^{\beta} \varphi=
D_{t}^{\alpha-\beta}\varphi,
\end{equation*}
where for $\alpha<0$, we define $D_t^{\alpha}\varphi:=I_t^{-\alpha}\varphi$.
Additionally, if
$\varphi(0)=\varphi^{(1)}(0)=\cdots = \varphi^{(n-1)}(0)=0$ 
then by definition of $\partial^{\alpha}_{t}$, we have
\begin{equation*}
I^{\alpha}_{t}\partial^{\alpha}_{t}u=I^{\alpha}_{t}D^{\alpha}_{t}\varphi=\varphi.
\end{equation*}

$\bullet$ \textbf{Definition of spatial non-local operators} 
\\
We now define the spatial nonlocal operator $\vec{\phi} \cdot \Delta_{\vec{d}}$. 
Let $B=(B_t)_{t\geq0}$ be a $d$-dimensional Brownian motion, and let $S = (S_{t})_{t\geq0}$ be a real-valued increasing L\'evy  process that is independent of $B_{t}$, and starts at $0$ with the Laplace transform given by
$$
\bE[\mathrm{e}^{-\lambda S_t}]:=\int_{\Omega} \mathrm{e}^{-\lambda S_t(\omega)} \,\bP(\mathrm{d}\omega)=\mathrm{e}^{-t\phi(\lambda)},\quad \forall (t,\lambda)\in[0,\infty)\times\bR_+.
$$
The process $X=(B_{S_t})_{t\geq0}$ is called a \textit{subordinate Brownian motion} (SBM) with subordinator $S$, and its infinitesimal generator is defined as
$$
\phi(\Delta_{x})f(x)=\phi(\Delta)f(x):=\lim_{t \downarrow 0} \frac{ \bE[f(x+X_t)] -f(x)}{t}.
$$
It follows that $S$ is a subordinator if and only if the Laplace exponent $\phi$ of $S$ is a \textit{Bernstein function}, meaning that $\phi$ is a nonnegative continuous function on $[0,\infty)$ satisfying
$$
(-1)^{n}D^{n}\phi(\lambda) \leq 0 \quad \forall\, \lambda > 0, \quad \forall\, n\in \mathbb{N}.
$$
Furthermore, $\phi$ admits the following representation (see, \textit{e.g.}, \cite[Theorem 3.2]{schbern})
\begin{align}
\label{23.03.08.12.52}
    \phi(\lambda)=b\lambda+\int_{0}^{\infty}\left(1-\mathrm{e}^{-\lambda t}\right)\,\mu(\mathrm{d}t) \quad \left( \int_0^{\infty}\left(1\wedge t\right)\mu(\mathrm{d}t)<\infty \right).
\end{align}
Here, the constant $b\geq0$ is called the \emph{drift} of $\phi$, and $\mu$ is referred to as the \emph{L\'evy measure} of $\phi$.
According to \cite[Theorem 31.5]{sato1999}, $\phi(\Delta_x)$ has the following equivalent representations:
\begin{align}\label{fourier200408}
\phi(\Delta_x)f(x) &= b\Delta_x f + \int_{\mathbb{R}^d}\left(f(x+y)-f(x)-\nabla_xf(x)\cdot y\mathbf{1}_{|y|\leq1}\right)J(y)\mathrm{d}y,\nonumber\\
&= \mathcal{F}^{-1}[-\phi(|\cdot|^2)\mathcal{F}[f]](x),
\end{align}
where $J(y) := j(|y|)$ and the function $j:(0,\infty)\rightarrow(0,\infty)$ is given by
\begin{equation*}
J(y) = j(|y|)=\int_{(0,\infty)} (4\pi t)^{-d/2} \mathrm{e}^{-|y|^2/(4t)} \mu(\mathrm{d}t).
\end{equation*}
Recalling \eqref{eqn 07.11.19:58}, for any vector $\vec{x} \in \bR^{\vec{d}} := \bR^{d_{1}} \times \cdots \times \bR^{d_{\ell}}$, we use the notation
$$
\vec{x}=(x_{1},\dots,x_{\ell}), \quad x_{i}=(x^{1}_{i},\dots,x^{d_{i}}_{i})\in \bR^{d_{i}} \quad (i=1,\dots,\ell).
$$
Let $X^{1},\dots,X^{\ell}$ be independent $d_{i}$-dimensional ($i=1,\dots,\ell$) SBMs with characteristic exponents $\phi_{i}(|\cdot|^{2})$, respectively. 
We say that $\vec{X} = (X^{1},\cdots,X^{\ell})$ is an \emph{independent array of SBM (IASBM)}. 
Then, $\vec{X}$ is an $\bR^d$-valued L\'evy process, and its characteristic exponent is given by
\begin{equation*}
    \bE[\mathrm{e}^{i\vec{\xi}\cdot \vec{X}_t}]=\prod_{i=1}^{\ell}\exp\left(-t\phi_i(|\xi_i|^2)\right),\quad \vec{\xi}=(\xi_1,\cdots,\xi_{\ell})\in\bR^{d}.
\end{equation*}
Since each component of $\vec{X}$ is independent, the infinitesimal generator of $\vec{X}$ can be expressed as
\begin{align*}
\lim_{t\downarrow0}\frac{\bE[f(\vec{x}+\vec{X}_t)]-f(\vec{x})}{t}
=\sum_{i=1}^{\ell}\phi_{i}(\Delta_{x_i})f(x)=:(\vec{\phi}\cdot\Delta_{\vec{d}})f(x),
\end{align*}
where $\Delta_{\vec{d}}:=(\Delta_{x_1},\Delta_{x_2},\cdots,\Delta_{x_{\ell}})$, $\Delta_{x_i}$ is the Laplacian operator on $\bR^{d_i}$. 
Using the vector notations
$$
d=\vec{d}\cdot\vec{1}=\sum_{i=1}^{\ell}d_{i}, \quad \vec{1}:=(1,\cdots,1),\,\vec{d}:=(d_1,\cdots,d_{\ell})\in \bN^{\ell}, \quad \vec{\phi}=(\phi_{1},\cdots,\phi_{\ell}),
$$
we express the operator $\vec{\phi}\cdot\Delta_{\vec{d}}$ as follows (recalling \eqref{fourier200408})
\begin{align*}
    (\vec{\phi}\cdot\Delta_{\vec{d}})f(\vec{x})&=\vec{b}\cdot\Delta_{\vec{d}}f(\vec{x})+\int_{\bR^{d}}(f(\vec{x}+\vec{y})-f(\vec{x})-\nabla_{\vec{x}}f(\vec{x})\cdot \vec{y} \, \mathbf{1}_{|\vec{y}|\leq1})\vec{1}\cdot\vec{J}(\mathrm{d}\vec{y})\\
    &=\cF^{-1}_{d}\left[-\sum_{i=1}^{\ell}\phi_{i}(|\xi_i|^2)\cF_{d}[f]\right](\vec{x}).\nonumber
\end{align*}
Here, $\vec{b}=(b_{1},\cdots,b_{\ell})$ is the drift of $\vec{\phi}$, and $\vec{J}(\mathrm{d}\vec{y})$ is a vector of L\'evy measures defined by
\begin{align}
    \label{23.03.04.19.19}
    \vec{J}(\mathrm{d}\vec{y})=(J_1(\mathrm{d}\vec{y}),\cdots,J_{\ell}(\mathrm{d}\vec{y})),\quad J_i(\mathrm{d}\vec{y}):=J_i(y_i)\mathrm{d}y_{i}\epsilon_0^i(\mathrm{d}y_1,\cdots,\mathrm{d}y_{i-1},\mathrm{d}y_{i+1},\cdots,\mathrm{d}y_{\ell}),
\end{align}
where $J_i(y_i)$ is the jumping kernel of $\phi_{i}(\Delta_{x_{i}})$ and $
\epsilon_0^i$ is the centered Dirac measure in $\mathbb{R}^{d-d_i}$.

We now introduce the assumptions imposed on $\vec{\phi}$.
A function $f : (0,\infty) \to (0,\infty)$ is said to satisfy \emph{weak lower scaling condition} denoted by \textbf{WLS($c_{0},\delta_{0}$)}, if it holds that
\begin{equation}\label{eqn 07.15.14:41}
c_{0} \left(\frac{R}{r}\right)^{\delta_{0}}\leq \frac{f(R)}{f(r)}, \qquad 0<r<R<\infty.
\end{equation}

\begin{assumption}[\textbf{Weak Lower Scaling Codition}]\label{23.03.03.16.04}
There exist constants $\delta_0\in (0,1]$ and $c_{0}>0$ such that the Bernstein functions $\phi_1,\cdots,\phi_{\ell}$ satisfy \textbf{WLS($c_{0},\delta_{0}$)}, \textit{i.e.},
\begin{equation*}
\begin{gathered}
c_{0} \left(\frac{R}{r}\right)^{\delta_{0}}\leq\min\left(\frac{\phi_{1}(R)}{\phi_{1}(r)},\cdots,\frac{\phi_{\ell}(R)}{\phi_{\ell}(r)}\right), \qquad 0<r<R<\infty.
\end{gathered}
\end{equation*}
\end{assumption}

\begin{remark}\label{rmk 11.05.17:16}
\begin{enumerate}[(i)]
\item If $\phi_{i}(r)=r^{\alpha_{i}}$ with $\alpha_{i}\in(0,1]$ ($i=1,\dots,\ell$), then Assumption \ref{23.03.03.16.04} holds with $c_{0}=1$, and $\delta_{0}=\min\{\alpha_{1},\dots,\alpha_{\ell}\}$.
Consequently, this assumption covers vectors consisting of stable processes and Brownian motions.
Furthermore, combining Assumption \ref{23.03.03.16.04} with the concavity of $\phi$ yields the following two-sided bound:
 \begin{equation}
\label{phiratio}
c_{0}\left(\frac{R}{r}\right)^{\delta_0}\leq \frac{\phi_{i}(R)}{\phi_{i}(r)} \leq \frac{R}{r} , \qquad 0<r<R<\infty.
\end{equation}
\item If $\phi_{i}$ satisfies \textbf{WLS($c_{0},\delta_{0}$)}, then the following inequalities hold: 
\begin{align} \label{int phi}
\int_{\lambda^{-1}}^\infty r^{-1}(\phi_{i}(r^{-2}))^{\nu}\mathrm{d}r   \nonumber
& \leq c^{-\nu}_{0} \lambda^{-2\delta_{0}\nu} \left(\phi_{i}(\lambda^{2})\right)^{\nu} \int_{\lambda^{-1}}^{\infty}  r^{-1-2\delta_{0}\nu} \mathrm{d}r  \nonumber
\\
& \leq \frac{c^{-\nu}_{0}}{2\delta_{0}\nu} (\phi_{i}(\lambda^2))^{\nu} \quad \forall\, i=1,\dots,\ell \quad \forall\, \lambda,\nu>0.
\end{align}
\item Let $f: (0,\infty) \to (0,\infty)$ be an increasing function with an inverse function $f^{-1}$, and suppose that $f$ satisfies \textbf{WLS($c_{0},\delta_{0}$)}. 
Applying \eqref{eqn 07.15.14:41} with $f^{-1}(R)$ and $f^{-1}(r)$ in place of $R$ and $r$ ($0<r<R$), we obtain
$$
c_{0} \left(  \frac{f^{-1}(R)}{f^{-1}(r)} \right)^{\delta_{0}} \leq \frac{R}{r},
$$
which implies
$$
\frac{f^{-1}(R)}{f^{-1}(r)} \leq c^{-1/\delta_{0}}_{0} \left( \frac{R}{r} \right)^{1/\delta_{0}}.
$$
Since each \( \phi_{i} \) is a nontrivial Bernstein function, we have \( \phi'_{i}(\lambda) > 0 \) for all \( \lambda > 0 \).
Consequently, by \eqref{phiratio}, the inverse function \( \phi^{-1}_{i} \) satisfies the following inequality:
\begin{align}\label{eqn 07.15.14:49}
\left( \frac{R}{r} \right) \leq \frac{\phi^{-1}_{i}(R)}{\phi^{-1}_{i}(r)} \leq c^{-1/\delta_{0}}_{0} \left( \frac{R}{r} \right)^{1/\delta_{0}} \quad \forall\, 0<r<R<\infty.
\end{align}
\end{enumerate}
\end{remark}

\subsection{Solution spaces}
Next, we introduce Sobolev spaces associated with the operator $\vec{\phi}\cdot\Delta_{\vec{d}}$ will serve as our solution spaces.
\begin{definition}\label{defn defining}
Let $1<p,q<\infty$, $\gamma\in \bR$, and $0<T<\infty$. 
For a Schwartz function $u$, we define $(1-\vec{\phi}\cdot\Delta_{\vec{d}})^{\gamma/2}u$ as
$$
\cF[(1-\vec{\phi}\cdot\Delta_{\vec{d}})^{\gamma/2}u](\vec{\xi}):=\left( 1-\cF[\vec{\phi}\cdot\Delta_{\vec{d}}](\vec{\xi}) \right)^{\gamma/2}\cF[u](\vec{\xi}):= \left( 1+\sum_{i=1}^{\ell}\phi_{i}(|\xi_{i}|^{2}) \right)^{\gamma/2}\cF[u] (\vec{\xi}).
$$
For the well-definedness of $(1-\vec{\phi}\cdot\Delta_{\vec{d}})^{\gamma/2}u$ in $\cS'(\bR^d)$, we refer the reader to \cite{farkaspsi}.

(i) The space $H^{\vec{\phi},\gamma}_{p}=H^{\vec{\phi},\gamma}_{p}(\bR^{d})$ is a closure of $\cS(\bR^d)$ under the norm
$$
\|u\|_{H^{\vec{\phi},\gamma}_{p}}:= \|(1-\vec{\phi}\cdot\Delta_{\vec{d}})^{\gamma/2}u\|_{L_{p}}<\infty.
$$

(ii) We denote $C^{\infty}_{p}([0,T]\times \bR^{d})$ as a collection of functions $u(t,x)$ such that $D^{m}_{x}u \in C([0,T];L_{p})$ for all $m\in \bN_{0}$. 

(iii) The space $H^{\vec{\phi},\gamma}_{q,p}(T)$ is a closure of $C^{\infty}_{p}([0,T]\times\bR^{d})$ under the norm
$$
\|u\|_{H^{\vec{\phi},\gamma}_{q,p}(T)}:= \left(\int_{0}^{T} \|u(t,\cdot)\|^{q}_{H^{\vec{\phi},\gamma}_{p}} \mathrm{d}t\right)^{1/q} <\infty.
$$
We also denote $L_{q,p}(T):=H^{\vec{\phi},0}_{q,p}(T)$.
\end{definition}

First, we list some properties of the space $H_p^{\vec{\phi},\gamma}$ whose proof is contained in \cite[Lemma 2.6]{CKP23}.
\begin{proposition}\label{H_p^phi,gamma space}
Let $1<p<\infty$ and $\gamma\in\bR$.

(i) The space $H_p^{\vec{\phi},\gamma}$ is a Banach space.

(ii) For any $\mu\in\bR$, the map $(1-\vec{\phi}\cdot\Delta_{\vec{d}})^{\mu/2}$ is an isometry from $H^{\vec{\phi},\gamma}_{p}$ to $H^{\vec{\phi},\gamma-\mu}_{p}$.

(iii) If $\mu>0$, then we have continuous embeddings $H_p^{\vec{\phi},\gamma+\mu}\subset H_p^{\vec{\phi},\gamma}$ in the sense that
\begin{equation*}
\|u\|_{H_p^{\vec{\phi},\gamma}}\leq C \|u\|_{H_p^{\vec{\phi},\gamma+\mu}},
\end{equation*}
where the constant $C$ is independent of $u$. 

(iv) For any $u\in H^{\vec{\phi},\gamma+2}_{p}$, we have
\begin{equation*}\label{eqn 03.25.15:03}
\|u\|_{H_p^{\vec{\phi},\gamma+2}} \simeq \left(\|u\|_{H^{\vec{\phi},\gamma}_p}+\|(\vec{\phi}\cdot\Delta_{\vec{d}})u\|_{H^{\vec{\phi},\gamma}_p}\right) \simeq \left(\|u\|_{H^{\vec{\phi},\gamma}_p}+\sum_{i=1}^{\ell}\|\phi_{i}(\Delta_{x_{i}})u\|_{H^{\vec{\phi},\gamma}_p}\right).
\end{equation*}
In particular, if $\phi_1(\lambda)=\cdots=\phi_{\ell}(\lambda)=\lambda^{\beta}$, then $H_{p}^{\vec{\phi},2}$ becomes the classical Bessel potential space $H_{p}^{2\beta}$. 
\end{proposition}

Now we introduce a Besov space which plays an essential role for the class of initial data. 
We choose a function $\Psi$ from the Schwartz class $\mathcal{S}(\mathbb{R})$, whose one-dimensional Fourier transform $\mathcal{F}_{1}[\Psi]$ is nonnegative, supported within the set $[-2,-1/2]\cup[1/2,2]$. We also assume that
$$
\sum_{j\in\mathbb{Z}}\mathcal{F}_1[\Psi](2^{-j}\lambda)=1 \quad \forall\, \lambda\in\mathbb{R}\setminus\{0\}.
$$
Let $m_{\vec{\phi}}(\xi):=\phi_1(|\xi_1|^2)+\cdots+\phi_{\ell}(|\xi_{\ell}|^2)$ and let
$$
\Psi_j^{\vec{\phi}}(x):=\mathcal{F}_d^{-1}[\mathcal{F}_1[\Psi](2^{-j}m_{\vec{\phi}})](x).
$$
We define the Littlewood-Paley projection operators $\Delta_j^{\vec{\phi}}$ ($j \in \mathbb{Z}$) and $S_0^{\vec{\phi}}$ as 
$$
\Delta_j^{\vec{\phi}}f(x):=\int_{\mathbb{R}^{d}}\Psi_j^{\vec{\phi}}(y)f(x-y)\mathrm{d}y,\quad S_0^{\vec{\phi}}f(x):=\int_{\mathbb{R}^{d}}\Phi^{\vec{\phi}}(y)f(x-y)\mathrm{d}y,
$$
where $\Phi^{\vec{\phi}}(x):=\sum_{j\leq 0}\Psi_j^{\vec{\phi}}(x)$, respectively.

\begin{definition}
    Let $\gamma\in\mathbb{R}$, and $p,q\in[1,\infty)$.
    The space $B_{p,q}^{\vec{\phi},\gamma}=B_{p,q}^{\vec{\phi},\gamma}(\mathbb{R}^d)$ is defined as closure of $\mathcal{S}(\mathbb{R}^d)$ under the norm
    $$
    \|f\|_{B_{p,q}^{\vec{\phi},\gamma}}:=\|S^{\vec{\phi}}_{0}f\|_{L_p}+\left(\sum_{j=1}^{\infty}2^{\gamma q}\|\Delta^{\vec{\phi}}_{j}f\|_{L_p}^{q}\right)^{1/q}.
    $$
\end{definition}

\begin{definition}\label{def 01.06.16:16}
Let $\alpha \in (0,1)$, $1<p,q < \infty$, $\gamma \in \bR$, and $T<\infty$.
\begin{enumerate}[(i)]
    \item We say that $u\in \mathbb{H}^{\alpha,\vec{\phi},\gamma}_{q,p,0}(T)$ if $u\in H^{\vec{\phi},\gamma}_{q,p}(T)$ and there exists $f\in H^{\vec{\phi},\gamma}_{q,p}(T)$ such that
\begin{align}\label{eqn 01.02.17:53}
& \int_{0}^{T}\int_{\mathbb{R}^{d}} \left(I^{1-\alpha}_{t}(1-\vec{\phi}\cdot\Delta_{\vec{d}})^{\gamma/2}u(t,x) \right) \partial_{t}\left((1-\vec{\phi}\cdot\Delta_{\vec{d}})^{-\gamma/2} \eta(t,x) \right) \mathrm{d}x\mathrm{d}t     \nonumber
\\
=& - \int_{0}^{T}\int_{\mathbb{R}^{d}} \left((1-\vec{\phi}\cdot\Delta_{\vec{d}})^{\gamma/2}f(t,x)\right) \left((1-\vec{\phi}\cdot\Delta_{\vec{d}})^{-\gamma/2}\eta(t,x)\right) \mathrm{d}x\mathrm{d}t   
\end{align}
holds for every $\eta \in C^{\infty}_{c}([0,T)\times \mathbb{R}^{d})$.

    \item For such $u$ and $f$ satisfying \eqref{eqn 01.02.17:53}, we say that  $\partial_{t}I^{1-\alpha}_{t}u=f =\partial^{\alpha}_{t}u$.   We define the norm $\|\cdot \|_{\mathbb{H}^{\alpha,\vec{\phi},\gamma}_{q,p,0}(T)}$ as
$$
\|u \|_{\mathbb{H}^{\alpha,\vec{\phi},\gamma}_{q,p,0}(T)}:= \|\partial^{\alpha}_{t}u\|_{H^{\vec{\phi},\gamma}_{q,p}(T)} + \|u\|_{H^{\vec{\phi},\gamma}_{q,p}(T)}.
$$

    \item We denote
\begin{align*}
U^{\alpha,\vec{\phi},\gamma}_{p,q} =: 
\begin{cases}
H^{\vec{\phi},\gamma}_{p} & \text{if} \quad \alpha q \leq 1
\\
\vspace{-4mm}
\\
B^{\vec{\phi},\gamma+2-2/(\alpha q)}_{p,q} & \text{if} \quad \alpha q > 1,
\end{cases}
\end{align*}
and we say that $u\in \mathbb{H}^{\alpha,\vec{\phi},\gamma}_{q,p}(T)$ if $u\in H^{\vec{\phi},\gamma}_{q,p}(T)$ and if there exists $u_{0}\in U^{\alpha,\vec{\phi},\gamma}_{p,q}$ such that $u-u_0\in \mathbb{H}_{q,p,0}^{\alpha,\vec{\phi},\gamma}(T)$. We denote $\partial_{t}I^{1-\alpha}_{t} (u-u_{0}) = \partial^{\alpha}_{t}u$.
\end{enumerate}
\end{definition}
For the well-definedness of \eqref{eqn 01.02.17:53}, we refer the reader to \cite{farkaspsi}. Note that \eqref{eqn 01.02.17:53} holds for every
$$
\eta \in \bigcup_{\gamma\in \mathbb{R}} (1-\vec{\phi}\cdot \Delta_{\vec{d}})^{\gamma/2}C^{\infty}_{c}([0,T) \times \mathbb{R}^{d}).
$$ 

The function spaces introduced in Definition \ref{def 01.06.16:16} serve as the solution space of our target equation.
The following proposition establishes the key properties of these spaces.
\begin{proposition}\label{lem 01.13.15:58}
Let $\alpha \in  (0,1)$, $1<p,q<\infty$, $\gamma\in \bR$, and $0<T<\infty$. 

\begin{enumerate}[(i)]
    \item Suppose that $\alpha \in (0,1/q)$ and $u_{0} \in H^{\vec{\phi},\gamma}_{p}$.
    Then $u_0\in\mathbb{H}_{q,p}^{\alpha,\vec{\phi},\gamma}(T)$.
    Moreover,
\begin{align}\label{eqn 01.07.14:11}
\|\partial_{t}I^{1-\alpha}_{t}u_{0}\|_{H^{\vec{\phi},\gamma}_{q,p}(T)} \leq C(\alpha,q)T^{1/q-\alpha} \|u_{0}\|_{H^{\vec{\phi},\gamma}_{p}}.
\end{align}
    \item Suppose that $\alpha \in (0,1/q)$, then $\mathbb{H}^{\alpha,\vec{\phi},\gamma}_{q,p,0}(T) = \mathbb{H}^{\alpha,\vec{\phi},\gamma}_{q,p}(T)$.
    \item Suppose that $\alpha \in [1/q,1)$ and $u_{0} \in H^{\vec{\phi},\gamma}_{p}$. If $\partial_{t}I^{1-\alpha}_{t}u_{0}$ exists in $H^{\vec{\phi},\gamma}_{q,p}(T)$, then $u_{0} \equiv 0$.
    \item Suppose that $\alpha \in [1/q,1)$. Then for any $u \in \mathbb{H}^{\alpha,\vec{\phi},\gamma}_{q,p}(T)$, there exists unique $u_{0}\in 
U^{\alpha,\vec{\phi},\gamma}_{p,q}$ such that $u-u_{0} \in \mathbb{H}^{\alpha,\vec{\phi},\gamma}_{q,p,0}(T)$.
\end{enumerate}
\end{proposition}

Using Proposition \ref{lem 01.13.15:58}, we define the norm in $\mathbb{H}^{\alpha,\vec{\phi},\gamma}_{q,p}(T)$ as follows.

\begin{definition}
Let $\alpha \in  (0,1)$, $1<p,q<\infty$, $\gamma\in \bR$, and $0<T<\infty$.
\begin{enumerate}[(i)]
 \item We define the norm in $\mathbb{H}^{\alpha,\vec{\phi},\gamma}_{q,p}(T)$ as 
\begin{align*}
\|u\|_{\mathbb{H}^{\alpha,\vec{\phi},\gamma}_{q,p}(T)} =: 
\begin{cases} 
\|u\|_{\mathbb{H}^{\alpha,\vec{\phi},\gamma}_{q,p,0}(T)}  & \text{if} \quad \alpha q <1,
\\
\|\partial^{\alpha}_{t}u\|_{H^{\vec{\phi},\gamma}_{q,p}(T)} + \|u\|_{H^{\vec{\phi},\gamma}_{q,p}(T)} + \|u_{0}\|_{U^{\alpha,\vec{\phi},\gamma}_{q,p}} & \text{if} \quad \alpha q \geq 1.
\end{cases}
\end{align*}
Since $u_{0} \in U^{\alpha,\vec{\phi},\gamma}_{q,p}$ can be uniquely chosen by Proposition \ref{lem 01.13.15:58} (iv), the norm $\|\cdot\|_{\mathbb{H}^{\alpha,\vec{\phi},\gamma}_{q,p}(T)}$ is well-defined.

\item We say that $u\in \mathbb{H}^{1,\vec{\phi},\gamma}_{q,p}(T)$ if $u\in H^{\vec{\phi},\gamma}_{q,p}(T)$ and there exist $f\in H^{\vec{\phi},\gamma}_{q,p}(T)$ and $u_{0} \in B^{\vec{\phi},\gamma+2-2/q}_{p,q}$ such that $u(0,\cdot) = u_{0}$, $\partial_{t} u = f$ in usual (distribution) sense. The norm $\|\cdot\|_{\mathbb{H}^{1,\vec{\phi},\gamma}_{q,p}(T)}$ is defined as
$$
\|u\|_{\mathbb{H}^{1,\vec{\phi},\gamma}_{q,p}(T)} =: \|\partial_{t}u\|_{H^{\vec{\phi},\gamma}_{q,p}(T)} + \|u\|_{H^{\vec{\phi},\gamma}_{q,p}(T)} + \|u_{0}\|_{U^{\alpha,\vec{\phi},\gamma}_{q,p}}.
$$
If $u_{0}\equiv 0$, then we say that $u\in \mathbb{H}^{1,\vec{\phi},\gamma}_{q,p,0}(T)$. 
\end{enumerate}
\end{definition}

\begin{proposition}\label{prop 03.21.16:12}
Let $\alpha \in  (0,1]$, $1<p,q<\infty$, $\gamma\in \bR$, and $0<T<\infty$. 
\begin{enumerate}[(i)]

    \item The spaces $H_{q,p}^{\vec{\phi},\gamma}(T)$ and $\bH_{q,p}^{\alpha,\vec{\phi},\gamma}(T)$ are Banach spaces.
    \item The space $\bH_{q,p,0}^{\alpha,\vec{\phi},\gamma}(T)$ is a closed subspace of $\bH_{q,p}^{\alpha,\vec{\phi},\gamma}(T)$.
    \item For any $\nu\in\bR$, 
    $$
    (1-\vec{\phi}\cdot\Delta_{\vec{d}})^{\nu/2}:\bH_{q,p}^{\alpha,\vec{\phi},\gamma}(T)\cap H_{q,p}^{\vec{\phi},\gamma+2}(T)\to\bH_{q,p}^{\alpha,\vec{\phi},\gamma-\nu}(T)\cap H_{q,p}^{\vec{\phi},\gamma-\nu+2}(T)
    $$
    is an isometry, where the norm is naturally given as
	$$
	\|u\|_{\mathbb{H}^{\alpha,\vec{\phi},s}_{q,p}(T) \cap H^{\vec{\phi},s+2}_{q,p}(T)} = \|u\|_{\mathbb{H}^{\alpha,\vec{\phi},s}_{q,p}(T)} + \|u\|_{H^{\vec{\phi},s+2}_{q,p}(T)}.
	$$
Furthermore, for any  $u\in\bH_{q,p}^{\alpha,\vec{\phi},\gamma}(T) \cap H_{q,p}^{\vec{\phi},\gamma+2}(T)$
\begin{equation*}
(1-\vec{\phi}\cdot\Delta_{\vec{d}})^{\nu/2}\partial^{\alpha}_t  u=\partial^{\alpha}_t  (1-\vec{\phi}\cdot\Delta_{\vec{d}})^{\nu/2} u,
\end{equation*}
where $u_{0}\in U^{\alpha,\vec{\phi},\gamma}_{q,p}$ is an element which makes $u$ satisfies the definition $u\in \mathbb{H}^{\alpha,\vec{\phi},\gamma}_{q,p}(T)$.
    \item $C_c^\infty(\bR^{d+1}_+)$ is dense in $\bH_{q,p,0}^{\alpha,\vec{\phi},\gamma}(T) \cap H_{q,p}^{\vec{\phi},\gamma+2}(T)$.
    \item $C_p^\infty([0,T]\times\mathbb{R}^{d})$ is dense in $\bH_{q,p}^{\alpha,\vec{\phi},\gamma}(T)\cap H_{q,p}^{\vec{\phi},\gamma+2}(T)$.
\end{enumerate}
\end{proposition}
The proofs of Proposition \ref{lem 01.13.15:58} and Proposition \ref{prop 03.21.16:12} are provided in Section \ref{23.03.08.12.24}.

\subsection{Statement of Main result}

Here is the main result of this article.

\begin{theorem} \label{main theorem}
Let $\alpha\in(0,1]$, $1<p,q<\infty$, $\gamma \in \bR$, and $0<T<\infty$. Suppose that $\vec{\phi}=(\phi_1,\cdots,\phi_{\ell})$ is a vector of Bernstein functions satisfying Assumption \ref{23.03.03.16.04} with drift $\vec{b}_{0}=(b_{01},\dots,b_{0\ell})$ and vector of L\'evy measures $\vec{J}(\mathrm{d}\vec{y})$ defined in \eqref{23.03.04.19.19}.
Then for any $u_{0}\in B^{\vec{\phi},\gamma+2-2/(\alpha q)}_{q,p}$ and $f\in H_{q,p}^{\vec{\phi},\gamma}(T)$, the equation
\begin{equation}\label{mainequation1}
\partial^{\alpha}_{t}u(t,\vec{x}) =  \vec{\phi}\cdot \Delta_{\vec{d}} \, u(t,\vec{x}) + f(t,\vec{x}), \quad (t,\vec{x}) \in (0,T)\times \bR^{d}, \quad u(0,\vec{x}) =  \mathbf{1}_{\alpha q>1}u_{0} 
\end{equation}
admits a unique solution $u$ in the class $\bH_{q,p}^{\alpha,\vec{\phi},\gamma}(T)\cap H_{q,p}^{\vec{\phi},\gamma+2}(T)$ ($u \in \mathbb{H}^{\alpha,\vec{\phi},\gamma}_{q,p,0}(T)\cap H^{\vec{\phi},\gamma+2}_{q,p}(T)$ if $\alpha q \leq 1$) and  we have

\begin{equation} \label{mainestimate1}
\|u\|_{\bH_{q,p}^{\alpha,\vec{\phi},\gamma}(T)}+\|u\|_{H_{q,p}^{\vec{\phi},\gamma+2}(T)}\leq C \left( \|f\|_{H_{q,p}^{\vec{\phi},\gamma}(T)} + \|\mathbf{1}_{\alpha q>1}u_{0}\|_{U^{\alpha,\vec{\phi},\gamma}_{q,p}} \right),
\end{equation}
where $C=C(\alpha,d,c_{0},\delta_0,p,q,\ell,\gamma,T)$.
Moreover,
\begin{equation}
   \label{mainestimate-111}
\|(\vec{\phi}\cdot\Delta_{\vec{d}}) u\|_{H_{q,p}^{\vec{\phi},\gamma}(T)}\leq C_0 \left( \|f\|_{H_{q,p}^{\vec{\phi},\gamma}(T)} + \| \mathbf{1}_{\alpha q>1}u_{0}\|_{U^{\alpha,\vec{\phi},\gamma}_{q,p}} \right),
\end{equation}
where $C_0=C_0(\alpha,d,\delta_0,c_0,p,q,\ell,\gamma)$. 
\end{theorem}

\begin{remark}\label{rmk 01.31.16:48}

(i) When $\alpha q >1$, the function space $U^{\alpha,\vec{\phi},\gamma}_{q,p}=B^{\vec{\phi},\gamma+2-2/(\alpha q)}_{p,q}$ is the optimal class for the initial data. 
This result is established using the real interpolation theory, which we discuss in detail in Section \ref{sec 01.17.17:00}.

(ii) When $\alpha q <1$, then by Proposition \ref{lem 01.13.15:58}-(ii) $\mathbb{H}^{\alpha,\vec{\phi},\gamma}_{q,p}(T)\cap H^{\vec{\phi},\gamma+2}_{q,p}(T) = \mathbb{H}^{\alpha,\vec{\phi},\gamma}_{q,p,0}(T)\cap H^{\vec{\phi},\gamma+2}_{q,p}(T)$, which precisely means that the initial data $u_{0}$ can be absorbed to the free term $f$. Therefore, if $\alpha q<1$, the condition $u_{0} = 0$ is natural.
For the case $\alpha q =1$, the situation is more delicate to treat non-trivial initial conditions.
Hence, for the case $\alpha q \leq 1$, we just set $u_{0} \equiv 0$. We refer to \cite[Remark 3.16 (ii)]{KW2025} for a concise explanation.

(iii) The definition of the norm $\|\cdot\|_{\mathbb{H}^{\alpha,\vec{\phi},\gamma}_{q,p}(T)}$ implies that the estimate \eqref{mainestimate1} is equivalent to
\begin{align*}
\| \partial^{\alpha}_{t}  u \|_{H^{\vec{\phi},\gamma}_{q,p}(T)} + \|u\|_{H^{\vec{\phi},\gamma}_{q,p}(T)} + \|\vec{\phi}\cdot\Delta_{\vec{d}} \, u\|_{H^{\vec{\phi},\gamma}_{q,p}(T)}  \leq  C \left( \|f\|_{H_{q,p}^{\vec{\phi},\gamma}(T)} + \|\mathbf{1}_{\alpha q>1}u_{0}\|_{U^{\alpha,\vec{\phi},\gamma}_{q,p}} \right)
\end{align*}
which provides a more immediately interpretable formulation.

\end{remark}

\mysection{Heat Kernel Estimates for Space-Time Anisotropic Non-Local Operators}\label{sec 07.31.15:14}

Since IASBM $\vec{X}_{t} = (X^{1}_{t},\dots,X^{\ell}_{t})$ consists of independent processes, the heat kernel of $\vec{X}$, denoted by $p(t,\vec{x})$, can be expressed as a product of the heat kernels $p_{i}$ of the component process $X^{i}$:
\begin{equation*}
    p(t,\vec{x})=\prod_{i=1}^{\ell}p_i(t,x_i),\quad \forall (t,\vec{x})\in(0,\infty)\times\bR^{\vec{d}}.
\end{equation*}
Let $Q_{t}$ be an increasing L\'evy process that is independent of $\vec{X}_{t}$, and has the Laplace exponent
$$
\mathbb{E} \exp{(-\lambda Q_{t})} = \exp{(-\lambda^{\alpha}t)}.
$$
Define $R_{t}$ as the right-continuous inverse process of $Q_{t}$ given by 
$$
R_{t} := \inf\{s>0: Q_{s}>t \},
$$
and let $\varphi(t,\cdot)$ denote the probability density function of $R_{t}$. 
It follows that $q(t,\vec{x})$, the transition density of the time-changed IASBM $\vec{X}_{R_{t}}$ admits the following representations (see, \textit{e.g.}, \cite[Section 3]{KPR21nonlocal}):
\begin{equation}
\label{eqn 12.31.17:50}
q(t,\vec{x}) := \int_{0}^{\infty} p(r,\vec{x}) \mathbb{P}(R_{t}\in \mathrm{d}r) = \int_{0}^{\infty} p(r,\vec{x}) \varphi(t,r) \mathrm{d}r.
\end{equation}
The primary objectives of this section are as follows:
\begin{enumerate}[1.]
    \item To show that \(q(t,\vec{x})\) serves as the fundamental solution to the equation (Theorem \ref{lem 06.24.15:35}):
    \begin{equation*}
    	\partial^{\alpha}_{t} q(t,\vec{x}) = \vec{\phi}\cdot\Delta_{\vec{d}} \, q(t,\vec{x}). \end{equation*}
    \item To establish estimates for $q_{\alpha,\beta}$, defined as (Theorem \ref{lem 06.01.17:52}): 
    \begin{equation}
\label{eqn 06.24.15:02}
    q_{\alpha,\beta}(t,\vec{x}) := \int_{0}^{\infty} p(r,\vec{x}) \varphi_{\alpha,\beta}(t,r)\mathrm{d}r,
\end{equation}
where $\varphi_{\alpha,\beta}(t,r):=D_t^{\beta-\alpha}\varphi(t,r)$ with $\alpha\in(0,1)$, and $\beta\in\mathbb{R}$.
\end{enumerate}
Now, we present the main results of this section.
\begin{theorem}\label{lem 06.24.15:35}
Let $f \in C^{\infty}_{c}(\bR^{d+1}_{+})$.
Then, the function
\begin{align}\label{eqn 07.22.16:51}
\mathcal{G}_{0}f(t,\vec{x}) = \int_{0}^{t} \int_{\bR^{d}} q_{\alpha,1}(t-s, \vec{x}-\vec{y}) f(s,\vec{y}) \mathrm{d}\vec{y} \mathrm{d}s
\end{align}
is a (strong) solution to the equation
\begin{equation}
\label{25.03.14.10.36}
    \begin{cases}
    \partial^{\alpha}_{t}u(t,\vec{x}) = \vec{\phi} \cdot \Delta_{\vec{d}} \, u(t,\vec{x}) +f(t,\vec{x}),\quad &(t,\vec{x})\in(0,\infty)\times\mathbb{R}^d,\\
    u(0,\vec{x})=0,\quad &\vec{x}\in\mathbb{R}^d.
    \end{cases}
\end{equation}
Moreover, for $u\in C^{\infty}_{c}(\bR^{d+1}_{+})$, if we define $f:=\partial^{\alpha}_{t} u - \vec{\phi}\cdot\Delta_{\vec{d}} \, u$, then $u$ admits the representation \eqref{eqn 07.22.16:51}.
\end{theorem}

\begin{theorem}\label{lem 06.01.17:52}
Let $\alpha\in (0,1)$, $\beta\in \bR$, and $m_{i}$ ($i=1,\dots,\ell$)  be $d_{i}$-dimensional multi-indices.
Additionally, let $\ell_{1},\ell_{2} \in \bN_{0}$ such that $\ell_{1}+\ell_{2} = \ell$, and let $\{j_{1},\dots,j_{\ell_{1}}, i_{1},\dots, i_{\ell_{2}} \}$ be a permutation of $\{1,\dots,\ell \}$.
Suppose that $(t,x)\in (0,\infty)\times (\bR^{d} \setminus \{0\})$ satisfies 
\begin{gather}
1 \leq t^{\alpha}\phi_{i_{\ell_{2}}}(|x_{i_{\ell_{2}}}|^{-2}) \leq \dots \leq t^{\alpha}\phi_{i_{1}}(|x_{i_{1}}|^{-2}),  \label{eqn 06.05.14:34}
\\
t^{\alpha} \phi_{j}(|x_{j}|^{-2}) \leq 1 \quad \forall\, j= j_{1},\dots,j_{\ell_{1}}.   \nonumber
\end{gather}
Then we have
\begin{align}\label{eqn 06.01.18:14}
\left|  D^{m_{1}}_{x_{1}}\dots D^{m_{\ell}}_{x_{\ell}}q_{\alpha,\beta}(t,\vec{x}) \right| \leq C t^{\frac{ \ell_{1} \alpha}{2}-\beta} \left( \prod_{n=1}^{\ell_{1}} \frac{(\phi_{j_{n}}(|x_{j_{n}}|^{-2}))^{\frac{1}{2}}}{|x_{j_{n}}|^{d_{j_{n}}+m_{j_{n}}}}  \right) \Lambda^{\ell_{2},\vec{m}}_{\alpha,\beta}(t,x_{i_{1}},\dots,x_{i_{\ell_{2}}}),
\end{align}
where
\begin{align*}
&\Lambda^{\ell_{2},\vec{m}}_{\alpha,\beta}(t,x_{i_{1}},\dots,x_{i_{\ell_{2}}})
\\
=&\prod_{k=1}^{\ell_{2}} \left( \int_{(\phi_{i_{k}}(|x_{i_{k}}|^{-2}))^{-\frac{1}{\ell_{2}}}}^{2^{\frac{1}{\ell_{2}}}t^{\frac{\alpha}{\ell_{2}}}} \left( \phi_{i_{k}}^{-1}(r^{-\ell_{2}}) \right)^{\frac{d_{i_{k}}+m_{i_{k}}}{2}} \mathrm{d}r \right) 
+\int_{(\phi_{i_{\ell_{2}}}(|x_{i_{\ell_{2}}}|^{-2}))^{-1}}^{2t^{\alpha}}   \left( \prod_{k=1}^{\ell_{2}}  (\phi^{-1}_{i_{k}}(r^{-1})^{\frac{d_{i_{k}}+m_{i_{k}}}{2}} \right) \mathrm{d}r
\\
&+ \sum_{k=2}^{\ell_{2}} \Bigg[ \int_{(\phi_{i_{k-1}}(|x_{i_{k-1}}|^{-2}))^{-\frac{1}{\ell_{2}}}}^{2^{\frac{1}{\ell_{2}}}t^{\frac{\alpha}{\ell_{2}}}} \left( \prod_{n=1}^{k-1}(\phi^{-1}_{i_{n}}(r^{-\ell_{2}}))^{\frac{d_{i_{n}}+m_{i_{n}}}{2}} \right)   r^{n-2} \mathrm{d}r 
\\
& \qquad \qquad \qquad \qquad \times \prod_{n=k}^{\ell_{2}} \left( \int_{(\phi_{i_{n}}(|x_{i_{n}}|^{-2}))^{-\frac{1}{\ell_{2}}}}^{2^{\frac{1}{\ell_{2}}}t^{\frac{\alpha}{\ell_{2}}}} \left( \phi_{i_{n}}^{-1}(r^{-\ell_{2}}) \right)^{\frac{d_{i_{n}}+m_{i_{n}}}{2}} \mathrm{d}r \right) \Bigg], 
\\
=:& \Lambda^{\ell_{2},\vec{m},1}_{\alpha,\beta}(t,x_{i_{1}},\dots,x_{i_{\ell_{2}}})
+\Lambda^{\ell_{2},\vec{m},2}_{\alpha,\beta}(t,x_{i_{1}},\dots,x_{i_{\ell_{2}}})
 + \sum_{k=2}^{\ell_{2}} \lambda^{\ell_{2},\vec{m},k}_{\alpha,\beta}(t,x_{i_{1}},\dots,x_{i_{\ell_{2}}})
\\
=:& \Lambda^{\ell_{2},\vec{m},1}_{\alpha,\beta}(t,x_{i_{1}},\dots,x_{i_{\ell_{2}}}) + \Lambda^{\ell_{2},\vec{m},2}_{\alpha,\beta}(t,x_{i_{1}},\dots,x_{i_{\ell_{2}}}) + \Lambda^{\ell_{2},\vec{m},3}_{\alpha,\beta}(t,x_{i_{1}},\dots,x_{i_{\ell_{2}}}).
\end{align*}
The constant $C$ depends only on $\alpha,\beta,d,c_{0},\delta_{0},\ell_{1},\ell_{2},m_{1},\dots,m_{\ell}$.
\end{theorem}
We provide some remarks on Theorem \ref{lem 06.01.17:52} to offer further motivation.
\begin{remark}\label{rmk 12.23.17:06}
(i) The transition density $p$ of a stochastic process $Y$ describes the displacement of the process.
Specifically, it depends on both the starting point and the endpoint in the sense that
\begin{align*}
\mathbb{P}(Y_{t}\in A | Y_{0}=x) = \int_{A} p(t,x,y) \mathrm{d}y.
\end{align*}
If $Y$ is a L\'evy process, 
then it is translation invariant, which implies that $p(t,x,y)=p(t,0,y-x)$.
Thus, for a L\'evy process $Y$, its transition density can be expressed as $p(t, y-x):=p(t,0,y-x)$. 
Furthermore, if $Y$ is isotropic, then the transition density function depends only on the distance from the origin, and we can write $p(t,x)=p(t,|x|)$.

The scaling condition \textbf{WLS($c_{0},\delta_{0}$)} on characteristic exponents of L\'evy processes determines the behavior of corresponding jumping kernels, and thus provides estimations of transition densities.
For instance, if $\phi_i$ is a Bernstein function satisfying \textbf{WLS($c_{0},\delta_{0}$)}, then the corresponding transition density $p_i$ satisfies the estimate
\begin{align*}
|p_{i}(t,x_{i}) | \leq C \left( \underbrace{(\phi^{-1}_{i}(t^{-1}))^{\frac{d_{i}}{2}}}_{\text{estimate in near diagonal regime}} \mathbf{1}_{t\phi_{i}(|x_{i}|^{-2}) \geq 1}  + \underbrace{t^{1/2}\frac{(\phi_{i}(|x_{i}|^{-2}))^{1/2}}{|x_{i}|^{d_{i}}}}_{\text{estimate in off-diagonal regime}} \mathbf{1}_{t\phi_{i}(|x_{i}|^{-2}) \leq 1} \right),\quad i=1,2,\cdots,\ell
\end{align*}
(see \cite[Theorem 3.3]{CKP23}).
The set $\{(t,x_i):t\phi_{i}(|x_{i}|^{-2}) \leq 1\}$ is referred to as the \textit{off-diagonal regime}, as it occurs when $|x_i|$ (\textit{i.e.}, distance between $0$ and $x_i$) is sufficiently large relative to $t$.
Likewise, the set $\{(t,x_i):t\phi_{i}(|x_{i}|^{-2}) \geq 1\}$ is called the \textit{near-diagonal regime}.
Therefore, if Bernstein functions $\phi_1\cdots,\phi_{\ell}$ satisfy \textbf{WLS($c_{0},\delta_{0}$)}, then the corresponding IASBM $\vec{X}_t=(X^1_t,\cdots,X^{\ell}_t)$ satisfies the following heat kernel estimate:
\begin{align}\label{eqn 12.23.17:59}
|p(t,\vec{x})|=\prod_{i=1}^{\ell}|p(t,x_i)|\leq C\prod_{i=1}^{\ell}\left( (\phi^{-1}_{i}(t^{-1}))^{\frac{d_{i}}{2}} \mathbf{1}_{t\phi_{i}(|x_{i}|^{-2}) \geq 1}  + t^{1/2}\frac{(\phi_{i}(|x_{i}|^{-2}))^{1/2}}{|x_{i}|^{d_{i}}} \mathbf{1}_{t\phi_{i}(|x_{i}|^{-2}) \leq 1} \right).
\end{align}
Theorem \ref{lem 06.01.17:52} provides a time-fractional analogue of \eqref{eqn 12.23.17:59}.
However, since the independence of the component processes $X^i_{R_t}$ is not guaranteed, the above argument cannot be directly applied.
Thus, estimating the transition density $q(t,\vec{x})$ of time-changed IASBM $\vec{X}_{R_t}$ requires a more delicate analysis.

(ii) For $(x_{j_{1}},\dots,x_{j_{\ell_{1}}})$ which lies in the off-diagonal regime, the corresponding term 
\begin{align*}
t^{\frac{ \ell_{1} \alpha}{2}-\beta} \left( \prod_{n=1}^{\ell_{1}} \frac{(\phi_{j_{n}}(|x_{j_{n}}|^{-2}))^{\frac{1}{2}}}{|x_{j_{n}}|^{d_{j_{n}}+m_{j_{n}}}}  \right)
\end{align*}
in the right-hand side of \eqref{eqn 06.01.18:14} follows directly from the off-diagonal upper bound for each $p_{j_n}$.
On the other hand, in the near-diagonal estimates for $(x_{i_{1}},\dots,x_{i_{\ell_{2}}})$, it is difficult to express the result as a product of near-diagonal upper bounds for each $p_{i_n}$, since the component processes $X^i_{R_t}$ are no longer independent after the time change.
However, while $\Lambda^{\ell_{2},\vec{m}}_{\alpha,\beta}$ appears complex at first glance, its structure can be understood by examining the hierarchical ordering imposed by \eqref{eqn 06.05.14:34}.
This condition determines which components of $(x_{i_1},\cdots,x_{i_{\ell_2}})$ are closer to the origin, which is crucial in representing the transition density as in \eqref{eqn 12.31.17:50}:
\begin{align*}
&\int_{0}^{\infty} D^{\vec{m}}_{\vec{x}}p(r,\vec{x})\varphi_{\alpha,\beta}(t,r)\mathrm{d}r
\\
=& \int_{t^{\alpha}}^{\infty} \cdots \mathrm{d}r + \int_{0}^{(\phi_{i_{1}}(|x_{i_{1}}|^{-2}))^{-1}} \cdots \mathrm{d}r + \cdots + \int_{(\phi_{i_{\ell_{2}-1}}(|x_{i_{\ell_{2}-1}}|^{-2}))^{-1}}^{(\phi_{i_{\ell_{2}}}(|x_{i_{\ell_{2}}}|^{-2}))^{-1}} \cdots \mathrm{d}r + \int_{(\phi_{i_{\ell_{2}}}(|x_{i_{\ell_{2}}}|^{-2}))^{-1}}^{t^{\alpha}} \cdots \mathrm{d}r.
\end{align*}
This explains why the function \( \Lambda^{\ell_{2},\vec{m}}_{\alpha,\beta} \) in \eqref{eqn 06.01.18:14} has a complex form

(iii) In particular, when we set $\ell=1$ (and let $\phi_{1}=\phi_{\ell}=\phi$), the estimate \eqref{eqn 06.01.18:14} simplifies to
\begin{align*}
\left|D^{m}_{x}q_{\alpha,\beta}(t,x) \right| &\leq C \left( t^{-\beta} \int_{(\phi(|x|^{-2}))^{-1}}^{2t^{\alpha}}\left( \phi^{-1}(r^{-1}) \right)^{\frac{d+m}{2}} \mathrm{d}r \mathbf{1}_{\ell_{2}=\ell}+  t^{\frac{\alpha}{2}-\beta} \frac{(\phi(|x|^{-2}))^{\frac{1}{2}}}{|x|^{d+m}}\mathbf{1}_{\ell_{1}=\ell}  \right)
\\
& \leq C \left( t^{-\beta} \int_{(\phi(|x|^{-2}))^{-1}}^{2t^{\alpha}}\left( \phi^{-1}(r^{-1}) \right)^{\frac{d+m}{2}} \mathrm{d}r \mathbf{1}_{t^{\alpha}\phi(|x|^{-2}) \geq 1}+  t^{\frac{\alpha}{2}-\beta} \frac{(\phi(|x|^{-2}))^{\frac{1}{2}}}{|x|^{d+m}}\mathbf{1}_{t^{\alpha}\phi(|x|^{-2}) \leq 1}  \right)
\end{align*}
which resembles the estimate in the isotropic case (see, e.g., \cite[Lemma 3.8]{KPR21nonlocal}). 
\end{remark}
The following outlines the proof structure for Theorem \ref{lem 06.24.15:35} and Theorem \ref{lem 06.01.17:52}:
\begin{equation*}
    \begin{aligned}
    \begin{rcases}
    &\text{Theorem \ref{pestimate}: Estimates of each $p_i$ } \\
    &\to
        \text{Lemma \ref{lem 12.16.18:23}: Lemma for the near-diagonal estimates}
        \end{rcases}
    \to&
        \text{Theorem \ref{lem 06.01.17:52}}\to\begin{rcases}
        \text{Lemma \ref{lem 06.08.14:52}}\\
        \text{Corollary \ref{lem 02.17.22:32}}
        \end{rcases}
        \to \text{Theorem \ref{lem 06.24.15:35}},
    \end{aligned}
\end{equation*}
where $A\to B$ indicates that $A$ is used in the proof of $B$. 

From \eqref{eqn 12.31.17:50}, we observe that $q(t,\vec{x})$ consists of two main components, $p$ and $\varphi$.
The function $\varphi_{\alpha,\beta}$ satisfies the following estimates (see \cite[Lemma 3.7 (ii)]{KPR21nonlocal}):
\begin{equation}\label{eqn 06.01.15:26}
|\varphi_{\alpha,\beta}(t,r)| \leq C  t^{-\beta} \exp{\left(-c(rt^{-\alpha})^{1/(1-\alpha)}\right)} \quad \text{for} \quad rt^{-\alpha}>1,
\end{equation}
and
\begin{align}\label{eqn 06.01.15:30}
|\varphi_{\alpha,\beta}(t,r)| \leq \begin{cases} C r t^{-\alpha-\beta} & \beta \in \bN
\\
C t^{-\beta} & \beta \notin \bN \quad \text{for} \quad rt^{-\alpha}\leq 1,
\end{cases}
\end{align}
where the constants $C,c>0$ depend only on $\alpha,\beta$.
The estimates for $p_i$ are given in \cite[Theorem 3.3]{CKP23} and are summarized as follows.
\begin{theorem} \label{pestimate}
Let $i=1,\cdots\ell$ and Assumption \ref{23.03.03.16.04} hold. 

(i) For any $m,k\in \bN_{0}$, and $\nu\in(0,1)$, we have
\begin{align*}
|\phi_{i}(\Delta_{x_{i}})^{\nu k}D^{m}_{x_{i}}p_{i}(t,x_{i})| \leq C_{i}  \left(t^{-\nu  k}\left(\phi_{i}^{-1}(t^{-1})\right)^{\frac{d_{i}+m}{2}}\wedge t^{\nu-\nu k}\frac{(\phi_{i}(|x_{i}|^{-2}))^{\nu}}{|x_{i}|^{d_{i}+m}}\right),
\end{align*} 
where the constant $C_{i}$ depends only on $c_{0},\delta_{0},d_{i},m,k,\nu$. In particular, we have

\begin{equation} \label{pestimate ineq}
|\phi_{i}(\Delta_{x_{i}})^{k}D^{m}_{x_{i}}p_{i}(t,x_{i})| \leq C_{i}  \left(t^{-k}\left(\phi_{i}^{-1}(t^{-1})\right)^{\frac{d_{i}+m}{2}}\wedge t^{\frac{1}{2}-k}\frac{(\phi_{i}(|x_{i}|^{-2}))^{1/2}}{|x_{i}|^{d_{i}+m}}\right),
\end{equation}

(ii) For any $k=0,1,\dots$, we have
\begin{equation*}
\int_{\bR^{d_{i}}} |\phi_{i}(\Delta_{x_{i}})^{k}p_{i}(t,x_{i})| \mathrm{d}x_{i} \leq C_{i} t^{-k},
\end{equation*}
where the constant $C_{i}$ depends only on $c_{0},\delta_{0},d_{i},k$.
\end{theorem}

\begin{lemma}\label{lem 12.16.18:23}
Let $\alpha\in (0,1)$, $\beta\in \bR$, $1 \leq \ell_{2} \leq \ell$, and let $m_{i}$ ($i=1,\dots, \ell$) be $d_{i}$-dimensional multi-indices. Suppose $\{i_{1},\dots, i_{\ell_{2}} \}$ is a nonempty subset of $\{1,\dots,\ell\}$ and $(t,x)\in (0,\infty)\times (\bR^{d} \setminus \{0\})$ satisfies 
\begin{gather*}
1 \leq t^{\alpha}\phi_{i_{\ell_{2}}}(|x_{i_{\ell_{2}}}|^{-2}) \leq \dots \leq t^{\alpha}\phi_{i_{1}}(|x_{i_{1}}|^{-2}).
\end{gather*}
Define $\vec{m}=(m_{1},\dots,m_{\ell})$ and
$$
P_{2}(t,\vec{x}) = \prod_{k=1}^{\ell_{2}}p_{i_{k}}(t,x_{i_{k}}).
$$
Let us also denote
\begin{align*}
    \tilde{I}_{1}(t,\vec{x}) &=: \int_{0}^{(\phi_{i_{1}}(|x_{i_{1}}|^{-2}))^{-1}} \left| D^{m_{i_{1}}}_{x_{i_{1}}} \cdots D^{m_{i_{\ell_{2}}}}_{x_{i_{\ell_{2}}}} P_{2}(r,\vec{x}) \right| \mathrm{d}r,\\
    \tilde{I}_{k}(t,\vec{x})=: &\int_{(\phi_{i_{k-1}}(|x_{i_{k-1}}|^{-2}))^{-1}}^{(\phi_{i_{k}}(|x_{i_{k}}|^{-2}))^{-1}} \left| D^{m_{i_{1}}}_{x_{i_{1}}}\cdots D^{m_{i_{\ell_2}}}_{x_{i_{\ell_2}}} P_{2}(r,\vec{x}) \right|  \mathrm{d}r,\quad 2\leq k\leq \ell_2,\\
    \tilde{I}_{\ell_{2}+1}(t,\vec{x}) &=: \int_{(\phi_{i_{\ell_{2}}}(|x_{i_{\ell_{2}}}|^{-2}))^{-1}}^{t^{\alpha}} \left| D^{m_{i_{1}}}_{x_{i_{1}}} \cdots D^{m_{i_{\ell_{2}}}}_{x_{i_{\ell_{2}}}} P_{2}(r,\vec{x}) \right| \mathrm{d}r.
\end{align*}

$(i)$ Then we have
\begin{align*}
\tilde{I}_{1}(t,\vec{x}) \leq C \prod_{k=1}^{\ell_{2}} \left( \int_{(\phi_{i_{k}}(|x_{i_{k}}|^{-2}))^{-\frac{1}{\ell_{2}}}}^{2 ^{\frac{1}{\ell_{2}}}t^{\frac{\alpha}{\ell_{2}}}} \left( \phi_{i_{k}}^{-1}(r^{-\ell_{2}}) \right)^{\frac{d_{i_{k}}+m_{i_{k}}}{2}} \mathrm{d}r \right) =: C\, \Lambda^{\ell_{2},\vec{m},1}_{\alpha,\beta}(t,x_{i_{1}},\dots,x_{i_{\ell_{2}}}),
\end{align*}
and
\begin{align*}
\tilde{I}_{\ell_{2}+1}(t,\vec{x}) \leq C \int_{(\phi_{i_{\ell_{2}}}(|x_{i_{\ell_{2}}}|^{-2}))^{-1}}^{2t^{\alpha}} \left( \prod_{k=1}^{\ell_{2}} (\phi^{-1}_{i_{k}}(r^{-1}))^{\frac{d_{i_{k}}+m_{i_{k}}}{2}} \right) =: C \,\Lambda^{\ell_{2},\vec{m},2}_{\alpha,\beta}(t,x_{i_{1}},\dots,x_{i_{\ell_{2}}}),
\end{align*}
where the constant $C$ depends only on $\alpha,\beta,d,c_{0},\delta_{0},\ell_{2},m_{1_{1}},\dots,m_{i_{\ell_{2}}}$.

$(ii)$ Then we have
\begin{align*}
\tilde{I}_{k}(t,\vec{x})
\leq C \, \lambda^{\ell_{2},\vec{m},k}_{\alpha,\beta}(t,x_{i_{1}},\dots,x_{i_{\ell_{2}}}),
\end{align*}
where the constant $C$ depends only on $\alpha,\beta,d,c_{0},\delta_{0},\ell_{2},m_{1_{1}},\dots,m_{i_{\ell_{2}}}$ and
\begin{align}\label{eqn 12.19.13:55}
&\lambda^{\ell_{2},m,k}_{\alpha,\beta}(t,x_{i_{1}},\dots,x_{i_{\ell_{2}}})  \nonumber
\\
=& \int_{(\phi_{i_{k-1}}(|x_{i_{k-1}}|^{-2}))^{-\frac{1}{\ell_{2}}}}^{2^{\frac{1}{\ell_{2}}}t^{\frac{\alpha}{\ell_{2}}}} \left( \prod_{n=1}^{k-1}(\phi^{-1}_{i_{n}}(r^{-\ell_{2}}))^{\frac{d_{i_{n}}+m_{i_{n}}}{2}} \right)   r^{k-2} \mathrm{d}r    \nonumber
\\
& \times \prod_{n=k}^{\ell_{2}} \Bigg( \int_{(\phi_{i_{n}}(|x_{i_{n}}|^{-2}))^{-\frac{1}{\ell_{2}}}}^{2^{\frac{1}{\ell_{2}}}t^{\frac{\alpha}{\ell_{2}}}} \left( \phi_{i_{n}}^{-1}(r^{-\ell_{2}}) \right)^{\frac{d_{i_{n}}+m_{i_{n}}}{2}} \mathrm{d}r \Bigg).
\end{align} 
\end{lemma}
\begin{proof}
$(i)$ The estimate of $\tilde{I}_{\ell_{2}+1}$ directly follows from \eqref{pestimate ineq}.
For $\tilde{I}_{1}$, using \eqref{pestimate ineq}, and \eqref{eqn 06.05.14:34}, we have
\begin{align}\label{eqn 06.02.16:41}
\tilde{I}_{1}(t,\vec{x}) &\leq C \int_{0}^{(\phi_{i_{1}}(|x_{i_{1}}|^{-2}))^{-1}} r^{\frac{\ell_{2}}{2}} \mathrm{d}r \left(  \prod_{k=1}^{\ell_{2}}  \frac{(\phi_{i_{k}}(|x_{i_{k}}|^{-2}))^{1/2}}{|x_{i_{k}}|^{d_{i_{k}}+m_{i_{k}}}}  \right)  \nonumber
\\
& = C  (\phi_{i_{1}}(|x_{i_{1}}|^{-2}))^{-\frac{\ell_{2}+2}{2}}  \left(\prod_{k=1}^{\ell_{2}}  \frac{(\phi_{i_{k}}(|x_{i_{k}}|^{-2}))^{1/2}}{|x_{i_{k}}|^{d_{i_{k}}+m_{i_{k}}}} \right)  \nonumber
\\
& = C \left( (\phi_{i_{1}}(|x_{i_{1}}|^{-2}))^{-\frac{1}{2}-\frac{1}{\ell_{2}}} \right)^{\ell_{2}} \left(\prod_{k=1}^{\ell_{2}}  \frac{(\phi_{i_{k}}(|x_{i_{k}}|^{-2}))^{1/2}}{|x_{i_{k}}|^{d_{i_{k}}+m_{i_{k}}}} \right)   \nonumber
\\
&\leq C   \prod_{k=1}^{\ell_{2}}  \frac{(\phi_{i_{k}}(|x_{i_{k}}|^{-2}))^{-\frac{1}{\ell_{2}}}}{|x_{i_{k}}|^{d_{i_{k}}+m_{i_{k}}}} .
\end{align}
Note that for any $1\leq k\leq \ell_2$,
\begin{align}\label{eqn 06.05.15:11}
&\left( \phi_{i}(|x_{i_{k}}|^{-2}) \right)^{-\frac{1}{\ell_{2}}} |x_{i_{k}}|^{-d_{i_{k}}-m_{i_{k}}}  \nonumber
\\
=& \left( 2^{\frac{1}{\ell_{2}}} -  1 \right) \int_{( \phi_{i_{k}}(|x_{i_{k}}|^{-2}) )^{-\frac{1}{\ell_{2}}}}^{2^{\frac{1}{\ell_{2}}}( \phi_{i_{k}}(|x_{i_{k}}|^{-2}) )^{-\frac{1}{\ell_{2}}}} |x_{i}|^{-d_{i_{k}}-m_{i_{k}}} \mathrm{d}r \nonumber
\\
\leq& C \int_{( \phi_{i_{k}}(|x_{i_{k}}|^{-2}) )^{-\frac{1}{\ell_{2}}}}^{2^{\frac{1}{\ell_{2}}}t^{\frac{\alpha}{\ell_{2}}}} (\phi^{-1}_{i_{k}}(r^{-\ell_{2}}))^{\frac{d_{i_{k}}+m_{i_{k}}}{2}} \mathrm{d}r \qquad (t^{\alpha}\phi_{i_{k}}(|x_{i_{k}}|^{-2}) \geq 1),
\end{align}
where for the last inequality, we used the fact that (recall \eqref{phiratio})
\begin{equation*}
|x_{i_{k}}|^{-2} \leq  C \phi^{-1}_{i_{k}}(r^{-\ell_{2}}) \quad \text{for} \quad r\leq 2(\phi_{i_{k}}(|x_{i_{k}}|^{-2}))^{-1/\ell_{2}}.
\end{equation*}
Applying \eqref{eqn 06.05.15:11} to \eqref{eqn 06.02.16:41}, we have
\begin{align*}
\tilde{I}_{1} (t,\vec{x}) & \leq C     \prod_{k=1}^{\ell_{2}} \left( \int_{(\phi_{i_{k}}(|x_{i_{k}}|^{-2}))^{-\frac{1}{\ell_{2}}}}^{2 ^{\frac{1}{\ell_{2}}}t^{\frac{\alpha}{\ell_{2}}}} \left( \phi_{i_{k}}^{-1}(r^{-\ell_{2}}) \right)^{\frac{d_{i_{k}}+m_{i_{k}}}{2}} \mathrm{d}r \right) \nonumber
\\
& := \Lambda^{\ell_{2},\vec{m},1}_{\alpha,\beta}(t,x_{i_{1}},\dots,x_{i_{\ell_{2}}}).
\end{align*}

$(ii)$ By \eqref{pestimate ineq}, the change of variable $r\to r^{\ell_{2}}$, and \eqref{eqn 06.05.15:11}, we have 
\begin{align}\label{eqn 12.16.18:19}
\tilde{I}_{k}(t,\vec{x}) &\leq C \int_{(\phi_{i_{k-1}}(|x_{i_{k-1}}|^{-2}))^{-1}}^{(\phi_{i_{k}}(|x_{i_{k}}|^{-2}))^{-1}} \left( \prod_{n=1}^{k-1}(\phi^{-1}_{i_{n}}(r^{-1}))^{\frac{d_{i_{n}}+m_{i_{n}}}{2}} \right) \left(\prod_{n=k}^{\ell_{2}} r^{\frac{1}{2}} \frac{(\phi_{i_{n}}(|x_{i_{n}}|^{-2}))^{\frac{1}{2}}}{|x_{i_{n}}|^{d_{i_{n}}+m_{i_{n}}}} \right) \mathrm{d}r       \nonumber
\\
&= C \int_{(\phi_{i_{k-1}}(|x_{i_{k-1}}|^{-2}))^{-\frac{1}{\ell_{2}}}}^{(\phi_{i_{k}}(|x_{i_{k}}|^{-2}))^{-\frac{1}{\ell_{2}}}} \left( \prod_{n=1}^{k-1}(\phi^{-1}_{i_{n}}(r^{-\ell_{2}}))^{\frac{d_{i_{n}}+m_{i_{n}}}{2}} \right) \left(\prod_{n=k}^{\ell_{2}} r^{\frac{\ell_{2}}{2}} \frac{(\phi_{i_{n}}(|x_{i_{n}}|^{-2}))^{\frac{1}{2}}}{|x_{i_{n}}|^{d_{i_{n}}+m_{i_{n}}}} \right)  r^{\ell_{2}-1} \mathrm{d}r         \nonumber
\\
&= C \int_{(\phi_{i_{k-1}}(|x_{i_{k-1}}|^{-2}))^{-\frac{1}{\ell_{2}}}}^{(\phi_{i_{k}}(|x_{i_{k}}|^{-2}))^{-\frac{1}{\ell_{2}}}} \left( \prod_{n=1}^{k-1}(\phi^{-1}_{i_{n}}(r^{-\ell_{2}}))^{\frac{d_{i_{n}}+m_{i_{n}}}{2}} \right) r^{k-2} \left(\prod_{n=k}^{\ell_{2}} r^{1+\frac{\ell_{2}}{2}} \frac{(\phi_{i_{n}}(|x_{i_{n}}|^{-2}))^{\frac{1}{2}}}{|x_{i_{n}}|^{d_{i_{n}}+m_{i_{n}}}} \right)  \mathrm{d}r          \nonumber
\\
&\leq C \int_{(\phi_{i_{k-1}}(|x_{i_{k-1}}|^{-2}))^{-\frac{1}{\ell_{2}}}}^{(\phi_{i_{k}}(|x_{i_{k}}|^{-2}))^{-\frac{1}{\ell_{2}}}} \left( \prod_{n=1}^{k-1}(\phi^{-1}_{i_{n}}(r^{-\ell_{2}}))^{\frac{d_{i_{n}}+m_{i_{n}}}{2}} \right)  r^{k-2} \mathrm{d}r  \left(\prod_{n=k}^{\ell_{2}}  \frac{(\phi_{i_{n}}(|x_{i_{n}}|^{-2}))^{-\frac{1}{\ell_{2}}}}{|x_{i_{n}}|^{d_{i_{n}}+m_{i_{n}}}} \right).
\end{align}
Applying \eqref{eqn 06.05.15:11} to \eqref{eqn 12.16.18:19}, for $2 \leq k \leq \ell_{2}$, we have
\begin{align*}
& \tilde{I}_{k}(t,\vec{x})
\\
\leq& C \int_{(\phi_{i_{k-1}}(|x_{i_{k-1}}|^{-2}))^{-\frac{1}{\ell_{2}}}}^{(\phi_{i_{k}}(|x_{i_{k}}|^{-2}))^{-\frac{1}{\ell_{2}}}} \left( \prod_{n=1}^{k-1}(\phi^{-1}_{i_{n}}(r^{-\ell_{2}}))^{\frac{d_{i_{n}}+m_{i_{n}}}{2}} \right)  r^{k-2} \mathrm{d}r  \left(\prod_{n=k}^{\ell_{2}}  \frac{(\phi_{i_{n}}(|x_{i_{n}}|^{-2}))^{-\frac{1}{\ell_{2}}}}{|x_{i_{n}}|^{d_{i_{n}}+m_{i_{n}}}} \right)
\\
\leq& C  \int_{(\phi_{i_{k-1}}(|x_{i_{k-1}}|^{-2}))^{-\frac{1}{\ell_{2}}}}^{(\phi_{i_{k}}(|x_{i_{k}}|^{-2}))^{-\frac{1}{\ell_{2}}}} \left( \prod_{n=1}^{k-1}(\phi^{-1}_{i_{n}}(r^{-\ell_{2}}))^{\frac{d_{i_{n}}+m_{i_{n}}}{2}} \right)   r^{k-2} \mathrm{d}r \prod_{n=k}^{\ell_{2}} \left( \int_{(\phi_{i_{n}}(|x_{i_{n}}|^{-2}))^{-\frac{1}{\ell_{2}}}}^{2^{\frac{1}{\ell_{2}}}t^{\frac{\alpha}{\ell_{2}}}} \left( \phi_{i_{n}}^{-1}(r^{-\ell_{2}}) \right)^{\frac{d_{i_{n}}+m_{i_{n}}}{2}} \mathrm{d}r \right)
\\
\leq& C \int_{(\phi_{i_{k-1}}(|x_{i_{k-1}}|^{-2}))^{-\frac{1}{\ell_{2}}}}^{2^{\frac{1}{\ell_{2}}}t^{\frac{\alpha}{\ell_{2}}}} \left( \prod_{n=1}^{k-1}(\phi^{-1}_{i_{n}}(r^{-\ell_{2}}))^{\frac{d_{i_{n}}+m_{i_{n}}}{2}} \right)   r^{k-2} \mathrm{d}r \prod_{n=k}^{\ell_{2}} \left( \int_{(\phi_{i_{n}}(|x_{i_{n}}|^{-2}))^{-\frac{1}{\ell_{2}}}}^{2^{\frac{1}{\ell_{2}}}t^{\frac{\alpha}{\ell_{2}}}} \left( \phi_{i_{n}}^{-1}(r^{-\ell_{2}}) \right)^{\frac{d_{i_{n}}+m_{i_{n}}}{2}} \mathrm{d}r \right)
\\
=& C \, \lambda^{\ell_{2},\vec{m},k}_{\alpha,\beta}(t,x_{i_{1}},\dots,x_{i_{\ell_{2}}}).
\end{align*}
This completes the proof of Lemma.
\end{proof}

\begin{proof}[Proof of Theorem \ref{lem 06.01.17:52}]
We denote $\vec{m}_{1}=:(m_{j_{1}},\dots,m_{j_{\ell_{1}}})$, $\vec{m}_{2}=:(m_{i_{1}},\dots,m_{i_{\ell_{2}}})$, and similarly define $D^{\vec{m}_{i}}_{\vec{x}}$ for $i=1,2$. If we let
$$
P_{1}(t,\vec{x}) =: \prod_{k=1}^{\ell_{1}} p_{j_{k}}(t,x_{j_{k}}), \quad P_{2}(t,\vec{x}) =: \prod_{k=1}^{\ell_{2}} p_{i_{k}}(t,x_{i_{k}}),
$$
then under the above setting, we see that $D^{\vec{m}}_{\vec{x}}p(t,\vec{x}) = D^{\vec{m}_{1}}_{\vec{x}}P_{1}(t,\vec{x}) D^{\vec{m}_{2}}_{\vec{x}} P_{2}(t,\vec{x})$. Thus we have
\begin{align*}
&\left|  D^{\vec{m}}_{\vec{x}} q_{\alpha,\beta}(t,\vec{x}) \right|
\\
\leq&  \int_{0}^{t^{\alpha}} \left|  D^{\vec{m_{1}}}_{\vec{x}} P_{1}(r,\vec{x}) \right| \left|  D^{\vec{m_{2}}}_{\vec{x}} P_{2}(r,\vec{x}) \right| \left| \varphi_{\alpha,\beta}(t,r) \right| \mathrm{d}r
 +  \int_{t^{\alpha}}^{\infty}  \left|  D^{\vec{m_{1}}}_{\vec{x}} P_{1}(r,\vec{x}) \right| \left|  D^{\vec{m_{2}}}_{\vec{x}} P_{2}(r,\vec{x}) \right| \left| \varphi_{\alpha,\beta}(t,r) \right| \mathrm{d}r
\\
:=& J_{1}(t,\vec{x}) + J_{2}(t,\vec{x}).
\end{align*}
We first consider $J_{2}(t,\vec{x})$. By Theorem \ref{pestimate}, representation \eqref{eqn 06.24.15:02} (with bound \eqref{eqn 06.01.15:26} of $\varphi_{\alpha,\beta}$), and change of variable $t^{-\alpha}r \to r$, we have
\begin{align}\label{eqn 06.09.15:12}
&J_{2}(t,\vec{x})  \nonumber
\\
&\leq C \left( \prod_{k=1}^{\ell_{1}} \frac{(\phi_{j_{k}}(|x_{j_{k}}|^{-2}))^{\frac{1}{2}}}{|x_{j_{k}}|^{d_{j_{k}}+m_{j_{k}}}} \right) \left( \int_{t^{\alpha}}^{\infty} r^{\frac{\ell_{1}}{2}} \left( \prod_{k=1}^{\ell_{2}} (\phi^{-1}_{i_{k}}(r^{-1}))^{\frac{d_{i_{k}}+m_{i_{k}}}{2}} \right) t^{-\beta} e^{-c (rt^{-\alpha})^{1/(1-\alpha)}} \mathrm{d}r \right)  \nonumber
\\
&\leq C  \left( \prod_{k=1}^{\ell_{1}} \frac{(\phi_{j_{k}}(|x_{j_{k}}|^{-2}))^{\frac{1}{2}}}{|x_{j_{k}}|^{d_{j_{k}}+m_{j_{k}}}} \right) \int_{t^{\alpha}}^{\infty} r^{\frac{\ell_{1}}{2}} \left( \prod_{k=1}^{\ell_{2}} (\phi^{-1}_{i_{k}}(t^{-\alpha}))^{\frac{d_{i_{k}}+m_{i_{k}}}{2}} \right)  t^{-\beta} e^{-c (rt^{-\alpha})^{1/(1-\alpha)}} \mathrm{d}r  \nonumber
\\
&= C  \left( \prod_{k=1}^{\ell_{1}} \frac{(\phi_{j_{k}}(|x_{j_{k}}|^{-2}))^{\frac{1}{2}}}{|x_{j_{k}}|^{d_{j_{k}}+m_{j_{k}}}} \right) \int_{1}^{\infty} (rt^{\alpha})^{\frac{\ell_{1}}{2}} \left( \prod_{k=1}^{\ell_{2}}  (\phi^{-1}_{i_{k}}(t^{-\alpha}))^{\frac{d_{i_{k}}+m_{i_{k}}}{2}} \right) t^{-\beta} e^{-cr^{1/(1-\alpha)}} t^{\alpha} \mathrm{d}r  \nonumber
\\
& \leq C t^{\frac{\ell_{1}\alpha}{2}-\beta}  \left( \prod_{k=1}^{\ell_{1}} \frac{(\phi_{j_{k}}(|x_{j_{k}}|^{-2}))^{\frac{1}{2}}}{|x_{j_{k}}|^{d_{j_{k}}+m_{j_{k}}}} \right) \left( t^{\alpha} \prod_{k=1}^{\ell_{2}}  (\phi^{-1}_{i_{k}}(t^{-\alpha}))^{\frac{d_{i_{k}}+m_{i_{k}}}{2}} \right) \nonumber
\\
&= C  t^{\frac{\ell_{1}\alpha}{2}-\beta}  \left( \prod_{k=1}^{\ell_{1}} \frac{(\phi_{j_{k}}(|x_{j_{k}}|^{-2}))^{\frac{1}{2}}}{|x_{j_{k}}|^{d_{j_{k}}+m_{j_{k}}}} \right) \int_{t^{\alpha}}^{2t^{\alpha}}   \left( \prod_{k=1}^{\ell_{2}}  (\phi^{-1}_{i_{k}}(t^{-\alpha}))^{\frac{d_{i_{k}}+m_{i_{k}}}{2}} \right) \mathrm{d}r.
\end{align} 
Using \eqref{eqn 07.15.14:49}, we see that
$$
\phi^{-1}_{i_{k}}(t^{-\alpha}) \leq c_{0}^{-1} (rt^{-\alpha})^{1/\delta_{0}} \phi^{-1}_{i_{k}}(r^{-1}) \leq c^{-1}_{0} 2^{1/\delta_{0}} \phi^{-1}_{i_{k}}(r^{-1}) \quad \forall\, t^{\alpha}<r<2t^{\alpha}.
$$
Hence, we have
\begin{align*}
\int_{t^{\alpha}}^{2t^{\alpha}}   \left( \prod_{k=1}^{\ell_{2}}  (\phi^{-1}_{i_{k}}(t^{-\alpha}))^{\frac{d_{i_{k}}+m_{i_{k}}}{2}} \right) \mathrm{d}r
&\leq C  \int_{t^{\alpha}}^{2t^{\alpha}}   \left( \prod_{k=1}^{\ell_{2}}  (\phi^{-1}_{i_{k}}(r^{-1})^{\frac{d_{i_{k}}+m_{i_{k}}}{2}} \right) \mathrm{d}r 
\\
&\leq C  \int_{(\phi_{i_{\ell_{2}}}(|x_{i_{\ell_{2}}}|^{-2}))^{-1}}^{2t^{\alpha}}   \left( \prod_{k=1}^{\ell_{2}}  (\phi^{-1}_{i_{k}}(r^{-1})^{\frac{d_{i_{k}}+m_{i_{k}}}{2}} \right) \mathrm{d}r
\\
&= C \Lambda^{\ell_{2},\vec{m},2}_{\alpha,\beta}(t,x_{i_{1}},\dots,x_{i_{\ell_{2}}}).
\end{align*}
Therefore, using this and \eqref{eqn 06.09.15:12} we have
\begin{align}\label{eqn 06.09.15:16}
J_{2}(t,\vec{x}) \leq C t^{\frac{\ell_{1}\alpha}{2}-\beta} \left( \prod_{k=1}^{\ell_{1}} \frac{(\phi_{j_{k}}(|x_{j_{k}}|^{-2}))^{\frac{1}{2}}}{|x_{j_{k}}|^{d_{j_{k}}+m_{j_{k}}}} \right)  \Lambda^{\ell_{2},\vec{m},2}_{\alpha,\beta}(t,x_{i_{1}},\dots,x_{i_{\ell_{2}}}).
\end{align}
Likewise (use \eqref{eqn 06.01.15:30} in place of \eqref{eqn 06.01.15:26}), we can check that
\begin{align*}
&J_{1}(t,\vec{x})  \nonumber
\\
&\leq C  t^{\frac{\ell_{1}\alpha}{2}-\beta} \left( \prod_{k=1}^{\ell_{1}} \frac{(\phi_{j_{k}}(|x_{j_{k}}|^{-2}))^{\frac{1}{2}}}{|x_{j_{k}}|^{d_{j_{k}}+m_{j_{k}}}} \right)  \int_{0}^{t^{\alpha}}   \left| D^{m_{i_{1}}}_{x_{j_{1}}} \cdots D^{m_{i_{\ell_{2}}}}_{x_{i_{\ell_{2}}}}P_{2}(r,\vec{x}) \right| \mathrm{d}r  \nonumber
\\
& \leq C  t^{\frac{\ell_{1}\alpha}{2}-\beta} \left( \prod_{k=1}^{\ell_{1}} \frac{(\phi_{j_{k}}(|x_{j_{k}}|^{-2}))^{\frac{1}{2}}}{|x_{j_{k}}|^{d_{j_{k}}+m_{j_{k}}}} \right) \int_{0}^{(\phi_{i_{1}}(|x_{i_{1}}|^{-2}))^{-1}} \left| D^{m_{i_{1}}}_{x_{i_{1}}} \cdots D^{m_{i_{\ell_{2}}}}_{x_{i_{\ell_{2}}}}P_{2}(r,\vec{x}) \right|\mathrm{d}r  \nonumber
\\
&\quad  + C  t^{\frac{\ell_{1}\alpha}{2}-\beta} \left( \prod_{k=1}^{\ell_{1}} \frac{(\phi_{j_{k}}(|x_{j_{k}}|^{-2}))^{\frac{1}{2}}}{|x_{j_{k}}|^{d_{j_{k}}+m_{j_{k}}}} \right)\sum_{k=2}^{\ell_{2}} \int_{(\phi_{i_{k-1}}(|x_{i_{k-1}}|^{-2}))^{-1}}^{(\phi_{i_{k}}(|x_{i_{k}}|^{-2}))^{-1}} \left| D^{m_{i_{1}}}_{x_{i_{1}}} \cdots D^{m_{i_{\ell_{2}}}}_{x_{i_{\ell_{2}}}}P_{2}(r,x) \right| \mathrm{d}r  \nonumber
\\
&\quad +  C  t^{\frac{\ell_{1}\alpha}{2}-\beta} \left( \prod_{k=1}^{\ell_{1}} \frac{(\phi_{j_{k}}(|x_{j_{k}}|^{-2}))^{\frac{1}{2}}}{|x_{j_{k}}|^{d_{j_{k}}+m_{j_{k}}}} \right) \int_{(\phi_{i_{\ell_{2}}}(|x_{i_{\ell_{2}}}|^{-2}))^{-1}}^{t^{\alpha}}  \left| D^{m_{i_{1}}}_{x_{i_{1}}} \cdots D^{m_{i_{\ell_{2}}}}_{x_{i_{\ell_{2}}}}P_{2}(r,x) \right|   \mathrm{d}r  \nonumber
\\
& :=  C  t^{\frac{\ell_{1}\alpha}{2}-\beta} \left( \prod_{k=1}^{\ell_{1}} \frac{(\phi_{j_{k}}(|x_{j_{k}}|^{-2}))^{\frac{1}{2}}}{|x_{j_{k}}|^{d_{j_{k}}+m_{j_{k}}}} \right) \left(  \tilde{I}_{1}(t,\vec{x}) + \sum_{k=2}^{\ell_{2}} \tilde{I}_{k}(t,\vec{x}) + \tilde{I}_{\ell_{2}+1}(t,\vec{x})  \right).
\end{align*}
Applying Lemma \ref{lem 12.16.18:23}-$(i)$ (for $\tilde{I}_{1}$ and $\tilde{I}_{\ell_{2}+1}$) and Lemma \ref{lem 12.16.18:23}-$(ii)$ (for other $\tilde{I}_{k}$), we have
\begin{align}\label{eqn 12.19.18:49}
J_{1}(t,\vec{x}) \leq C  t^{\frac{\ell_{1}\alpha}{2}-\beta} \left( \prod_{k=1}^{\ell_{1}} \frac{(\phi_{j_{k}}(|x_{j_{k}}|^{-2}))^{\frac{1}{2}}}{|x_{j_{k}}|^{d_{j_{k}}+m_{j_{k}}}} \right) \sum_{i=1}^{3} \Lambda^{\ell_{2},\vec{m},i}_{\alpha,\beta}(t,x_{i_{1}},\dots,x_{i_{\ell_{2}}}).
\end{align}
We have the desired result by combining \eqref{eqn 06.09.15:16} and \eqref{eqn 12.19.18:49}. The theorem is proved.
\end{proof}

\begin{lemma}\label{lem 06.08.14:52}
For $\alpha\in (0,1)$, $\beta\in \bR$, and $t>0$, we have
\begin{align*}
\int_{\bR^{d}} \left|  q_{\alpha,\beta} (t,\vec{x}) \right| \mathrm{d}\vec{x} \leq C t^{\alpha-\beta},
\end{align*}
where the constant $C$ depends only on $\alpha,\beta,d,c_{0},\delta_{0},\ell$.
\end{lemma}

\begin{proof}
For each $t>0$, $\ell_{1},\ell_{2} \in \bN_{0}$, and $\{j_{1},\dots,j_{\ell_{1}},i_{1},\dots,i_{\ell_{2}} \}$ given as in Theorem \ref{lem 06.01.17:52}, let $A_{\ell_{1},\ell_{2}}(t)$ be a subset of $\bR^{d} \setminus \{0\}$  satisfying 
\begin{gather*}
1 \leq t^{\alpha}\phi_{i_{\ell_{2}}}(|x_{i_{\ell_{2}}}|^{-2}) \leq \dots \leq t^{\alpha}\phi_{i_{1}}(|x_{i_{1}}|^{-2}),  
\\
t^{\alpha} \phi_{j}(|x_{j}|^{-2}) \leq 1 \quad \forall\, j= j_{1},\dots,j_{\ell_{1}}.  
\end{gather*}
Since 
\begin{align*}
\int_{\bR^{d}} \left|  q_{\alpha,\beta}(t,\vec{x}) \right| \mathrm{d}\vec{x} \leq \sum \int_{A_{\ell_{1},\ell_{2}}(t)}  \left|  q_{\alpha,\beta}(t,\vec{x}) \right| \mathrm{d}\vec{x},
\end{align*}
where the summation is taken over all possible permutations $\{j_{1},\dots,j_{\ell_{1}},i_{1},\dots,i_{\ell_{2}}\}$ of $\{1,\dots,\ell\}$, we only prove 
\begin{align*}
\int_{A_{\ell_{1},\ell_{2}}(t)} |q_{\alpha,\beta}(t,\vec{x})| \mathrm{d}\vec{x} \leq C t^{\alpha-\beta}.
\end{align*}
For simplicity, we define $\tilde{x}:=(x_{i_{1}},\dots,x_{i_{\ell_{2}}})$,
\begin{gather*}
A_{\ell_{1}}(t) := \{(x_{j_{1}},\dots,x_{j_{\ell_{1}}}) : t^{\alpha}\phi_{j_{k}}(|x_{j_{k}}|^{-2}) \leq 1 \quad \forall\, k=1,\dots, \ell_{1}\},
\\
A_{\ell_{2}}(t) := \{\tilde{x}=(x_{i_{1}},\dots,x_{i_{\ell_{2}}}) :1 \leq t^{\alpha}\phi_{i_{\ell_{2}}}(|x_{i_{\ell_{2}}}|^{-2}) \leq \dots \leq t^{\alpha}\phi_{i_{1}}(|x_{i_{1}}|^{-2})\},
\end{gather*}
and (recall \eqref{eqn 12.19.13:55})
\begin{align}\label{eqn 12.19.15:53}
&\Lambda^{\ell_{2},0,3}_{\alpha,\beta}(t,\tilde{x})    \nonumber
\\
&=\sum_{k=2}^{\ell_{2}} \int_{(\phi_{i_{k-1}}(|x_{i_{k-1}}|^{-2}))^{-\frac{1}{\ell_{2}}}}^{2^{\frac{1}{\ell_{2}}}t^{\frac{\alpha}{\ell_{2}}}} \left( \prod_{n=1}^{k-1}(\phi^{-1}_{i_{n}}(r^{-\ell_{2}}))^{\frac{d_{i_{n}}}{2}} \right)   r^{k-2} \mathrm{d}r \prod_{n=k}^{\ell_{2}} \left( \int_{(\phi_{i_{n}}(|x_{i_{n}}|^{-2}))^{-\frac{1}{\ell_{2}}}}^{2^{\frac{1}{\ell_{2}}}t^{\frac{\alpha}{\ell_{2}}}} \left( \phi_{i_{n}}^{-1}(r^{-\ell_{2}}) \right)^{\frac{d_{i_{n}}}{2}} \mathrm{d}r \right)   \nonumber
\\
&:=\sum_{k=2}^{\ell_{2}} \lambda^{\ell_{2},0,k}_{\alpha,\beta}(t,\tilde{x}).
\end{align}
One can directly check that $A_{\ell_{1},\ell_{2}}(t) = A_{\ell_{1}}(t) \times A_{\ell_{2}}(t)$. 
Hence, by \eqref{eqn 06.01.18:14},  with $m=(0, \dots ,0) \in \bN_{0}^{\ell}$, we have

\begin{align}\label{eqn 07.22.18:38}
&\int_{A_{\ell_{1},\ell_{2}}(t)} \left| q_{\alpha,\beta}(t,\vec{x}) \right| \mathrm{d}\vec{x}  \nonumber
\\
\leq& C \int_{A_{\ell_{1},\ell_{2}}(t)}  t^{\frac{\alpha  \ell_{1}}{2}-\beta} \left( \prod_{k=1}^{\ell_{1}} \frac{(\phi_{j_{k}}(|x_{j_{k}}|^{-2}))^{\frac{1}{2}}}{|x_{j_{k}}|^{d_{j_{k}}}}  \right) \Lambda^{\ell_{2},m}_{\alpha,\beta}(t,\tilde{x}) \mathrm{d}\vec{x}    \nonumber
\\
\leq& C  t^{\frac{\alpha  \ell_{1}}{2}-\beta} \int_{A_{\ell_{1}}(t)\times A_{\ell_{2}}(t)} \left( \prod_{k=1}^{\ell_{1}} \frac{(\phi_{j_{k}}(|x_{j_{k}}|^{-2}))^{\frac{1}{2}}}{|x_{j_{k}}|^{d_{j_{k}}}}  \right) \prod_{k=1}^{\ell_{2}} \left( \int_{(\phi_{i_{k}}(|x_{i_{k}}|^{-2}))^{-\frac{1}{\ell_{2}}}}^{2^{\frac{1}{\ell_{2}}}t^{\frac{\alpha}{\ell_{2}}}} \left( \phi_{i_{k}}^{-1}(r^{-\ell_{2}}) \right)^{\frac{d_{i_{k}}}{2}} \mathrm{d}r \right) \mathrm{d}\vec{x}    \nonumber
\\
& + C t^{\frac{\alpha  \ell_{1}}{2}-\beta} \int_{A_{\ell_{1}}(t) \times A_{\ell_{2}}(t)} \left( \prod_{k=1}^{\ell_{1}} \frac{(\phi_{j_{k}}(|x_{j_{k}}|^{-2}))^{\frac{1}{2}}}{|x_{j_{k}}|^{d_{j_{k}}}}  \right) \Lambda^{\ell_{2},0,3}_{\alpha,\beta}(t,\tilde{x}) \mathrm{d}\vec{x}     \nonumber
\\
& + C t^{\frac{\alpha  \ell_{1}}{2}-\beta} \int_{A_{\ell_{1}}(t) \times A_{\ell_{2}}(t)} \left( \prod_{k=1}^{\ell_{1}} \frac{(\phi_{j_{k}}(|x_{j_{k}}|^{-2}))^{\frac{1}{2}}}{|x_{j_{k}}|^{d_{j_{k}}}}  \right)  \left( \int_{(\phi_{i_{\ell_{2}}}(|x_{i_{\ell_{2}}}|^{-2}))^{-1}}^{2t^{\alpha}} \left( \prod_{k=1}^{\ell_{2}} (\phi^{-1}_{i_{k}}(r^{-1}))^{\frac{d_{i_{k}}}{2}} \right)   \mathrm{d}r \right) \mathrm{d}\vec{x}    \nonumber
\\
:=& C  t^{\frac{\alpha  \ell_{1}}{2}-\beta} \left(  I_{1}(t) + I_{2}(t) + I_{3}(t)  \right).
\end{align}
Due to \eqref{int phi},
\begin{align}\label{eqn 06.09.14:40}
\int_{t^{\alpha}\phi_{j_{k}}(|x_{j_{k}}|^{-2}) \leq 1} \frac{(\phi_{j_{k}}(|x_{j_{k}}|^{-2}))^{\frac{1}{2}}}{|x_{j_{k}}|^{d_{j_{k}}}} \mathrm{d}x_{j_{k}}
&= \int_{|x_{j_{k}}| \geq (\phi^{-1}_{j_{k}}(t^{-\alpha}))^{-1/2}} \frac{(\phi_{j_{k}}(|x_{j_{k}}|^{-2}))^{\frac{1}{2}}}{|x_{j_{k}}|^{d_{j_{k}}}} \mathrm{d}x_{j_{k}}   \nonumber
\\
&= C \int_{(\phi^{-1}_{j_{k}}(t^{-\alpha}))^{-1/2}}^{\infty} (\phi_{j_{k}}(\rho^{-2}))^{1/2} \rho^{-1} \mathrm{d} \rho
\leq C t^{-\alpha/2}.
\end{align}
Therefore, by integrating on $A_{\ell_{1}}(t)$ first (use \eqref{eqn 06.09.14:40}), we have
\begin{align*}
&t^{\frac{\alpha \ell_{1}}{2}-\beta}\left( I_{1}(t) + I_{2}(t) + I_{3}(t) \right)
\\
\leq& C t^{-\beta} \int_{A_{\ell_{2}}(t)} \prod_{k=1}^{\ell_{2}} \left( \int_{(\phi_{i_{k}}(|x_{i_{k}}|^{-2}))^{-\frac{1}{\ell_{2}}}}^{2^{\frac{1}{\ell_{2}}}t^{\frac{\alpha}{\ell_{2}}}} \left(  \phi^{-1}_{i_{k}}(r^{-\ell_{2}}) \right)^{\frac{d_{i_{k}}}{2}} \mathrm{d}r \right) \mathrm{d}\tilde{x}
\\
& + C t^{-\beta} \int_{A_{\ell_{2}}(t)}  \Lambda^{\ell_{2}0,3}_{\alpha,\beta}(t,\tilde{x}) \mathrm{d}\tilde{x}
\\
&+ C t^{-\beta} \int_{A_{\ell_{2}}(t)}   \left( \int_{(\phi_{i_{\ell_{2}}}(|x_{i_{\ell_{2}}}|^{-2}))^{-1}}^{2t^{\alpha}} \left( \prod_{k=1}^{\ell_{2}} (\phi^{-1}_{i_{k}}(r^{-1}))^{\frac{d_{i_{k}}}{2}} \right)   \mathrm{d}r \right) \mathrm{d}\tilde{x}.
\end{align*}
We can check that 
\begin{align}\label{eqn 06.09.14:46}
&\int_{t^{\alpha}\phi_{i_{k}}(|x_{i_{k}}|^{-2}) \geq 1}  \int_{(\phi_{i_{k}}(|x_{i_{k}}|^{-2}))^{-\frac{1}{\ell_{2}}}}^{2^{\frac{1}{\ell_{2}}}t^{\frac{\alpha}{\ell_{2}}}} \left( \phi_{i_{k}}^{-1}(r^{-\ell_{2}}) \right)^{\frac{d_{i_{k}}}{2}} \mathrm{d}r \mathrm{d}x_{i_{k}}  \nonumber
\\
=& \int_{|x_{i_{k}}| \leq (\phi^{-1}_{i_{k}}(t^{-\alpha}))^{-1/2}}  \int_{(\phi_{i_{k}}(|x_{i_{k}}|^{-2}))^{-\frac{1}{\ell_{2}}}}^{2^{\frac{1}{\ell_{2}}}t^{\frac{\alpha}{\ell_{2}}}} \left( \phi_{i_{k}}^{-1}(r^{-\ell_{2}}) \right)^{\frac{d_{i_{k}}}{2}} \mathrm{d}r \mathrm{d}x_{i_{k}}  \nonumber
\\
\leq& \int_{0}^{2^{\frac{1}{\ell_{2}}}t^{\alpha/\ell_{2}}} \int_{|x_{i_{k}}| \leq (\phi_{i_{k}}(r^{-\ell_{2}}))^{-1/2}} \left( \phi_{i_{k}}^{-1}(r^{-\ell_{2}}) \right)^{\frac{d_{i_{k}}}{2}} \mathrm{d}x_{i_{k}} \mathrm{d}r =  C t^{\alpha/\ell_{2}}.
\end{align}
Therefore, we have
\begin{align}\label{eqn 02.17.21:26}
t^{\frac{\alpha \ell_{1}}{2}-\beta}I_{1}(t)&\leq C t^{-\beta} \int_{A_{\ell_{2}}(t)} \prod_{k=1}^{\ell_{2}} \left( \int_{(\phi_{i_{k}}(|x_{i_{k}}|^{-2}))^{-\frac{1}{\ell_{2}}}}^{2^{\frac{1}{\ell_{2}}}t^{\frac{\alpha}{\ell_{2}}}} \left(  \phi^{-1}_{i_{k}}(r^{-\ell_{2}}) \right)^{\frac{d_{i_{k}}}{2}} \mathrm{d}r \right) \mathrm{d}\tilde{x}    \nonumber
\\
&=Ct^{-\beta}  \left( \prod_{k=1}^{\ell_{2}} \int_{t^{\alpha}\phi_{i_{k}}(|x_{i_{k}}|^{-2}) \geq 1}  \int_{(\phi_{i_{k}}(|x_{i_{k}}|^{-2}))^{-\frac{1}{\ell_{2}}}}^{2^{\frac{1}{\ell_{2}}}t^{\frac{\alpha}{\ell_{2}}}} \left( \phi_{i_{k}}^{-1}(r^{-\ell_{2}}) \right)^{\frac{d_{i_{k}}}{2}} \mathrm{d}r \mathrm{d}x_{i_{k}}\right)    \nonumber
\\  
&\leq C t^{\alpha-\beta}.
\end{align}
Also, due to the definition of $A_{\ell_{2}}(t)$, we have 
$
(\phi_{i_{1}}(|x_{i_{1}}|^{-2}))^{-1} \leq \cdots \leq (\phi_{i_{\ell_{2}}}(|x_{i_{\ell_{2}}}|^{-2}))^{-1} \leq t^{\alpha}$ on $A_{\ell_{2}}(t)$.
Therefore, we have
\begin{align}\label{eqn 06.09.14:59}
t^{\frac{\alpha \ell_{1}}{2}-\beta}I_{2}(t)&\leq C t^{-\beta}\int_{A_{\ell_{2}}(t)} \int_{(\phi_{i_{\ell_{2}}}(|x_{i_{\ell_{2}}}|^{-2}))^{-1}}^{2t^{\alpha}} \left( \prod_{k=1}^{\ell_{2}} (\phi^{-1}_{i_{k}}(r^{-1}))^{\frac{d_{i_{k}}}{2}} \right)   \mathrm{d}r \mathrm{d}\tilde{x}  \nonumber
\\
&=C t^{-\beta}\int_{A_{\ell_{2}}(t)} \int_{0}^{2t^{\alpha}} \mathbf{1}_{\{(\phi_{i_{1}}(|x_{i_{1}}|^{-2}))^{-1} \leq r\}} \cdots \mathbf{1}_{\{(\phi_{i_{\ell_{2}}}(|x_{i_{\ell_{2}}}|^{-2}))^{-1} \leq r\}} \left( \prod_{k=1}^{\ell_{2}} (\phi^{-1}_{i_{k}}(r^{-1}))^{\frac{d_{i_{k}}}{2}} \right)  \mathrm{d}r \mathrm{d}\tilde{x}  \nonumber
\\
& \leq C t^{-\beta} \int_{0}^{2t^{\alpha}} \int_{|x_{i_{1}}| \leq (\phi^{-1}_{i_{1}}(r^{-1}))^{-1/2}}\cdots  \int_{|x_{i_{\ell_{2}}}| \leq (\phi^{-1}_{i_{\ell_{2}}}(r^{-1}))^{-1/2}}  \left( \prod_{k=1}^{\ell_{2}} (\phi^{-1}_{i_{k}}(r^{-1}))^{\frac{d_{i_{k}}}{2}} \right)  \mathrm{d}x_{i_{\ell_{2}}}\dots \mathrm{d}x_{i_{1}} \mathrm{d}r \nonumber
\\
&=C t^{\alpha-\beta}. 
\end{align}
For each $k=2,\dots,\ell_{2}$, using \eqref{eqn 06.09.14:46}, we have
\begin{align}\label{eqn 02.17.22:14}
&\int_{A_{\ell_{2}}(t)}  \lambda^{\ell_{2},0,k}_{\alpha,\beta}(t,\tilde{x}) \mathrm{d}\tilde{x}  \nonumber
\\
\leq& Ct^{\frac{\alpha(\ell_{2}-k+1)}{\ell_{2}}} \int_{\tilde{A}_{\ell_{2}}(t)}  \int_{(\phi_{i_{n-1}}(|x_{i_{n-1}}|^{-2}))^{-\frac{1}{\ell_{2}}}}^{2^{\frac{1}{\ell_{2}}}t^{\frac{\alpha}{\ell_{2}}}} \left( \prod_{n=1}^{k-1}(\phi^{-1}_{i_{n}}(r^{-\ell_{2}}))^{\frac{d_{i_{n}}}{2}} \right)   r^{k-2} \mathrm{d}r \mathrm{d}x',
\end{align}
where $\tilde{A}_{\ell_{2}}(t) := \{ x'=(x_{i_{1}},\dots,x_{i_{k-1}}):1 \leq t^{\alpha}\phi_{i_{k-1}}(|x_{i_{k-1}}|^{-2}) \leq \dots \leq t^{\alpha}\phi_{i_{1}}(|x_{i_{1}}|^{-2}) \}$. Repeating the argument in \eqref{eqn 06.09.14:59},
\begin{align}\label{eqn 02.17.22:15}
&\int_{\tilde{A}_{\ell_{2}}(t)} \int_{(\phi_{i_{n-1}}(|x_{i_{n-1}}|^{-2}))^{-\frac{1}{\ell_{2}}}}^{2^{\frac{1}{\ell_{2}}}t^{\frac{\alpha}{\ell_{2}}}} \left( \prod_{n=1}^{k-1}(\phi^{-1}_{i_{k}}(r^{-\ell_{2}}))^{\frac{d_{i_{k}}}{2}} \right)   r^{k-2} \mathrm{d}r \mathrm{d}x'  \nonumber
\\
\leq& C \int_{0}^{2^{\frac{1}{\ell_{2}}}t^{\alpha/\ell_{2}}} \int_{|x_{i_{1}}| \leq (\phi^{-1}_{i_{1}}(r^{-\ell_{2}}))^{-1/2}}\dots  \int_{|x_{i_{n-1}}| \leq (\phi^{-1}_{i_{\ell_{2}}}(r^{-\ell_{2}}))^{-1/2}}  \left( \prod_{n=1}^{k-1} (\phi^{-1}_{i_{n}}(r^{-\ell_{2}}))^{\frac{d_{i_{n}}}{2}} \right)   r^{k-2} \mathrm{d}x_{i_{k-1}}\dots \mathrm{d}x_{i_{1}} \mathrm{d}r  \nonumber
\\
\leq& C t^{\frac{\alpha(k-1)}{\ell_{2}}}.
\end{align}
Thus we have
\begin{align*}
t^{-\beta} \int_{A_{\ell_{2}}(t)} \lambda^{\ell_{2},0,k}(t,\tilde{x}) \mathrm{d}\tilde{x} \leq C t^{\alpha-\beta} \quad \forall \, 2\leq k \leq \ell_{2},
\end{align*}
which directly yields (recall \eqref{eqn 12.19.15:53})
\begin{align}\label{eqn 12.19.16:13}
t^{\frac{\alpha \ell_{1}}{2}-\beta}I_{3}(t) \leq C t^{-\beta}\int_{A_{\ell_{2}}(t)}  \Lambda^{\ell_{2},0,3}_{\alpha,\beta}(t,\tilde{x}) \mathrm{d}\vec{x} \leq C t^{\alpha-\beta}.
\end{align}
One gets the desired result by combining \eqref{eqn 02.17.21:26}, \eqref{eqn 06.09.14:59}, and \eqref{eqn 12.19.16:13}. The lemma is proved.
\end{proof}

\begin{corollary}\label{lem 02.17.22:32}
Let $\alpha \in (0,1)$, $\beta\in \bR$, $i \in \{1,\dots, \ell \}$, and let $m$ be a $d_{i}$-dimensional multi-index.
\\
(i) Suppose that $t^{\alpha}\phi_{i}(|x_{i}|^{-2}) \leq 1$. Then we have
\begin{align*}
\int_{\bR^{d-d_{i}}} \left| D^{m}_{x_{i}}q_{\alpha,\beta}(t,\vec{x}) \right| \mathrm{d}x_{1} \cdots \mathrm{d}x_{i-1}\mathrm{d}x_{i+1}\cdots \mathrm{d}x_{\ell} \leq C t^{\frac{3\alpha}{2}-\beta} \frac{(\phi_{i}(|x_{i}|^{-2}))^{1/2}}{|x_{i}|^{d_{i}+m}},
\end{align*}
where the constant $C>0$ depends only on $\alpha,\beta,c_{0},\delta_{0},d,\ell,m$.

(ii) Suppose that $t^{\alpha}\phi_{i}(|x_{i}|^{-2}) \geq 1$. Then we have
\begin{align}\label{eqn 02.17.23:45}
\int_{\bR^{d-d_{i}}} \left| D^{m}_{x_{i}}q_{\alpha,\beta}(t,\vec{x}) \right| \mathrm{d}x_{1} \cdots \mathrm{d}x_{i-1}\mathrm{d}x_{i+1}\cdots \mathrm{d}x_{\ell} \nonumber
& \leq C \sum_{k=1}^{\ell} \left( t^{\alpha-\frac{\alpha}{k}-\beta} \int_{(\phi_{i}(|x_{i}|^{-2}))^{-\frac{1}{k}}}^{2^{\frac{1}{k}}t^{\frac{\alpha}{k}}} \left( \phi^{-1}_{i}(r^{-k}) \right)^{\frac{d_{i}+m}{2}} \mathrm{d}r \right)  \nonumber
\\
& \left( \leq C  t^{\frac{3\alpha}{2}-\beta} \frac{(\phi_{i}(|x_{i}|^{-2}))^{1/2}}{|x_{i}|^{d_{i}+m}} \right),
\end{align}
where the constant $C>0$ depends only on $\alpha,\beta,c_{0},\delta_{0},d,\ell,m$.

(iii) For any $0<\varepsilon<T$, we have
\begin{align}\label{eqn 06.24.16:34}
\int_{\bR^{d}} \sup_{t \in [\varepsilon,T]}  \left| q_{\alpha,\beta}(t,\vec{x}) \right| \mathrm{d}\vec{x} <C,    
\end{align}
where the constant $C>0$ depends only on $\alpha,\beta,d,c_{0},\delta_{0},\ell,\varepsilon,T$. 
\end{corollary}

\begin{proof}
As in Theorem \ref{lem 06.01.17:52}, take $0\leq \ell_{1},\ell_{2} \leq \ell$ and let $\{j_{1},\dots,j_{\ell_{1}},i_{1},\dots,i_{\ell_{2}}\}$ be a permutation of $\{1,\dots,\ell\}$.

(i) Since $t^{\alpha}\phi_{i}(|x_{i}|^{-2}) \leq 1$, $\ell_{1} \geq 1$, and $i \in \{j_{1},\dots,j_{\ell_{1}} \}$, without loss of generality, we assume that $i=j_{1}$. Then by following the proof of Lemma \ref{lem 06.08.14:52}, only ignoring the integration with respect to $x_{j_{1}}$, we have the desired result. For example, if we start from \eqref{eqn 02.17.21:26}, and replace 
$$
 t^{\frac{\alpha\ell_{1}}{2}-\beta} \left(\prod_{k=1}^{\ell_{1}}  \int_{t^{\alpha}\phi_{j_{k}}(|x_{j_{k}}|^{-2}) \leq 1} \frac{(\phi_{j_{k}}(|x_{j_{k}}|^{-2}))^{\frac{1}{2}}}{|x|^{d_{j_{k}}+m_{j_{k}}}} \mathrm{d}x_{j_{k}} \right)
$$
by (recall also \eqref{eqn 06.01.18:14})
$$
 t^{\frac{\alpha\ell_{1}}{2}-\beta}\left( \frac{(\phi_{i}(|x_{i}|^{-2}))^{1/2}}{|x_{i}|^{d_{i}+m}} \right) \left(\prod_{k=2}^{\ell_{1}}  \int_{t^{\alpha}\phi_{j_{k}}(|x_{j_{k}}|^{-2}) \leq 1} \frac{(\phi_{j_{k}}(|x_{j_{k}}|^{-2}))^{\frac{1}{2}}}{|x|^{d_{j_{k}}}} \mathrm{d}x_{j_{k}} \right),
$$
we get, instead,  
$$
t^{\frac{\alpha \ell_{1}}{2}-\beta}I_{1}(t) \leq C t^{\frac{3\alpha}{2}-\beta}\frac{(\phi_{i}(|x_{i}|^{-2}))^{1/2}}{|x_{i}|^{d_{i}+m}}.
$$

(ii) Like (i), assume that $i = i_{k}$ for $k \in \{1,\dots,\ell_{2} \}$. Then if we follow \eqref{eqn 02.17.21:26} and ignoring the integration with respect to $x_{i}$, then we have
\begin{align}\label{eqn 02.17.22:22}
t^{\frac{\alpha \ell_{1}}{2}-\beta} I_{1}(t) \leq C t^{\alpha-\frac{\alpha}{\ell_{2}}-\beta} \int_{(\phi_{i}(|x_{i}|^{-2}))^{-\frac{1}{\ell_{2}}}}^{2^{\frac{1}{\ell_{2}}}t^{\frac{\alpha}{\ell_{2}}}} \left( \phi^{-1}_{i}(r^{-\ell_{2}}) \right)^{\frac{d_{i}+m}{2}} \mathrm{d}r.
\end{align}
If we follow the \eqref{eqn 06.09.14:59} and ignoring the integration with respect to $x_{i}$, we get
\begin{align}\label{eqn 02.17.22:22-2}
t^{\frac{\alpha \ell_{1}}{2}-\beta}I_{3}(t) &\leq C t^{-\beta}\int_{(\phi_{i}(|x_{i}|^{-2}))^{-1}}^{2t^{\alpha}} (\phi_{i}(r^{-1}))^{\frac{d_{i}+m}{2}} \mathrm{d}r.
\end{align}
Also, for $n=2,\dots \ell_{2}$, by following \eqref{eqn 02.17.22:14}, and \eqref{eqn 02.17.22:15} ignoring the integration with respect to $x_{i}$, we have
\begin{align*}
&t^{\frac{\alpha \ell_{1}}{2}-\beta} I_{2}(t)
\\
\leq& \mathbf{1}_{n\leq i \leq \ell_{2}} \, C  t^{\alpha-\frac{\alpha}{\ell_{2}}-\beta} \int_{(\phi_{i}(|x_{i}|^{-2}))^{-\frac{1}{\ell_{2}}}}^{2^{\frac{1}{\ell_{2}}}t^{\frac{\alpha}{\ell_{2}}}} (\phi^{-1}_{i}(r^{-\ell_{2}}))^{\frac{d_{i}+m}{2}} \mathrm{d}r
\\
& + \mathbf{1}_{1\leq i \leq n-1} \, C t^{\frac{\alpha(\ell_{2}-n+1)}{\ell_{2}}-\beta} \int_{(\phi_{i}(|x_{i}|^{-2}))^{-\frac{1}{\ell_{2}}}}^{2^{\frac{1}{\ell_{2}}}t^{\frac{\alpha}{\ell_{2}}}} \left( \phi_{i}(r^{-\ell_{2}}) \right)^{\frac{d_{i}+m}{2}} r^{n-2} \mathrm{d}r
\\
\leq& C   t^{\alpha-\frac{\alpha}{\ell_{2}}-\beta} \int_{(\phi_{i}(|x_{i}|^{-2}))^{-\frac{1}{\ell_{2}}}}^{2^{\frac{1}{\ell_{2}}}t^{\frac{\alpha}{\ell_{2}}}} (\phi^{-1}_{i}(r^{-\ell_{2}}))^{\frac{d_{i}+m}{2}}  \mathrm{d}r.
\end{align*}
Combining this with \eqref{eqn 02.17.22:22} \eqref{eqn 02.17.22:22-2}, and then summing those terms with respect to $\ell_{2}=1,\dots,\ell$, we have the desire result. Finally, for the estimation \eqref{eqn 02.17.23:45}, use the fact that $\phi^{-1}_{i}(r^{-\ell_{2}}) \leq |x_{i}|^{-2}$ for $r\geq (\phi_{i}(|x_{i}|^{-2}))^{-1/\ell_{2}}$, and the assumption $t^{-\alpha} \leq \phi_{i}(|x_{i}|^{-2})$. The lemma is proved.

(iii) By Theorem \ref{lem 06.01.17:52} (with $\vec{m}=0$), we have
\begin{align*}
\sup_{t\in [\varepsilon,T]}\left|  q_{\alpha,\beta}(t,\vec{x}) \right| \leq C(\alpha,\beta,d,c_{0},\delta_{0},\varepsilon,T) \left( \prod_{n=1}^{\ell_{1}} \frac{(\phi_{j_{n}}(|x_{j_{n}}|^{-2}))^{\frac{1}{2}}}{|x_{j_{n}}|^{d_{j_{n}}+m_{j_{n}}}}  \right) \Lambda^{\ell_{2},0}_{\alpha,\beta}(T,x_{i_{1}},\dots,x_{i_{\ell_{2}}}),
\end{align*}
where $\Lambda^{\ell_{2},0}_{\alpha,\beta}$ is taken from the statement of Theorem \ref{lem 06.01.17:52}. 
Hence, we have 
\begin{align*}
\int_{\bR^{d}} \sup_{t\in [\varepsilon,T]} \left| q_{\alpha,\beta}(t,\vec{x}) \right| \mathrm{d}\vec{x} 
\leq C \sum \int_{A_{\ell_{1},\ell_{2}}(\varepsilon)} \left( \prod_{n=1}^{\ell_{1}} \frac{(\phi_{j_{n}}(|x_{j_{n}}|^{-2}))^{\frac{1}{2}}}{|x_{j_{n}}|^{d_{j_{n}}+m_{j_{n}}}}  \right) \Lambda^{\ell_{2},0}_{\alpha,\beta}(T,x_{i_{1}},\dots,x_{i_{\ell_{2}}}) \mathrm{d}\vec{x},
\end{align*}
where $A_{\ell_{1},\ell_{2}}(\varepsilon)$ comes from Lemma \ref{lem 06.08.14:52}, and the summation is taken over all possible permutations $\{j_{1},\dots,j_{\ell_{1}},i_{1},\dots,i_{\ell_{2}} \}$ of $\{1,\dots,\ell\}$.
Then splitting each integral in the right-hand side like \eqref{eqn 07.22.18:38}, we have
\begin{align*}
&\int_{A_{\ell_{1},\ell_{2}}(\varepsilon)} \left( \prod_{n=1}^{\ell_{1}} \frac{(\phi_{j_{n}}(|x_{j_{n}}|^{-2}))^{\frac{1}{2}}}{|x_{j_{n}}|^{d_{j_{n}}+m_{j_{n}}}}  \right) \Lambda^{\ell_{2},0}_{\alpha,\beta}(T,x_{i_{1}},\dots,x_{i_{\ell_{2}}}) \mathrm{d}\vec{x}
\\
=&  \int_{A_{\ell_{1}}(\varepsilon)\times A_{\ell_{2}}(\varepsilon)} \left( \prod_{k=1}^{\ell_{1}} \frac{(\phi_{j_{k}}(|x_{j_{k}}|^{-2}))^{\frac{1}{2}}}{|x|^{d_{j_{k}}}}  \right) \prod_{k=1}^{\ell_{2}} \left( \int_{(\phi_{i_{k}}(|x_{i_{k}}|^{-2}))^{-\frac{1}{\ell_{2}}}}^{2^{\frac{1}{\ell_{2}}}T^{\frac{\alpha}{\ell_{2}}}} \left( \phi_{i_{k}}^{-1}(r^{-\ell_{2}}) \right)^{\frac{d_{i_{k}}}{2}} \mathrm{d}r \right) \mathrm{d}\vec{x}
\\
& +  \int_{A_{\ell_{1}}(\varepsilon) \times A_{\ell_{2}}(\varepsilon)} \left( \prod_{k=1}^{\ell_{1}} \frac{(\phi_{j_{k}}(|x_{j_{k}}|^{-2}))^{\frac{1}{2}}}{|x|^{d_{j_{k}}}}  \right) \Lambda^{\ell_{2},0,3}_{\alpha,\beta}(T,\tilde{x}) \mathrm{d}\vec{x}
\\
& + C  \int_{A_{\ell_{1}}(\varepsilon) \times A_{\ell_{2}}(\varepsilon)} \left( \prod_{k=1}^{\ell_{1}} \frac{(\phi_{j_{k}}(|x_{j_{k}}|^{-2}))^{\frac{1}{2}}}{|x|^{d_{j_{k}}}}  \right)  \int_{(\phi_{i_{\ell_{2}}}(|x_{i_{\ell_{2}}}|^{-2}))^{-1}}^{2T^{\alpha}} \left( \prod_{k=1}^{\ell_{2}} (\phi^{-1}_{i_{k}}(r^{-1}))^{\frac{d_{i_{k}}}{2}} \right)   \mathrm{d}r \mathrm{d}\vec{x}
\\
:=&\left(  I_{1} + I_{2} + I_{3}  \right),
\end{align*}
where $A_{\ell_{1}}(\varepsilon)$, $A_{\ell_{2}}(\varepsilon)$ are defined as in Lemma \ref{lem 06.08.14:52}. Following \eqref{eqn 06.09.14:40}, and \eqref{eqn 06.09.14:46} we have
$$
\int_{\varepsilon^{\alpha}\phi_{j_{k}}(|x_{j_{j}}|^{-2})\leq 1} \frac{(\phi_{j_{k}}(|x_{j_{k}}|^{-2}))^{\frac{1}{2}}}{|x_{j_{k}}|^{d_{j_{k}}}} \mathrm{d}x_{j_{k}} \leq C \varepsilon^{-\alpha/2},
$$
and
\begin{align*}
&\int_{\varepsilon^{\alpha}\phi_{i_{k}}(|x_{i_{k}}|^{-2}) \geq 1}  \int_{(\phi_{i_{k}}(|x_{i_{k}}|^{-2}))^{-\frac{1}{\ell_{2}}}}^{2^{\frac{1}{\ell_{2}}}T^{\frac{\alpha}{\ell_{2}}}} \left( \phi_{i_{k}}^{-1}(r^{-\ell_{2}}) \right)^{\frac{d_{i_{k}}}{2}} \mathrm{d}r \mathrm{d}x_{i_{k}} \leq   C T^{\alpha/\ell_{2}}
\end{align*}
respectively. Hence, $I_{1} \leq C$. Similarly, following the argument in \eqref{eqn 06.09.14:59}, \eqref{eqn 02.17.22:14}, and \eqref{eqn 02.17.22:15}, we have $I_{2}+I_{3} \leq C$. 
This completes the proof of corollary.
\end{proof}

\begin{proof}[Proof of Theorem \ref{lem 06.24.15:35}]
    Theorem \ref{lem 06.24.15:35} can be proved by following \cite[Lemma 3.5]{KKL17} using Lemma \ref{lem 06.08.14:52} and Corollary \ref{lem 02.17.22:32}-$(iii)$.
\end{proof}

\mysection{Maximal Regularity Estimates of Solutions in Mixed-Norm Lebesgue Space}\label{25.04.07.15.44}

In this section, we establish maximal regularity estimates for solutions to the equation
\begin{equation*}
    \begin{cases}
    \partial^{\alpha}_{t}u(t,\vec{x}) = \vec{\phi} \cdot \Delta_{\vec{d}} \, u(t,\vec{x}) +f(t,\vec{x}),\quad &(t,\vec{x})\in(0,\infty)\times\mathbb{R}^d,\\
    u(0,\vec{x})=0,\quad &\vec{x}\in\mathbb{R}^d,
    \end{cases}
\end{equation*}
in the mixed-norm space $L_q((0,T);L_p(\mathbb{R}^d))$ for $f\in C_c^{\infty}(\mathbb{R}^{d+1}_+)$, \textit{i.e.},
\begin{equation}
\label{25.03.13.16.55}
\|\vec{\phi}\cdot\Delta_{\vec{d}}\,u\|_{L_q((0,T);L_p(\mathbb{R}^d))}\leq C\|f\|_{L_q((0,T);L_p(\mathbb{R}^d))}.
\end{equation}
To derive \eqref{25.03.13.16.55}, we utilize Theorem \ref{lem 06.24.15:35}, which reduces the problem to prove the boundedness of the solution operator $\mathcal{G}_0$ in $L_q((0,T);L_p(\mathbb{R}^d))$, where $\mathcal{G}_0$ is an operator given by
\begin{align*}
f\mapsto \mathcal{G}_{0}f(t,\vec{x}) = \int_{0}^{t} \int_{\bR^{d}} q_{\alpha,1}(t-s, \vec{x}-\vec{y}) f(s,\vec{y}) \mathrm{d}\vec{y} \mathrm{d}s.
\end{align*}
We now present the main result of this section.
\begin{theorem}\label{thm 07.22.11:31}
Let $1<p,q<\infty$. 
Then for any $f \in C^{\infty}_{c}(\bR^{d+1})$, we have
\begin{equation}\label{eqn 07.02.14:56}
\|\vec{\phi}\cdot\Delta_{\vec{d}}\,\mathcal{G}_0f\|_{L_{q}(\bR;L_{p}(\bR^{d}))} \leq C \|f\|_{L_{q}(\bR;L_{p}(\bR^{d}))},
\end{equation}
where $C=C(\alpha,d,c_{0},\delta_{0},\ell,p,q)$. 
Therefore, the operator $\vec{\phi}\cdot\Delta_{\vec{d}}\,\mathcal{G}_0$ extends continuously to $L_{q}(\bR ;L_{p}(\bR^{d}))$.
\end{theorem}
The proof of Theorem \ref{thm 07.22.11:31} will be given at the end of this section.
To establish Theorem \ref{thm 07.22.11:31}, we first derive a key estimate for the operator $\vec{\phi}\cdot\Delta_{\vec{d}}\,\mathcal{G}_0$.
The following lemma provides an integral representation of $\vec{\phi}\cdot\Delta_{\vec{d}}\,\mathcal{G}_0$ and its boundedness in $L_2(\mathbb{R}^{d+1})$.
\begin{lemma}
Let $\mathcal{G}$ be defined by
\begin{align*}
\mathcal{G}f(t,\vec{x}):=\vec{\phi}\cdot\Delta_{\vec{d}}\left[ \mathcal{G}_{0}f(t,\cdot) \right](\vec{x}).
\end{align*}
Then for $f\in C^{\infty}_{c}(\bR^{d+1}_{+})$,
\begin{align}\label{eqn 04.28.10:33}
\mathcal{G}f(t,\vec{x}) = \int_{0}^{t} \int_{\bR^{d}} q_{\alpha,\alpha+1}(t-s, \vec{x}-\vec{y}) f(s,\vec{y}) \mathrm{d}\vec{y} \mathrm{d}s,
\end{align}
and
\begin{align}\label{eqn 06.24.15:27}
\|\mathcal{G}f\|_{L_{2}(\bR^{d+1})} \leq C \|f\|_{L_{2}(\bR^{d+1})}.
\end{align}
Here $q_{\alpha,\alpha+1}$ is the function defined in \eqref{eqn 06.24.15:02}.
\end{lemma}
\begin{proof}
Let $f \in C^{\infty}_{c}(\bR^{d+1}_{+})$, and for $0<s<t$, define
\begin{gather*}
G_{0}f(t-s,\vec{x}) = \int_{\bR^{d}}  \vec{\phi} \cdot \Delta_{\vec{d}} \,\, q_{\alpha,1}(t-s,\vec{x}-\vec{y}) f(s,\vec{y}) \mathrm{d}\vec{y},
\\
Gf(t-s,\vec{x}) = \int_{\bR^{d}} q_{\alpha,\alpha+1}(t-s, \vec{x}-\vec{y}) f(s,\vec{y}) \mathrm{d}\vec{y}.
\end{gather*}
Since $f\in C^{\infty}_{c}(\bR^{d+1}_{+})$, both integrals above are well-defined and continuous in $\vec{x}$. From \cite[Lemma 3.7 (iv)]{KPR21nonlocal}, we have
\begin{equation}\label{eqn 06.01.16:14}
\cF_{d}[q_{\alpha,\beta}(t,\cdot)](\xi) = t^{\alpha-\beta}E_{\alpha,1-\beta+\alpha}\left( -t^{\alpha}\sum_{i=1}^{\ell} \phi_{i}(|\xi_{i}|^{2}) \right),
\end{equation}
where $E_{a,b}(z)$ is the two-parameter Mittag-Leffler function defined as
$$
E_{a,b}(z) = \sum_{k=1}^{\infty}\frac{z^{k}}{\Gamma(ak+b)} \quad z\in \mathbb{C}.
$$
Applying \eqref{eqn 06.01.16:14} together with \eqref{eqn 06.24.16:34}, we obtain
\begin{align*}
\mathcal{F}_{d}\left[ G_{0}f(t-s,\cdot) \right] (\vec{\xi}) &= \mathcal{F}_{d}[q_{\alpha,1}(t-s,\cdot)](\vec{\xi}) \left( \sum_{i=1}^{\ell} -\phi_{i}(|\xi_{i}|^{2}) \right)\mathcal{F}_{d}[f(s,\cdot)](\vec{\xi})
\\
& = \mathcal{F}_{d}[q_{\alpha,\alpha+1}(t-s,\cdot)](\vec{\xi}) \mathcal{F}_{d}[f(s,\cdot)](\vec{\xi})
\\
&= \mathcal{F}_{d}[Gf(t-s,\cdot)](\vec{\xi}).
\end{align*}
Therefore, we have $G_{0}f(t-s,\vec{x}) = Gf(t-s,\vec{x})$ for all $0<s<t$ and $\vec{x} \in \bR^{d}$.
This establishes \eqref{eqn 04.28.10:33}.
For the estimate \eqref{eqn 06.24.15:27}, we follow the proof of \cite[Lemma 4.2]{KPR21nonlocal}.
This completes the proof.
\end{proof}

The next key step in proving Theorem \ref{thm 07.22.11:31} is to establish mean oscillation estimates for $\mathcal{G}f$.
To describe these estimates, we first introduce some notions related to BMO spaces.
For measurable subsets $E\subset \R^{d+1}$ with finite measure and locally integrable functions $h$, we define the average of $h$ over $E$ as
\begin{equation*}
h_E:=\aint_{E}h(s,\vec{y})\mathrm{d}\vec{y}\mathrm{d}s:=\aint_{E}h(s,y_{1},\dots,y_{\ell})\mathrm{d}y_{1}\cdots\mathrm{d}y_{\ell}\mathrm{d}s:=\frac{1}{|E|}\int_{E}h(s,y_{1},\dots,y_{\ell})\mathrm{d}y_{1}\cdots\mathrm{d}y_{\ell}\mathrm{d}s,
\end{equation*}
where $|E|$ is the $(d+1)$-dimensional Lebesgue measure of $E$. 
To specify the class of measurable sets under consideration, we introduce the following notations:
\begin{equation*}
\kappa_{i}(b):=\left(\phi^{-1}_{i}(b^{-\alpha})\right)^{-1/2}, \quad b>0,
\end{equation*}
\begin{equation*}\label{eqn 12.26.17:26}
Q_b(t, \vec{x}):=(t-b,t+b)\times  \prod_{i=1}^{\ell}B^{i}_{\kappa_{i}(b)}(x_{i}):= (t-b,t+b)\times \mathbb{B}_{\kappa(b)}(\vec{x}),
\end{equation*}
and
\begin{equation*}
B^{i}_{\kappa_{i}(b)}=B^{i}_{\kappa_{i}(b)}(0), \quad \mathbb{B}_{\kappa(b)} = \prod_{i=1}^{\ell}B^{i}_{\kappa_{i}(b)},\quad Q_b=Q_b(0,\vec{0})= (-b,b)\times \mathbb{B}_{\kappa(b)}.
\end{equation*}
From \eqref{eqn 07.15.14:49}, we have
\begin{equation}\label{eqn 07.15.19:01}
|Q_{\lambda b}(t,x)| \leq \lambda c^{-\frac{d}{\delta_{0}}}_{0}\lambda^{\frac{\alpha d}{2\delta_{0}}} |Q_{b}(t,x)| \quad \forall\, \lambda >1.
\end{equation}
For a locally integrable function $h$ on $\bR^{d+1}$, we define the BMO semi-norm of $h$ on $\bR^{d+1}$ as
\begin{equation*}
\|h\|_{BMO(\bR^{d+1})}:=\|h^{\#}\|_{L_{\infty}(\bR^{d+1})}
\end{equation*}
where
\begin{align*}
h^{\#}(t,\vec{x}):=\sup_{(t,\vec{x})\in Q_b(r,\vec{z})} \aint_{Q_b(r,\vec{z})}|h(s,\vec{y})-h_{Q_b(r,\vec{z})}| \mathrm{d}y \mathrm{d}s.
\end{align*}
We now state the following theorem, which establishes the mean oscillation estimates for $\mathcal{G}f$.
\begin{theorem}
\label{23.02.22.17.33}
For any $f\in L_2(\bR^{d+1})\cap L_\infty(\bR^{d+1})$,
\begin{align} \label{bmoestimate}
\|\cG f\|_{BMO(\bR^{d+1})} \leq C(\alpha,d,c_{0}, \delta_{0},\ell) \|f\|_{L_\infty(\bR^{d+1})}.
\end{align}
\end{theorem}
\begin{proof}
Let $(t_0,\vec{x}_0)\in\bR^{d+1}$. Then due to the definition of BMO semi-norm, it suffices to prove
\begin{align*}
\int_{Q_{b}(t_{0},\vec{x_{0}})} \left|  \mathcal{G}f(t,\vec{x}) - (\mathcal{G}f)_{Q_{b}(t_{0},\vec{x}_{0})} \right| \mathrm{d}\vec{x} \mathrm{d}t \leq C \|f\|_{L_{\infty}(\bR^{d+1})}, \quad \forall\, b>0,
\end{align*}
where $C$ is independent of $b$ and $(t_0,\vec{x}_0)$.
Applying the change of variable formula, we observe that
\begin{eqnarray*}
\aint_{Q_b(t_0,\vec{x}_0)} |\cG f(t,\vec{x})-(\cG f)_{Q_b(t_0,\vec{x}_0)}| \mathrm{d}\vec{x} \mathrm{d}t 
= \aint_{Q_b} |\cG \tilde{f}(t,\vec{x})-(\cG \tilde{f})_{Q_b}| \mathrm{d}\vec{x} \mathrm{d}t ,
\end{eqnarray*}
where $\tilde{f}(t,\vec{x}):=f(t+t_0, \vec{x}+\vec{x}_0)$. 
Since the $L_{\infty}(\bR^{d+1})$-norm is invariant under translation, the problem reduces to proving
\begin{equation} \label{meanaverageineq}
\aint_{Q_b} |\cG f(t,\vec{x})-(\cG f)_{Q_b}|  \mathrm{d}\vec{x} \mathrm{d}t \leq C \|f\|_{L_\infty(\bR^{d+1})}, \quad  \forall \, b>0.
\end{equation}
The proof of \eqref{meanaverageineq} for $f\in C^{\infty}_{c}(\bR^{d+1})$ will be presented in Lemma \ref{outwholeestimate}. 
For general case, choose a sequence of functions $f_n\in C_c^\infty(\bR^{d+1})$ such that $\cG f_n \to \cG f \ (a.e.)$, and $\|f_n\|_{L_\infty(\bR^{d+1})}\leq \|f\|_{L_\infty(\bR^{d+1})}$. Then by Fatou's lemma, and Lemma \ref{outwholeestimate}, we obtain
\begin{align*}
\aint_{Q_b} |\cG f(t,\vec{x})-(\cG f)_{Q_b}| \mathrm{d}t \mathrm{d}\vec{x}& \leq \aint_{Q_b} \aint_{Q_b} |\cG f(t,\vec{x})-\cG f(s,\vec{y})| \mathrm{d}t \mathrm{d}\vec{x} \mathrm{d}s \mathrm{d}\vec{y}
\\
& \leq  \liminf_{n\to \infty} \aint_{Q_b} \aint_{Q_b} |\cG f_n(t,\vec{x})-\cG f_n(s,\vec{y})| \mathrm{d}t \mathrm{d}\vec{x} \mathrm{d}s \mathrm{d}\vec{y}
\\
& \leq C  \liminf_{n\to \infty} \|f_n\|_{L_\infty(\bR^{d+1})}
\leq C \|f\|_{L_\infty(\bR^{d+1})}.
\end{align*}
This completes the proof.
\end{proof}

The final step in this section is to establish \eqref{meanaverageineq} for $f\in C_c^{\infty}(\mathbb{R}^{d+1})$.
To achieve this, it suffices to show that
\begin{align*}
\aint_{Q_b}\aint_{Q_b}|\cG f (t,\vec{x})-\cG f(s,\vec{y})| \mathrm{d}\vec{y} \mathrm{d}s  \mathrm{d}\vec{x} \mathrm{d}t    \leq C \|f\|_{L_\infty(\bR^{d+1})}.
\end{align*}
Once this is established, \eqref{meanaverageineq} follows immediately from the inequality
$$
\aint_{Q_b} |\cG f(t,\vec{x})-(\cG f)_{Q_b}|  \mathrm{d}\vec{x} \mathrm{d}t\leq \aint_{Q_b}\aint_{Q_b}|\cG f (t,\vec{x})-\cG f(s,\vec{y})| \mathrm{d}\vec{y} \mathrm{d}s\mathrm{d}\vec{x} \mathrm{d}t .
$$
We now state the key lemma that completes the proof of \eqref{meanaverageineq}.
\begin{lemma}
\label{outwholeestimate}
Let $f\in C_c^\infty(\bR^{d+1})$ and $b>0$. Then we have
\begin{align*}
\aint_{Q_b}\aint_{Q_b}|\cG f (t,\vec{x})-\cG f(s,\vec{y})| \mathrm{d}\vec{y} \mathrm{d}s  \mathrm{d}\vec{x} \mathrm{d}t    \leq C \|f\|_{L_\infty(\bR^{d+1})},
\end{align*}
where $C$ depends only on $\alpha,d, c_{0},\delta_{0},\ell$.
\end{lemma}
\begin{proof}
Take functions $\eta=\eta(t) \in C^\infty(\bR)$ and $\zeta = \zeta(\vec{x}) \in C_c^\infty(\bR^{d})$ satisfying
\begin{itemize}
    \item $0\leq \eta\leq 1$, $\eta=1$ on $(-\infty, -8b/3)$ and $\eta(t)=0$ for $t\geq -7b/3$.
    \item $0\leq \zeta\leq 1$, $\zeta=1$ on $\mathbb{B}_{7\kappa(b)/3}$ and $\zeta=0$ outside of $\mathbb{B}_{8\kappa(b)/3}$.
\end{itemize}
Then using $\eta$ and $\zeta$, we split the integrand as follows (exploit the linearity of $\mathcal{G}$);
\begin{equation*}
\begin{aligned}
|\mathcal{G} f(t,\vec{x})-\mathcal{G} f(s,\vec{y})|&\leq |\mathcal{G} f_1(t,\vec{x})-\mathcal{G} f_1(s,\vec{y})|+|\mathcal{G} f_2(t,\vec{x})-\mathcal{G} f_2(s,\vec{x})|\\
&\quad+ |\mathcal{G} f_3(s,\vec{x}) - \mathcal{G} f_3(s,\vec{y})| + |\mathcal{G} f_4(s,\vec{x}) - \mathcal{G} f_4(s,\vec{y})| 
\\
&=:G_1(t,s,\vec{x},\vec{y})+G_2(t,s,\vec{x},\vec{y})+G_3(t,s,\vec{x},\vec{y})+G_4(t,s,\vec{x},\vec{y}),
\end{aligned}
\end{equation*}
where
\begin{itemize}
    \item $f_1:=f(1-\eta)$; $f_1$ is supported in $(-3b,\infty)\times\bR^{\vec{d}}$.
    \item $f_2:=f\eta$; $f_2$ is supported in $(-\infty,-2b)\times\bR^{\vec{d}}$.
    \item $f_3:=f\eta(1-\zeta)$; $f_3$ is supported in $(-\infty,-2b)\times (\mathbb{B}_{\kappa(b)})^{c}$.
    \item $f_4:=f\eta\zeta$; $f_4$ is supported in
    $(-\infty,-2b)\times \mathbb{B}_{2\kappa(b)}$.
\end{itemize}
Therefore, it is enough to show
$$
\aint_{Q_b}\aint_{Q_b}(G_1+G_2+G_3+G_4)(t,s,\vec{x},\vec{y})\mathrm{d}t \mathrm{d}\vec{x} \mathrm{d}s \mathrm{d}\vec{y}\leq C\|f\|_{L_{\infty}(\bR^{d+1})}.
$$

\textbf{Step 1.} In Step 1, we prove
\begin{align}\label{eqn 06.25.14:02}
\aint_{Q_b}\aint_{Q_b} G_1(t,s,\vec{x},\vec{y}) \mathrm{d}\vec{x} \mathrm{d}t \vec{y} \mathrm{d}s:=\aint_{Q_b}\aint_{Q_b}|\cG f_1 (t,\vec{x})-\cG f_1 (s,\vec{y})| \mathrm{d}\vec{x} \mathrm{d}t \vec{y} \mathrm{d}s\leq C \|f\|_{L_\infty(\bR^{d+1})}.
\end{align}
Recall that $f_1$ is supported in $(-3b,\infty)\times\bR^{\vec{d}}$. To show \eqref{eqn 06.25.14:02} we prove 
\begin{align}
\label{23.03.10.13.13}
    \aint_{Q_b}|\cG f_1(t,\vec{x})|\mathrm{d}\vec{x} \mathrm{d}t\leq C\|f\|_{L_{\infty}(\bR^{d+1})},
\end{align}
which certainly implies \eqref{eqn 06.25.14:02}. We divide the proof of \eqref{23.03.10.13.13} into two steps.
\\
\textbf{Step 1-1.} The support of $f_1$ is contained in $(-3b,3b)\times \mathbb{B}_{3\kappa(b)}$.
\\
By the assumption and \eqref{eqn 07.15.19:01},
$$
\|f_1\|_{L_2(\bR^{d+1})}\leq C |Q_{b}|^{1/2}\|f\|_{L_{\infty}(\bR^{d+1})}.
$$
Thus,  by H\"older's inequality and \eqref{eqn 06.24.15:27},
\begin{align*}
\aint_{Q_b}|\cG f_1 (t,\vec{x})|\mathrm{d}\vec{x}\mathrm{d}t \leq & \left(\int_{Q_b}|\cG f_1 (t,\vec{x})|^2 \mathrm{d}\vec{x} \mathrm{d}t\right)^{1/2}|Q_b|^{-1/2}
\leq  C \|f\|_{L_{\infty}(\bR^{d+1})}.
\end{align*}
\\
\textbf{Step 1-2.} General case.

Take $\zeta_0=\zeta_0(t)\in C^{\infty}(\bR)$ such that $0\leq \zeta_0\leq 1$, $\zeta_0(t)=1$ for $t\leq 2b$, and $\zeta_0(t)=0$ for $t\geq 5b/2$.
Note that $\cG f_{1}=\cG (f_{1}\zeta_0)$ on $Q_b$ and $|f_{1}\zeta_0|\leq |f_{1}|$. Hence, replacing $f_{1}$ by $f_{1}\zeta_{0}$ in \eqref{23.03.10.13.13}, we may assume that $f_{1}(t,\vec{x})=0$ if $|t|\geq 3b$.

Recall that $\zeta=\zeta(\vec{x}) \in C_c^\infty (\bR^d)$ is the function satisfying that $\zeta=1$ in $\mathbb{B}_{7\kappa(b)/3}$ and $\zeta=0$ outside of $\mathbb{B}_{8\kappa(b)/3}$ and $0\leq\zeta\leq1$.
Set $f_{1,1}=\zeta f_{1}$ and $f_{1,2}=(1-\zeta)f_{1}$. Then $\cG f_{1} = \cG f_{1,1} + \cG f_{1,2}$. Since $\cG f_{1,1}$ can be estimated by Step 1-1, we may further assume that $f_{1}(t,\vec{x})=0$ if $\vec{x}\in \mathbb{B}_{2\kappa(b)}$.  Therefore, for any $\vec{x}\in \mathbb{B}_{\kappa(b)}$,
\begin{align*}
\int_{\bR^{\vec{d}}} \left|q_{\alpha,\alpha+1} (t-s,\vec{x}-\vec{y}) f_{1}(s,\vec{y})\right| \mathrm{d}\vec{y}
&= \int_{(\mathbb{B}_{2\kappa(b)})^{c}} |q_{\alpha,\alpha+1}(t-s,\vec{x}-\vec{y})f_{1}(s,\vec{y})| \mathrm{d}\vec{y} 
\\
&\leq \sum_{i=1}^{\ell} \int_{(B^{i}_{2\kappa_{i}(b)})^{c}\times \bR^{d-d_{i}}} |q_{\alpha,\alpha+1}(t-s,\vec{x}-\vec{y})f_{1}(s,\vec{y})| \mathrm{d}\vec{y}
\\
&\leq  \|f\|_{L_{\infty}(\bR^{d+1})} \sum_{i=1}^{\ell}  \int_{(B^{i}_{\kappa_{i}(b)})^{c}\times \bR^{d-d_{i}}} |q_{\alpha,\alpha+1}(t-s,\vec{y})| \mathrm{d}\vec{y}
\\
&:=  \|f\|_{L_{\infty}(\bR^{d+1})}\sum_{i=1}^{\ell} G_{1,i}.
\end{align*}
By Corollary \ref{lem 02.17.22:32} and \eqref{int phi},
\begin{align*}
G_{1,i}
&\leq C{\bf1}_{|s|\leq 3b} |t-s|^{\frac{\alpha}{2}-1}  \int_{(\phi_{i}^{-1}(b^{-\alpha}))^{-1/2}}^{\infty} \frac{(\phi_{i}(\rho^{-2}))^{1/2}}{\rho^{d_{i}}}\rho^{d_{i}-1} \mathrm{d}\rho
\\
&\leq C {\bf1}_{|s|\leq 3b} |t-s|^{\frac{\alpha}{2}-1}  b^{-\alpha/2}.
\end{align*}
Note that if $|t| \leq b$ and $|s|\leq 3b$ then $|t-s|\leq 4b$.  Hence, it follows that for any $(t,\vec{x})\in Q_{b}$,  
\begin{align*}
|\cG f(t,\vec{x})| & \leq \|f\|_{L_{\infty}(\bR^{d+1})} \int_{-\infty}^{t} \sum_{i=1}^{\ell}  G_{1,i} \mathrm{d}s  \leq C\|f\|_{L_\infty(\bR^{d+1})} b^{-\alpha/2} \int_{|t-s|\leq 4b} |t-s|^{-1+\alpha/2} \mathrm{d}s\leq C \|f\|_{L_{\infty}(\R^{d+1})}.
\end{align*}
By taking the average over $Q_{b}$ on both sides, we have \eqref{23.03.10.13.13}.

\textbf{Step 2.} In Step 2, we prove
\begin{align*}
\aint_{Q_b}\aint_{Q_b} G_2(t,s,\vec{x},\vec{y}) \mathrm{d}\vec{x} \mathrm{d}t \vec{y} \mathrm{d}s:=\aint_{Q_b}\aint_{Q_b}|\cG f_2 (t,\vec{x})-\cG f_2 (s,\vec{x})| \mathrm{d}\vec{x} \mathrm{d}t \vec{y} \mathrm{d}s\leq C \|f\|_{L_\infty(\bR^{d+1})}.
\end{align*}
Recall that $f_2$ is supported in $(-\infty,-2b)\times\bR^{\vec{d}}$. If we show that  
\begin{align}
\label{23.03.10.01.27}
|\cG f_2 (t_1,\vec{x})-\cG f_2(t_2,\vec{x})|\leq C \|f\|_{L_\infty(\bR^{d+1})} \quad \forall\, (t_1,\vec{x}), (t_2,\vec{x})\in Q_b,
\end{align}
then by taking the average over $Q_{b}$ on both sides, we have the desired result. Thus we only prove \eqref{23.03.10.01.27}. Also, due to the symmetry of the left-hand side of \eqref{23.03.10.01.27}, we may assume $t_{1}>t_{2}$. Then, since $f_2(s,\vec{x})=0$ for $s\geq -2b$ and $t_1, t_2\geq -b$, using this and the fundamental theorem of calculus, it follows that
\begin{align*}
&|\cG f_2(t_1,\vec{x})-\cG f_2(t_2,\vec{x})|\\
&=\Big|\int_{-\infty}^{t_1}  \int_{\bR^{\vec{d}}} q_{\alpha,\alpha+1}(t_{1}-s,\vec{y})f(s,\vec{x}-\vec{y}) \mathrm{d}\vec{y} \mathrm{d}s  - \int_{-\infty}^{t_2}  \int_{\bR^{\vec{d}}} q_{\alpha,\alpha+1}(t_{2}-s,\vec{y})f(s,\vec{x}-\vec{y})  \mathrm{d}\vec{y} \mathrm{d}s \Big|
\\
&= \left| \int_{-\infty}^{-2b} \int_{\bR^{d}} \int^{t_{1}}_{t_{2}} q_{\alpha,\alpha+2}(t-s,\vec{x}-\vec{y}) f(s,\vec{y}) \mathrm{d}t \mathrm{d}\vec{y} \mathrm{d}s \right|.
\end{align*}
By Lemma \ref{lem 06.08.14:52}, and Fubini's theorem,
\begin{align*}
|\cG f_2(t_1,\vec{x})-\cG f_2(t_2,\vec{x})|
& \leq C \|f\|_{L_\infty(\bR^{d+1})}  \int_{-\infty}^{-2b} 
 \int_{t_2}^{t_1} (t-s)^{-2}\mathrm{d}t \mathrm{d}s.
\end{align*}
Therefore,  for  $-b\leq t_2<t_1 \leq b$, 
\begin{align*}
|\cG f_2(t_1,\vec{x})-\cG f_2(t_2,\vec{x})|
&\leq C \|f\|_{L_\infty(\bR^{d+1})} \left(\int_{t_2}^{t_1}\int_{-\infty}^{-2b} (t-s)^{-2}\mathrm{d}s \mathrm{d}t \right)
\\
&\leq C \|f\|_{L_\infty(\bR^{d+1})} \left(\int_{t_2}^{t_1} b^{-1} \mathrm{d}t \right)
\leq C \|f\|_{L_\infty(\bR^{d+1})}.
\end{align*}
This certainly proves \eqref{23.03.10.01.27}.

\textbf{Step 3.} In Step 3, we prove
\begin{align*}
\aint_{Q_b}\aint_{Q_b} G_4(t,s,\vec{x},\vec{y}) \mathrm{d}\vec{x} \mathrm{d}t \vec{y} \mathrm{d}s:=\aint_{Q_b}\aint_{Q_b}|\cG f_4 (s,\vec{x})-\cG f_4 (s,\vec{y})| \mathrm{d}\vec{x} \mathrm{d}t \vec{y} \mathrm{d}s\leq C \|f\|_{L_\infty(\bR^{d+1})}.
\end{align*}
Recall that $f_4$ is supported in $(-\infty,-2b)\times \mathbb{B}_{3\kappa(b)}$. Like Step 2, it suffices to show
\begin{align}
    \label{23.03.10.01.56}
|\cG f_4(t,\vec{x})|\leq C\|f\|_{L_{\infty}(\bR^{d+1})} \quad \forall\, (t,\vec{x}) \in Q_{b}.
\end{align}
For $(t,\vec{x})\in Q_b$,
\begin{align*}\label{eqn 01.04.14:27}
|\cG f_4(t,\vec{x})| \leq&\int_{-\infty}^{-2b} \int_{\mathbb{B}_{3\kappa(b)}} |q_{\alpha,\alpha+1} (t-s,\vec{x}-\vec{y}) f(s,\vec{y})|\mathrm{d}\vec{y}\mathrm{d}s  \nonumber
\\
\leq&  \|f\|_{L_{\infty}(\bR^{d+1})} \int_{-\infty}^{-2b} \int_{\mathbb{B}_{3\kappa(b)}} |q_{\alpha,\alpha+1} (t-s,\vec{x}-\vec{y})|\mathrm{d}\vec{y} \mathrm{d}s   \nonumber
\\
\leq&  \|f\|_{L_{\infty}(\bR^{d+1})} \int_{b}^{\infty} \int_{\mathbb{B}_{4\kappa (b)}} |q_{\alpha,\alpha+1} (s,\vec{y})|\mathrm{d}\vec{y}\mathrm{d}s   = \|f\|_{L_{\infty}(\bR^{d+1})} \left( G_{4,1} +G_{4,2} \right),
\end{align*}
where
$$
G_{4,1}=\int^{16b}_{b} \int_{\mathbb{B}_{4\kappa(b)}} |q_{\alpha,\alpha+1} (s,\vec{y})|\mathrm{d}\vec{y} \mathrm{d}s, \quad  
G_{4,2}=\int_{16b}^{\infty} \int_{\mathbb{B}_{4\kappa(b)}} |q_{\alpha,\alpha+1} (s,\vec{y})|\mathrm{d}\vec{y}\mathrm{d}s.
$$
Using Lemma \ref{lem 06.08.14:52}, we have 
\begin{align*}
G_{4,1}\leq C \int^{4b}_{b}  s^{-1}\mathrm{d}s =C.
\end{align*}
For $G_{4,2}$, observe that
$$
G_{4,2} \leq \int_{16b}^{\infty} \int_{B^{1}_{4\kappa_{1}(b)}} \left( \int_{\bR^{d-d_{1}}} \left| q_{\alpha,\alpha+1}(s,\vec{y}) \right| \mathrm{d}y_{2}\dots \mathrm{d}y_{\ell} \right) \mathrm{d}y_{1} \mathrm{d}s.
$$
For $s \geq 16b$, by Fubini's theorem and Corollary \ref{lem 02.17.22:32} (ii),
\begin{align*}
&  \int_{B^{1}_{4\kappa_{1}(b)}} \left( \int_{\bR^{d-d_{1}}} \left| q_{\alpha,\alpha+1}(s,\vec{y}) \right| \mathrm{d}y_{2}\dots \mathrm{d}y_{\ell} \right)  \mathrm{d}y_{1}
\\
& \leq  C \sum_{k=1}^{\ell}  \int_{B^{1}_{4\kappa_{1}(b)}}  s^{-1-\alpha/k} \int_{(\phi_{1}(|y_{1}|^{-2}))^{-1/k}}^{2s^{\alpha/k}} (\phi_{1}^{-1}(r^{-k}))^{d_{1}/2} \mathrm{d}r  \mathrm{d}y_{1}
\\
& \leq C \sum_{k=1}^{\ell} \int_{B^{1}_{4\kappa_{1}(b)}} s^{-1-\alpha/k} \int_{(\phi_{1}(|y_{1}|^{-2}))^{-1/k}}^{(16b)^{\alpha/k}} (\phi_{1}^{-1}(r^{-k}))^{d_{1}/2} \mathrm{d}r  \mathrm{d}y_{1}
\\
&\quad + C \sum_{k=1}^{\ell} \int_{B^{1}_{4\kappa_{1}(b)}} s^{-1-\alpha/k} \int_{(16b)^{\alpha/k}}^{2s^{\alpha/k}} (\phi_{1}^{-1}(r^{-k}))^{d_{1}/2} \mathrm{d}r  \mathrm{d}y_{1}
\\
&\leq C \sum_{k=1}^{\ell} \int_{0}^{(16b)^{\alpha/k}} \int_{|y_{1}| \leq \left( \phi_{1}(r^{-k}) \right)^{-1/2}} \left( \phi^{-1}_{1}(r^{-k}) \right)^{d_{1}/2} s^{-1-\alpha/k} \mathrm{d}y_{1}\mathrm{d}r
\\
& \quad + C \sum_{k=1}^{\ell} \int_{B^{1}_{4\kappa_{1}(b)}} \int_{(16b)^{\alpha/k}}^{2s^{\alpha/k}} s^{-1-\alpha/k} (\phi^{-1}_{1}(r^{-k}))^{d_{1}/2} \mathrm{d}r \mathrm{d}y_{1}
\\
&\leq C \sum_{k=1}^{\ell} b^{\alpha/k} s^{-1-\alpha/k} 
 +  C \sum_{k=1}^{\ell} \int_{B^{1}_{4\kappa_{1}(b)}} \int_{(16b)^{\alpha/k}}^{2s^{\alpha/k}} s^{-1-\alpha/k} (\phi^{-1}_{1}(r^{-k}))^{d_{1}/2} \mathrm{d}r \mathrm{d}y_{1}.
\end{align*}
Since
$\sum_{k=1}^{\ell} b^{\alpha/k} \int_{16b}^{\infty} s^{-1-\alpha/k} \mathrm{d}s = C$ which independent of $b$,
it only remains to consider
\begin{align*}
&\sum_{k=1}^{\ell} \int_{16b}^{\infty}  \int_{B_{4\kappa_{1}(b)}} \int_{(16b)^{\alpha/k}}^{2s^{\alpha/k}} s^{-1-\alpha/k} \left( \phi^{-1}(r^{-k}) \right)^{d_{1}/2} \mathrm{d}r \mathrm{d}y_{1} \mathrm{d}s
\\
& =\sum_{k=1}^{\ell} \int_{B_{4\kappa_{1}(b)}} \int_{(16b)^{\alpha/k}}^{\infty} \int_{(r/2)^{k/\alpha}}^{\infty} s^{-1-\alpha/k} \left( \phi^{-1}(r^{-k}) \right)^{d_{1}/2}  \mathrm{d}s\mathrm{d}r \mathrm{d}y_{1} 
\\
&= C \sum_{k=1}^{\ell} \int_{B_{4\kappa_{1}(b)}} \int_{(16b)^{\alpha/k}}^{\infty} \int_{(r/2)^{k/\alpha}}^{\infty} s^{-1-\alpha/k} \left( \phi^{-1}(r^{-k}) \right)^{d_{1}/2}  \mathrm{d}s\mathrm{d}r \mathrm{d}y_{1} 
\\
&=\sum_{k=1}^{\ell} \int_{B_{4\kappa_{1}(b)}} \int_{(16b)^{\alpha/k}}^{\infty} r^{-1} \left( \phi^{-1}(r^{-k}) \right)^{d_{1}/2}  \mathrm{d}s\mathrm{d}r \mathrm{d}y_{1}.
\end{align*}
Using \eqref{eqn 07.15.14:49}, we check that
\begin{align*}
& \sum_{k=1}^{\ell} \int_{B_{4\kappa_{1}(b)}} \int_{(16b)^{\alpha/k}}^{\infty} r^{-1} \left( \phi^{-1}(r^{-k}) \right)^{d_{1}/2}  \mathrm{d}s\mathrm{d}r \mathrm{d}y_{1}
\\
&\leq C \sum_{k=1}^{\ell} \int_{B^{1}_{4\kappa_{1}(b)}} (16b)^{\alpha d_{1}/2} (\kappa_{1}(16b))^{-d_{1}} \int_{(16b)^{\alpha/k}}^{\infty} r^{-1-\frac{k d_{1}}{2}} \mathrm{d}r \mathrm{d}y_{1}
\\
&\leq C \sum_{k=1}^{\ell} b^{\alpha  d_{1}/2} (\kappa_{1}(b))^{d_{1}} (\kappa_{1}(16b))^{-d_{1}} b^{-\alpha  d_{1}/2}
\leq C.
\end{align*}
Hence, we have \eqref{23.03.10.01.56}.

\textbf{Step 4.}
In Step 4, we prove
\begin{align*}
\aint_{Q_b}\aint_{Q_b} G_3(t,s,\vec{x},\vec{y}) \mathrm{d}\vec{x} \mathrm{d}t \vec{y} \mathrm{d}s:=\aint_{Q_b}\aint_{Q_b}|\cG f_3 (s,\vec{x})-\cG f_3 (s,\vec{y})| \mathrm{d}\vec{x} \mathrm{d}t \vec{y} \mathrm{d}s\leq C \|f\|_{L_\infty(\bR^{d+1})}.
\end{align*}
Recall that $f_3$ is supported in $(-\infty,-2b)\times (\mathbb{B}_{3\kappa(b)})^c$. It suffices to prove
\begin{equation}\label{eqn 06.25.16:56}
|\cG f_3(t,\vec{x})-\cG f_3(t,\vec{z})|\leq C\|f\|_{L_{\infty}(\bR^{d+1})} \quad \forall\,  (t,\vec{x}),(t,\vec{z})\in Q_b.
\end{equation}
Since $f_3(s,\vec{y})=0$ if $s\geq -2b$ or $\vec{y} \in \mathbb{B}_{2\kappa(b)}$, we see that for $t>-b$,  
\begin{align*}
|\cG f_{3}(t,\vec{x})-\cG f_{3}(t,\vec{z})|
= \left |\int_{-\infty}^{-2b}  \int_{(\mathbb{B}_{2\kappa(b)})^{c}} \left( q_{\alpha,\alpha+1}(t-s,\vec{x}-\vec{y})- q_{\alpha,\alpha+1}(t-s,\vec{z}-\vec{y}) \right) f(s,\vec{y}) \mathrm{d}\vec{y} \mathrm{d}s \right|.
\end{align*}
By the fundamental theorem of calculus, we have
\begin{align*}
&\left|\mathcal{G}f_{3}(t,\vec{x}) - \mathcal{G}f_{3} (t,\vec{z})  \right| 
\\
&\leq \|f\|_{L_{\infty}(\bR^{d+1})} \sum_{i=1}^{\ell} \int_{-\infty}^{-2b} \int_{(\mathbb{B}_{2\kappa(b)})^{c}} \int_{0}^{1} \left| (\nabla_{x_{i}} q_{\alpha,\alpha+1})(t-s,\vec{\theta}(\vec{x},\vec{z},\vec{y},u)) \cdot (x_{i}-z_{i}) \right|  \mathrm{d}u \mathrm{d}\vec{y} \mathrm{d}s
\\
&\leq \|f\|_{L_{\infty}(\bR^{d+1})} \sum_{i=1}^{\ell} \int_{-\infty}^{-2b} \int_{(B^{i}_{2\kappa_{i}(b)})^{c}\times \bR^{d-d_{i}}} \int_{0}^{1} \left| (\nabla_{x_{i}} q_{\alpha,\alpha+1})(t-s,\vec{\theta}(\vec{x},\vec{z},\vec{y},u)) \cdot (x_{i}-z_{i}) \right|  \mathrm{d}u \mathrm{d}\vec{y} \mathrm{d}s
\\
&:= \|f\|_{L_{\infty}(\bR^{d+1})} \sum_{i=1}^{\ell} G_{3,i}.
\end{align*}
where $\vec{\theta}(\vec{x},\vec{z},\vec{y},u)=(1-u)\vec{z} + u \vec{x} - \vec{y}$.  
By Fubini's theorem, and change of variables $\vec{\theta}(\vec{x},\vec{z},\vec{y},u) \to \vec{y}$,  we have
\begin{align*}
& \int_{(B^{i}_{2\kappa_{i}(b)})^{c}\times \bR^{d-d_{i}}} \int_{0}^{1} \left| (\nabla_{x_{i}} q_{\alpha,\alpha+1})(t-s,\vec{\theta}(\vec{x},\vec{z},\vec{y},u)) \cdot (x_{i}-z_{i}) \right|  \mathrm{d}u \mathrm{d}\vec{y}
\\
& =  \int_{0}^{1}  \int_{(B^{i}_{2\kappa_{i}(b)})^{c}\times \bR^{d-d_{i}}} \left| (\nabla_{x_{i}} q_{\alpha,\alpha+1})(t-s,\vec{\theta}(x_{i},z_{i},u)-y_{i}) \cdot (x_{i}-z_{i}) \right| \mathrm{d}\vec{y} \mathrm{d}u
\\
& \leq    \int_{(B^{i}_{\kappa_{i}(b)})^{c}\times \bR^{d-d_{i}}} \left| (\nabla_{x_{i}} q_{\alpha,\alpha+1})(t-s,\vec{y}) \cdot (x_{i}-z_{i}) \right| \mathrm{d}\vec{y}.
\end{align*}
Therefore, for each $i=1,\dots,\ell$
\begin{align*}
G_{3,i} &\leq \int_{-\infty}^{-2b}  \int_{(B^{i}_{\kappa_{i}(b)})^{c}\times \bR^{d-d_{i}}} \left| (\nabla_{x_{i}} q_{\alpha,\alpha+1})(t-s,\vec{y}) \cdot (x_{i}-z_{i}) \right| \mathrm{d}\vec{y} \mathrm{d}s  \nonumber
\\
& \leq C \kappa_{i}(b) \|f\|_{L_\infty(\bR^{d+1})}  \int_{-\infty}^{ -2b} \int_{(B^{i}_{\kappa_{i}(b)})^{c}\times \bR^{d-d_{i}}} |\nabla_{x_{i}} q_{\alpha,\alpha+1} (t-s,\vec{y})| \mathrm{d}\vec{y} \mathrm{d}s  \nonumber
\\
& \leq C \kappa_{i}(b)   \int_{b}^{\infty} \int_{(B^{i}_{\kappa_{i}(b)})^{c}\times \bR^{d-d_{i}}} |\nabla_{x_{i}} q_{\alpha,\alpha+1} (s,\vec{y})| \mathrm{d}\vec{y} \mathrm{d}s.
\end{align*}
By Corollary \ref{lem 02.17.22:32},
\begin{align*}
& \int_{b}^{\infty} \int_{(B^{i}_{\kappa_{i}(b)})^{c}\times \bR^{d-d_{i}}} |\nabla_{x_{i}} q_{\alpha,\alpha+1} (s,\vec{y})| \mathrm{d}\vec{y} \mathrm{d}s
\\
& \leq C \int_{b}^\infty \int_{\left(\phi_{i}^{-1}(s^{-\alpha})\right)^{-1/2}}^\infty s^{\frac{\alpha}{2}-1} \frac{(\phi_{i}(\rho^{-2}))^{1/2}}{\rho^2} \mathrm{d}\rho \mathrm{d}s 
\\
&\quad + C  \sum_{k=1}^{\ell} \int_{b}^\infty \int_{\left(\phi_{i}^{-1}(b^{-\alpha})\right)^{-1/2}}^{\left(\phi_{i}^{-1}(s^{-\alpha})\right)^{-1/2}} \int^{2s^{\alpha/k}}_{(\phi_{i}(\rho^{-2}))^{-1/k}} \rho^{d_{i}-1} s^{-1-\frac{\alpha}{k}} (\phi_{i}^{-1}(r^{-k}))^{(d_{i}+1)/2 } \mathrm{d} r \mathrm{d}\rho \mathrm{d}s.
\end{align*}
We now estimate the last two integrals above.  First, by \eqref{int phi},
\begin{align*}
\int_{b}^\infty \int_{\left(\phi_{i}^{-1}(s^{-\alpha})\right)^{-1/2}}^\infty s^{\frac{\alpha}{2}-1} \frac{(\phi_{i}(\rho^{-2}))^{1/2}}{\rho^2} \mathrm{d}\rho \mathrm{d}s 
&\leq \int_{b}^\infty s^{\frac{\alpha}{2}-1} \left(\phi_{i}^{-1}(s^{-\alpha})\right)^{1/2} \int_{\left(\phi_{i}^{-1}(s^{-\alpha})\right)^{-1/2}}^\infty  \frac{(\phi_{i}(\rho^{-2}))^{1/2}}{\rho} \mathrm{d}\rho  \mathrm{d}s \nonumber
\\
&\leq C\int_{b}^\infty \left(\phi_{i}^{-1}(s^{-\alpha})\right)^{1/2} s^{-1} \mathrm{d}s. 
\end{align*}
Second, for each $k=1,\dots \ell$, by Fubini's theorem, it is easy to see that
\begin{align*}
& \int_{0}^{b} \int_{\left(\phi_{i}^{-1}(b^{-\alpha})\right)^{-1/2}}^{\left(\phi_{i}^{-1}(s^{-\alpha})\right)^{-1/2}} \int^{2s^{\alpha/k}}_{(\phi_{i}(\rho^{-2}))^{-1/k}} \rho^{d_{i}-1} s^{-1-\frac{\alpha}{k}} (\phi_{i}^{-1}(r^{-k}))^{(d_{i}+1)/2 } \, \mathrm{d} r \mathrm{d}\rho \mathrm{d}s  \nonumber
\\
&\leq \int_{0}^{b}   \int^{2s^{\alpha/k}}_{b^{\alpha/k}}  \int_{0}^{\left(\phi_{i}^{-1}(r^{-k})\right)^{-1/2}}  \rho^{d_{i}-1} s^{-1-\frac{\alpha}{k}} (\phi_{i}^{-1}(r^{-k}))^{(d_{i}+1)/2 } \, \mathrm{d}\rho \mathrm{d} r  \mathrm{d}s  \nonumber
\\
&\leq C \int_{0}^{b}   \int^{2s^{\alpha/k}}_{b^{\alpha/k}}   s^{-1-\frac{\alpha}{k}} (\phi_{i}^{-1}(r^{-k}))^{1/2 } \, \mathrm{d} r  \mathrm{d}s   \nonumber
\\
&\leq C \int_{b^{\alpha/k}}^{\infty} \int_{(r/2)^{k/\alpha}} s^{-1-\frac{\alpha}{k}} (\phi_{i}^{-1}(r^{-k}))^{1/2 } \, \mathrm{d}s \mathrm{d} r     \nonumber
\\
&\leq C \int_{b^{\alpha/k}}^{\infty}  r^{-1} (\phi_{i}^{-1}(r^{-k}))^{1/2} \, \mathrm{d} r  \mathrm{d}s 
\end{align*}
Using \eqref{eqn 07.15.14:49}, we have
\begin{align*} 
&\int_{b}^\infty \left(\phi_{1}^{-1}(s^{-\alpha})\right)^{1/2} s^{-1} \mathrm{d}s +\sum_{k=1}^{\ell} \int_{b^{\alpha/k}}^{\infty} r^{-1} (\phi_{i}^{-1}(r^{-k}))^{1/2} \mathrm{d}r  \nonumber
\\
& \leq  C  \left(\phi_{i}^{-1}(b^{-\alpha})\right)^{1/2}b^{\alpha/2} \left( \int_{b}^{\infty} s^{-1-\frac{\alpha}{2}} \mathrm{d}s + \sum_{k=1}^{\ell} \int_{b^{\alpha/k}}^{\infty} r^{-1-\frac{k}{2}} \mathrm{d}r \right) \leq C (\kappa_{i}(b))^{-1}. 
\end{align*}
Therefore, we have $G_{3,i} \leq C$ for $i=1, \dots, \ell$, and thus \eqref{eqn 06.25.16:56} follows.
The lemma is proved.
\end{proof}

We conclude this section with the proof of Theorem \ref{thm 07.22.11:31}.
\begin{proof}[Proof of Theorem \ref{thm 07.22.11:31}]
The first part of the proof is based on the Fefferman-Stein theorem (see \textit{e.g.} \cite[Theorem I.3.1., Theorem IV.2.2.]{stein1993harmonic}) and the Marcinkiewicz interpolation theorem (see \textit{e.g.} \cite[Theorem 1.3.2.]{grafakos2014classical}). To use these theorems, we remark that the cubes $Q_b(s,y)$ satisfy the conditions (i)-(iv) in  \cite[Section 1.1]{stein1993harmonic} (recall \eqref{eqn 07.15.19:01}), and that the map $f\mapsto \cG f$ is sublinear.

\textbf{Step 1.} Proof of \eqref{eqn 07.02.14:56} when $p=q$. 
\\
First, assume that $p\geq 2$. Then using \eqref{eqn 06.24.15:27} and then the Fefferman-Stein theorem, for any $f\in L_2(\bR^{d+1})\cap L_\infty(\bR^{d+1})$, we have
\begin{equation*}
\|(\cG f)^{\#}\|_{L_2(\bR^{d+1})}\leq C \|f\|_{L_2(\bR^{d+1})}.
\end{equation*}
Due to \eqref{bmoestimate}, we also have
\begin{equation*}
\|(\cG f)^{\#}\|_{L_\infty(\bR^{d+1})}\leq C \|f\|_{L_\infty(\bR^{d+1})}.
\end{equation*}
Using these estimates and the Marcinkiewicz interpolation theorem, for any $p\in [2,\infty)$ we have
\begin{equation*}
\|(\cG f)^{\#}\|_{L_p(\bR^{d+1})}\leq C \|f\|_{L_p(\bR^{d+1})}
\end{equation*}
for all $f\in L_2(\bR^{d+1})\cap L_\infty(\bR^{d+1})$. Using the Fefferman-Stein theorem again, we get
\begin{equation}\label{eqn 07.02.15:15}
\|\cG f\|_{L_p(\bR^{d+1})}\leq C \|f\|_{L_p(\bR^{d+1})}
\end{equation}
for $p\in[2,\infty)$. For $p\in (1,2)$ one can prove \eqref{eqn 07.02.15:15} using the standard duality argument.

\textbf{Step 2.} Proof of \eqref{eqn 07.02.14:56} for general $p,q\in(1,\infty)$.  
\\
Extend $q_{\alpha,\alpha+1}(t,\cdot):=0$ for $t\leq 0$. For each $(t,s)\in\bR^2$, we define the operator $G_{t,s}$ as follows:
\begin{equation*}
G_{t,s}f(\vec{x}):=\int_{\bR^d} q_{\alpha,\alpha+1}(t-s,\vec{x}-\vec{y})  f(\vec{y}) \, \mathrm{d}\vec{y}, \quad f\in C_c^\infty(\R^{d}).
\end{equation*}
Let $p\in(1,\infty)$. Then, by Lemma \ref{lem 06.08.14:52}, we have
\begin{align*}
\|G_{t,s}f\|_{L_p(\bR^{d})}
\leq \|f\|_{L_p(\bR^{d})} \int_{\bR^d} |q_{\alpha,\alpha+1}(t-s,\vec{x}-\vec{y})|\mathrm{d}y  \leq C(t-s)^{-1}\|f\|_{L_p(\bR^{d})}.
\end{align*}
Hence, the operator $G_{t,s}$ is uniquely extendible to $L_p(\R^{d})$ for $t\neq s$. Denote
\begin{equation*}
Q:=[t_0,t_0+\delta), \quad Q^*:=[t_0-\delta,t_0+2\delta), \quad  \delta>0.
\end{equation*}
Then for $t\notin Q^*$ and $s_1,s_2\in Q$, we can easily see that
$$
|s_1-s_2|\leq\delta, \quad |t-(t_0+\delta)|\geq\delta.
$$
Also for such $t,s_{1},s_{2}$, and for any $f\in L_p$ such that $\|f\|_{L_p}=1$, using Minkowski's inequality, we have
\begin{align*}
\|G_{t,s_{1}}f-G_{t,s_{2}}f\|_{L_p}&\leq \|f\|_{L_p} \int_{\bR^d} \left|q_{\alpha,\alpha+1}(t-s_{1},\vec{x}-\vec{y})-q_{\alpha,\alpha+1}(t-s_{2},\vec{x}-\vec{y})\right| \mathrm{d}\vec{y}
\\
&\leq  \int_{\bR^{d_{1}}} \int_0^1 |\partial_{t}q_{\alpha,\alpha+1}(t-us_1-(1-u)s_2,y_{1})| |s_1-s_2| \mathrm{d}u \mathrm{d}y
\\
&\leq \frac{C|s_1-s_2|}{(t-(t_0+\delta))^2},
\end{align*}
where the last inequality holds due to Lemma \ref{lem 06.08.14:52}. Here, recall that $\cK(t,s)=0$ if $t\leq s$. This yields that
\begin{align*}
\|G_{t,s_{1}}-G_{t,s_{2}}\|_{\Lambda} \leq \frac{C|s_1-s_2|}{(t-(t_0+\delta))^2}.
\end{align*}
where $\|\cdot\|_{\Lambda}$ denotes the operator norm on $L_p(\R^{d})$. Therefore,
\begin{align*}
&\int_{\bR\setminus Q^*} \|G_{t,s_{1}}-G_{t,s_{2}}\|_{\Lambda} \mathrm{d}t \leq C \int_{\bR\setminus Q^*}\frac{|s_1-s_2|}{(t-(t_0+\delta))^2} \mathrm{d}t
\\
&\leq C|s_1-s_2|\int_{|t-(t_0+\delta)|\geq \delta}\frac{1}{(t-(t_0+\delta))^2} \mathrm{d}t \leq N\delta \int_\delta^\infty t^{-2}\mathrm{d}t \leq C.
\end{align*}
Furthermore, by following the argument of \cite[Section 7]{krylov2001caideron}, one can easily check that for almost every $t$ outside of the support of $f\in C_c^\infty(\bR;L_p(\R^{d}))$,
\begin{equation*}
\cG f(t,\vec{x})=\int_{-\infty}^\infty G_{t,s}f(s,\vec{x})\mathrm{d}s
\end{equation*}
where $\cG$ denotes the  extension  to  $L_p(\bR^{d+1})$ which is verified in Step 1. Hence, by the Banach space-valued version of the Calder\'on-Zygmund theorem (\textit{e.g.} \cite[Theorem 4.1]{krylov2001caideron}), our assertion is proved for $1<q\leq p$.

For $1<p<q<\infty$, use the duality argument in Step 1 again. The theorem is proved.
\end{proof}

\mysection{Trace and Extension Theorem for Solution Spaces}\label{sec 01.17.17:00}
In this section, we establish the trace and extension theorem for the solution space $\mathbb{H}_{q,p}^{\alpha,\vec{\phi},\gamma+2}(T)$.
\begin{theorem}
\label{25.03.13.20.52}
    Let $p,q\in(1,\infty)$ and  $\alpha\in(0,1]$.
    Suppose that $\alpha q>1$.
    \begin{enumerate}[(i)]
        \item Then for any $u\in\mathbb{H}_{q,p}^{\alpha,\vec{\phi},\gamma}(T) \cap H^{\vec{\phi},\gamma+2}_{q,p}(T)$,
    $$
    \|u(0,\cdot)\|_{B_{p,q}^{\vec{\phi},\gamma+2-2/(\alpha q)}}\leq C\left(\|u\|_{\mathbb{H}_{q,p}^{\alpha,\vec{\phi},\gamma}(T)} + \|u\|_{H^{\vec{\phi},\gamma+2}_{q,p}(T)}\right),
    $$
    where $C$ is independent of $u$ and $u(0,\cdot)$.
    \item Then for any $u_{0} \in B^{\vec{\phi},\gamma+2-2/(\alpha q)}_{q,p}$, there exists $u\in \mathbb{H}_{q,p}^{\alpha,\vec{\phi},\gamma}(T) \cap H^{\vec{\phi},\gamma+2}_{q,p}(T)$ such that $u(0)=u_0$ in the sense of Definition \ref{def 01.06.16:16} with the estimate
$$
\|u\|_{\mathbb{H}_{q,p}^{\alpha,\vec{\phi},\gamma}(T)} + \|u\|_{H^{\vec{\phi},\gamma+2}_{q,p}(T)}\leq C\|u_0\|_{B^{\vec{\phi},\gamma+2-2/\alpha q}_{q,p}},
$$
where $C$ is independent of $u$, $f$ and $u_0$.
    \end{enumerate}
\end{theorem}

To prove this theorem, we employ established trace and extension results, such as those in \cite{ALM2023,CLSW2023trace,KW2025,Z05abstract,Z06abstract}. 
In particular, we utilize the framework developed in \cite{CLSW2023trace}, which provides a detailed characterization of real interpolation spaces.
Since generalized real interpolation theory plays a central role in \cite{CLSW2023trace}, we begin by recalling several fundamental concepts, following the exposition therein.

\begin{definition}
A function $\psi:\mathbb{R}_{+} \to \mathbb{R}_{+}$ is said to belong to the class $\mathcal{I}_{o}(0,1)$ if it satisfies the following conditions: 
\begin{align*}
&\sup_{t>0} \frac{\psi(\lambda t)}{\psi(t)} = o(1) \quad \text{as} \quad \lambda \downarrow 0, 
\\
&\sup_{t>0} \frac{\psi(\lambda t)}{\psi(t)} = o(\lambda) \quad \text{as} \quad \lambda \to \infty.
\end{align*}
\end{definition}
\begin{definition}
Let \( A_0 \) and \( A_1 \) be Banach spaces. The pair \( (A_0, A_1) \) is called an \emph{interpolation couple} if both \( A_0 \) and \( A_1 \) are continuously embedded in a common topological vector space \( V \). 
\end{definition}
It follows that the two subspaces of \( V \) 
\begin{equation*} 
\begin{gathered} 
A_0\cap A_1 = \{a\in V: a\in A_0,\, a\in A_1\},\\ A_0+A_1 = \{a\in V: a = a_0 + a_1,\, a_0\in A_0,\, a_1\in A_1\} 
\end{gathered} 
\end{equation*} 
are Banach spaces with the respective norms: 
\begin{equation*} 
\begin{gathered} 
\|a\|_{A_0\cap A_1} = \max(\|a\|_{A_0},\|a\|_{A_1}),\\ \|a\|_{A_0+A_1} = \inf\{\|a_0\|_{A_0}+\|a_1\|_{A_1}: a=a_0+a_1,\, a_0\in A_0,\, a_1\in A_1\}. 
\end{gathered} 
\end{equation*}
Given an interpolation couple \( (A_0, A_1) \),  we define the \( K \)-functional for \( t>0 \) as 
\begin{equation*} 
K(t, a; A_0, A_1) := \inf \{ \|a_0\|_{A_0} + t\|a_1\|_{A_1}: a=a_0+a_1,\, a_0\in A_0,\, a_1\in A_1\}. 
\end{equation*}
For measurable functions $F:\bR_+\rightarrow [0,\infty]$, a function \( \psi \in \mathcal{I}_{o}(0,1) \), and a parameter \( p \in [1,\infty] \), the functional $\Phi^{\psi}_{p}(F)$ is defined by
$$
\Phi^{\psi}_{p}(F):=\begin{cases}
    \left(\int_0^\infty \left(\psi(t^{-1})F(t) \right)^p \frac{\mathrm{d}t}{t}\right)^{1/p}\quad &\textrm{if}\quad p\in[1,\infty),\\
    \sup_{t>0}\psi(t^{-1})F(t)\quad &\textrm{if}\quad p=\infty.
    \end{cases}
$$
The interpolation space \( (A_0, A_1)_{\psi, p} \) is given by 
\begin{equation*} 
(A_0, A_1)_{\psi, p} := \{a\in A_0 + A_1 : \|a\|_{(A_0, A_1)_{\psi, p}} := \Phi^{\psi}_{p}(K(\cdot,a;A_0,A_1)) <\infty \}. 
\end{equation*}
For details on $(A_0, A_1)_{\psi, p}$, see \cite{CLSW2023trace}.

As a preliminary step for proving Theorem \ref{25.03.13.20.52}, we introduce the Littlewood-Paley characterization of \( H_{p}^{\vec{\phi},s} \).
\begin{proposition}
\label{25.02.07.16.36}
    Let $p\in(1,\infty)$ and $s\in\mathbb{R}$.
    For $f\in\mathcal{S}(\mathbb{R}^d)$, we have the equivalence
    $$
    \|f\|_{H_p^{\vec{\phi},s}}\simeq \|S_0^{\vec{\phi}}f\|_{L_p}+\left\|\left(\sum_{j=1}^{\infty}2^{js}|\Delta_j^{\vec{\phi}}f|^2\right)^{1/2}\right\|_{L_p},
    $$
    and
    $$
    \|f\|_{\mathring{H}_p^{\vec{\phi},s}}\simeq \left\|\left(\sum_{j\in\mathbb{Z}}2^{js}|\Delta_j^{\vec{\phi}}f|^2\right)^{1/2}\right\|_{L_p},
    $$
where $\mathring{H}^{\vec{\phi},s}_{p}$ is the space of distributions equipped with the norm
$\|f\|_{\mathring{H}^{\vec{\phi},s}_{p}} =: \| (\vec{\phi} \cdot \Delta_{\vec{d}})^{s/2} f\|_{L_{p}}$.
\end{proposition}
\begin{proof}
    First, we prove the second relation.
    Let $\{Z_j\}_{j\in\mathbb{Z}}$ be a sequence of independent identically distributed random variables with
    $$
    \mathbb{P}(Z_i=1)=\mathbb{P}(Z_i=-1)=\frac{1}{2}.
    $$
    One can check that
\begin{align}\label{eqn 02.10.13:21}
    2^{js/2}\mathcal{F}_d[\Delta_j^{\vec{\phi}}f](\xi)=\frac{\mathcal{F}_1[\Psi](2^{-j}m_{\vec{\phi}}(\xi))}{2^{-js/2}(m_{\vec{\phi}}(\xi))^{s/2}}(m_{\vec{\phi}}(\xi))^{s/2}\mathcal{F}_d[f](\xi)=\eta_{s/2}(2^{-j}m_{\vec{\phi}}(\xi))\mathcal{F}_d[(\vec{\phi}\cdot\Delta_{\vec{d}})^{s/2}f],
\end{align}
    where $\eta_{s/2}(\lambda):=\mathcal{F}_1[\Psi](\lambda)\lambda^{-s/2}$.
    By Khintchine's inequality and \eqref{eqn 02.10.13:21},
    \begin{align}\label{eqn 02.10.13:47}
    \left\|\left(\sum_{j\in\mathbb{Z}}2^{js}|\Delta_j^{\vec{\phi}}f|^2\right)^{1/2}\right\|_{L_p}^p&=\int_{\mathbb{R}^d}\left(\sum_{j\in\mathbb{Z}}2^{js}|\Delta_j^{\vec{\phi}}f(x)|^2\right)^{p/2}\mathrm{d}x   \nonumber
\\
    &\simeq \int_{\mathbb{R}^d}\mathbb{E}\left[\left|\sum_{j\in\mathbb{Z}}2^{js/2}\Delta_j^{\vec{\phi}}f(x)Z_j\right|^{p}\right]\mathrm{d}x   \nonumber
\\
    &=\mathbb{E}\left[\left\|M_{Z}^{\vec{\phi},s}((\vec{\phi}\cdot\Delta_{\vec{d}})^{\gamma/2}f)\right\|_{L_p}^p\right],
    \end{align}
where
$$
\mathcal{F}_d[M_{Z}^{\vec{\phi},s}f](\xi)=m_{Z}^{\vec{\phi},s}(\xi)\mathcal{F}_d[f](\xi)=\left(\sum_{j\in\mathbb{Z}}\eta_{s/2}(2^{-j}m_{\vec{\phi}}(\xi))Z_j\right)\mathcal{F}_d[f](\xi).
$$
Using the inequality
\begin{align*}
|\phi^{(n)}_{i}(\lambda)| \leq C(n) \lambda^{-n}\phi_{i}(\lambda), \quad \forall\, \lambda>0, \quad \forall\, n\in \bN,
\end{align*}
which can be derived from \eqref{23.03.08.12.52} we see that
\begin{align}\label{eqn 02.10.13:42}
	|D_{\xi^{j_1}}\cdots D_{\xi^{j_k}}m_{\vec{\phi}}(\xi)|\leq C(d) m_{\vec{\phi}}(\xi)\prod_{i=1}^{k}|\xi^{j_i}|^{-1}
\end{align}
Applying \eqref{eqn 02.10.13:42} to Fa\'a di Bruno's formula (see \cite[Proposition 1]{Hardy2006}), we obtain
$$
\left|D_{\xi^{j_1}}\cdots D_{\xi^{j_k}}\left(\sum_{j\in\mathbb{Z}}\eta_{s/2}^{\vec{\phi}}(2^{-j}m_{\vec{\phi}}(\cdot))Z_j\right)(\xi)\right|\leq C(d,\gamma,\Psi,k)\prod_{i=1}^{k}|\xi^{j_i}|^{-1}.
$$
Hence, we can apply the Marcinkiewicz multiplier theorem (\textit{e.g.} \cite[Corollary 6.2.5]{grafakos2014classical}) to deduce (recall \eqref{eqn 02.10.13:47})
\begin{align}\label{eqn 02.10.15:22}
\left\|\left(\sum_{j\in\mathbb{Z}}2^{js}|\Delta_j^{\vec{\phi}}f|^2\right)^{1/2}\right\|_{L_p}^p & \leq C \, \mathbb{E}\left[\left\|M_{Z}^{\vec{\phi},s}((\vec{\phi}\cdot\Delta_{\vec{d}})^{s/2}f)\right\|_{L_p}^p\right]  \nonumber
\\
&\leq C \|(\vec{\phi}\cdot\Delta_{\vec{d}})^{s/2}f\|_{L_p}^p=C \|f\|_{\mathring{H}_p^{\vec{\phi},s}}^p.
\end{align}
Using the duality, we also obtain the converse inequality.

Now we consider the first relation.
Since 
$$
\mathcal{F}_1[\Psi](2^{-j}\lambda)=\mathcal{F}_1[\Psi](2^{-j}\lambda)(\mathcal{F}_1[\Psi](2^{-(j-1)}\lambda)+\mathcal{F}_1[\Psi](2^{-j}\lambda)+\mathcal{F}_1[\Psi](2^{-(j+1)}\lambda)),
$$
we have
\begin{align}\label{eqn 02.10.14:56}
\Delta_j^{\vec{\phi}}=
\Delta_j^{\vec{\phi}}(\Delta_{j-1}^{\vec{\phi}}+\Delta_j^{\vec{\phi}}+\Delta_{j+1}^{\vec{\phi}}) \quad \forall\, j\in \mathbb{Z}.
\end{align}
Using \eqref{eqn 02.10.14:56} we have the following correspondence of \eqref{eqn 02.10.13:21}
\begin{align}\label{eqn 02.10.15:19}
& 2^{js/2}\mathcal{F}_d[\Delta_j^{\vec{\phi}}f](\xi)   \nonumber
\\
 =& \eta_{s/2}^{\vec{\phi}}(2^{-j}m_{\vec{\phi}}(\xi))\mathcal{F}_d[M^{\vec{\phi},s}_{\infty}(1-\vec{\phi}\cdot\Delta_{\vec{d}})^{s/2}f] \quad \forall\, j\geq 1,
\end{align}
where
\begin{align*}
M^{\vec{\phi},s}_{\infty} := (\vec{\phi}\cdot\Delta_{\vec{d}})^{s/2}(1-S^{\vec{\phi}}_0 + \Delta^{\vec{\phi}}_{0}) (1-\vec{\phi}\cdot\Delta_{\vec{d}})^{-s/2}.
\end{align*}
By following the argument from \eqref{eqn 02.10.13:47}  to \eqref{eqn 02.10.15:22} with \eqref{eqn 02.10.15:19}
$$
\|S_0^{\vec{\phi}}f\|_{L_p}+\left\|\left(\sum_{j=1}^{\infty}2^{js}|\Delta_j^{\vec{\phi}}f|^2\right)^{1/2}\right\|_{L_p} \leq C \left( \|M_0^{\vec{\phi},s}(1-\vec{\phi}\cdot\Delta_{\vec{d}})^{s/2}f\|_{L_p}+\|M_{\infty}^{\vec{\phi},s}(1-\vec{\phi}\cdot\Delta_{\vec{d}})^{s/2}f\|_{L_p} \right),
$$
where $M_0^{\vec{\phi},s}:=S^{\vec{\phi}}_0(1-\vec{\phi}\cdot\Delta_{\vec{d}})^{-s/2}$.
Using \eqref{eqn 02.10.13:42} and the Marcinkiewicz multiplier theorem, we obtain $L_{p}$-boundedness of operators
$M_0^{\vec{\phi,s}}$ and $M_{\infty}^{\vec{\phi},s}$. Hence, we prove that
\begin{align*}
\|S^{\vec{\phi}}_{0}f\|_{L_{p}} + \left\|   \left( \sum^{\infty}_{j=1} 2^{js} \left| \Delta^{\vec{\phi}}_{j}f  \right|^{2} \right)^{1/2}  \right\|_{L_{p}} \leq C \|f\|_{H^{\vec{\phi},s}_{p}}.
\end{align*}
For the converse, we observe that
$$
\|f\|_{H_p^{\vec{\phi},s}}\leq \|S_0^{\vec{\phi}}f\|_{H_p^{\vec{\phi},s}}+\left\|(1-\vec{\phi}\cdot\Delta_{\vec{d}})^{s/2}(\vec{\phi}\cdot\Delta_{\vec{d}})^{-s/2}(\vec{\phi}\cdot\Delta_{\vec{d}})^{s/2}(1-S_0^{\vec{\phi}})f\right\|_{L_p}.
$$
By the Marcinkiewicz multiplier theorem, $(1-\vec{\phi}\cdot\Delta_{\vec{d}})^{s/2}(S_0^{\vec{\phi}}+\Delta_1^{\vec{\phi}})$  bounded in $L_p$. Hence, using \eqref{eqn 02.10.14:56} we have
\begin{align}\label{eqn 02.10.18:09}
\|S^{\vec{\phi}}_{0}f\|_{H^{\vec{\phi},s}_{p}} &= \|(1-\vec{\phi}\cdot\Delta_{\vec{d}})^{s/2}S^{\vec{\phi}}_{0}f\|_{L_{p}}   \nonumber
\\
& = \|(1-\vec{\phi}\cdot\Delta_{\vec{d}})^{s/2}(S^{\vec{\phi}}_{0}+\Delta^{\vec{\phi}}_{1})S^{\vec{\phi}}_{0}f\|_{L_{p}}  \nonumber
\\
&\leq C \|S^{\vec{\phi}}_{0}f\|_{L_{p}}.
\end{align}
Since
\begin{align*}
(1-S_0)(\vec{\phi}\cdot\Delta_{\vec{d}})^{s/2}f&=\sum_{j=1}^{\infty}(\vec{\phi}\cdot\Delta_{\vec{d}})^{s/2}\Delta_j^{\vec{\phi}}f
\\
&=\sum_{j=1}^{\infty}\left(\frac{(\vec{\phi}\cdot\Delta_{\vec{d}})^{s/2}(\Delta_{j-1}^{\vec{\phi}}+\Delta_j^{\vec{\phi}}+\Delta_{j+1}^{\vec{\phi}})}{2^{js/2}}\right)2^{js/2}\Delta_j^{\vec{\phi}}f
\\
&=:\sum_{j=1}^{\infty}M_j^{\vec{\phi},s}2^{js/2}\Delta_j^{\vec{\phi}}f,
\end{align*}
for any $g\in \mathcal{S}(\mathbb{R}^{d})$, by H\"older's inequality, we have ($p'=p/(p-1)$)
\begin{align*}
\int_{\mathbb{R}^{d}}(1-S_0)(\vec{\phi}\cdot\Delta_{\vec{d}})^{s/2}f(x)g(x)\mathrm{d}x&=\int_{\mathbb{R}^d}\sum_{j=1}^{\infty}(2^{js/2}\Delta_j^{\vec{\phi}})f(x)M_j^{\vec{\phi},s}g(x)\mathrm{d}x
\\
&\leq \left\|\left(\sum_{j=1}^{\infty}2^{js}|\Delta_j^{\vec{\phi}}f|^2\right)^{1/2}\right\|_{L_p}\left\|\left(\sum_{j=1}^{\infty}|M_j^{\vec{\phi},s}g|^2\right)^{1/2}\right\|_{L_{p'}}.
\end{align*}
By following the argument from \eqref{eqn 02.10.13:47}  to \eqref{eqn 02.10.15:22} again,
\begin{align*}
    \left\|\left(\sum_{j=1}^{\infty}|M_j^{\vec{\phi},s}g|^2\right)^{1/2}\right\|_{L_{p'}}^{p'} \simeq  \int_{\mathbb{R}^d}\mathbb{E}\left[\left|\sum_{j=1}^{\infty}M_j^{\vec{\phi},s}g(x)Z_j\right|^{p'}\right]\mathrm{d}x \leq C \|g\|_{L_{p'}}^{p'}.
\end{align*}
Hence, a proper choice of $g$ gives
$$
\|(1-S_0)(\vec{\phi}\cdot\Delta_{\vec{d}})^{s/2}f\|_{L_p} \leq C \left\|\left(\sum_{j=1}^{\infty}2^{js}|\Delta_j^{\vec{\phi}}f|^2\right)^{1/2}\right\|_{L_p}.
$$
Combining this with \eqref{eqn 02.10.18:09}, we have the desired inequality.
The proposition is proved.
\end{proof}

For a Banach space $A$, by $\ell_p(A)$, we denote the set of all $A$-valued sequences $a = (a_j)_{j \in \bZ}$ satisfying $\|a\|_{\ell_p(A)}<\infty$, where 
$$
\| a \|_{\ell_p(A)}:=\begin{cases}
     \left(\sum_{j \in \bZ} \|a_j\|_{A}^{p} \right)^{1/p} \quad &\textrm{for} \quad p\in[1,\infty),\\
     \sup_{j \in \bZ} \|a_j\|_{A} \quad &\textrm{for} \quad p=\infty.
    \end{cases}
$$
Using the Littlewood-Paley characterization of the space $H_p^{\vec{\phi},s}$,  we can derive generalized real interpolation results for Sobolev and Besov spaces.
\begin{proposition}\label{exs}
    Let $p,p_0,p_1\in[1,\infty]$, $q_0,q_1,q\in[1,\infty]$, $s,s_0,s_1\in\bR$ and $\psi\in\cI_o(0,1)$, and let $s_0\neq s_1$.

(i) We have
        \begin{align*}
        (B_{p, q_0}^{\vec{\phi},s_0}, B_{p, q_1}^{\vec{\phi},s_1})_{\psi, q} &= B_{p,q}^{\vec{\phi},\psi(s_0, s_1)},\\
        (\mathring{B}_{p, q_0}^{\vec{\phi},s_0}, \mathring{B}_{p, q_1}^{\vec{\phi},s_1})_{\psi, q} &= \mathring{B}_{p,q}^{\vec{\phi},\psi(s_0, s_1)},
    \end{align*}
    where $B_{p,q}^{\vec{\phi},\psi(s_0, s_1)}$ and $\mathring{B}_{p,q}^{\vec{\phi},\psi(s_0, s_1)}$ are spaces equipped with norms given by
    \begin{align*}
 \|f\|_{B_{p,q}^{\vec{\phi},\psi(s_0,s_1)}}&:=\|S_0^{\vec{\phi}}f\|_{L_p}+\left(\sum_{j=1}^{\infty}\psi(2^{(s_1-s_0)/2})^q2^{js_0 q/2}\|\Delta_{j}^{\vec{\phi}}f\|_{L_p}^{q}\right)^{1/q},\\
 \|f\|_{\mathring{B}_{p,q}^{\vec{\phi},\gamma}}&:=\left(\sum_{j\in\mathbb{Z}}\psi(2^{(s_1-s_0)/2})^q2^{js_0 q/2}\|\Delta_{j}^{\vec{\phi}}f\|_{L_p}^{q}\right)^{1/q}
    \end{align*}

(ii) If $p\in(1,\infty)$, then 
        \begin{align*}
    (H_p^{\vec{\phi},s_0}, H_p^{\vec{\phi},s_1})_{\psi, q} &= B_{p,q}^{\vec{\phi},\psi(s_0, s_1)},\\
    (\mathring{H}_p^{\vec{\phi},s_0}, \mathring{H}_p^{\vec{\phi},s_1})_{\psi, q} &= \mathring{B}_{p,q}^{\vec{\phi},\psi(s_0, s_1)}.
\end{align*}
\end{proposition}
\begin{proof}
    (i) Consider two maps;
    \begin{equation*}
        \begin{gathered}
            f\mapsto I(f):=(S_0^{\vec{\phi}}f,\{2^{j\gamma/2}\Delta_j^{\vec{\phi}}f\}_{j\in\mathbb{N}}),\\
            \boldsymbol{f}=(f_0,f_1,\cdots)\mapsto P(\boldsymbol{f}):=(S_0^{\vec{\phi}}+\Delta_1^{\vec{\phi}})f_0+(S_0^{\vec{\phi}}+\Delta_1^{\vec{\phi}}+\Delta_{2}^{\vec{\phi}})f_1+\sum_{j=2}^{\infty}(\Delta_{j-1}^{\vec{\phi}}+\Delta_j^{\vec{\phi}}+\Delta_{j+1}^{\vec{\phi}})f_j.
        \end{gathered}
    \end{equation*}
By \eqref{eqn 02.10.14:56}, $PI$ is an identity operator on $B_{p,q}^{\vec{\phi},\gamma}$.
It can be easily checked that $I:B_{p,q}^{\vec{\phi},\gamma}\to\ell_q^{\gamma}(L_p)$ is a linear transformation with
\begin{equation}
\label{25.03.04.14.26}
\|I(f)\|_{\ell_q^{\gamma}(L_p)}= \|f\|_{B_{p,q}^{\vec{\phi},\gamma}}.
\end{equation}
Using
$$
\|\Delta_j^{\vec{\phi}}P(\boldsymbol{f})\|_{L_p}\lesssim \sum_{r=j-2}^{j+2}\|f_r\|_{L_p},
$$
we also have
\begin{equation}
\label{25.03.04.14.25}
    \|P(\boldsymbol{f})\|_{B_{p,q}^{\vec{\phi},\gamma}}\lesssim\|\boldsymbol{f}\|_{\ell_q^{\gamma}(L_p)}.
\end{equation}
Therefore, $P:\ell_q^{\gamma}(L_p)\to B_{p,q}^{\vec{\phi},\gamma}$ is a bounded linear transformation.
By \eqref{25.03.04.14.26},
\begin{align*}
    K(t,I(f);\ell_{q_0}^{s_0},\ell_{q_1}^{s_1})\leq \|I(f^0)\|_{\ell_{q_0}^{s_0}}+t\|I(f^1)\|_{\ell_{q_1}^{s_1}}=\|f^0\|_{B_{p,q_0}^{\vec{\phi},s_0}}+t\|f^1\|_{B_{p,q_1}^{\vec{\phi},s_1}},
\end{align*}
for $f=f^0+f^1$, where $f^0\in B_{p,q_0}^{\vec{\phi},s_0}$ and $f^1\in B_{p,q_1}^{\vec{\phi},s_1}$.
Taking the infimum, we have
\begin{equation}
\label{25.03.14.15.38}
K(t,I(f);\ell_{q_0}^{s_0},\ell_{q_1}^{s_1})\leq K(t,f;B_{p,q_0}^{\vec{\phi},s_0},B_{p,q_1}^{\vec{\phi},s_1}).
\end{equation}
For the converse, consider a pair $(\boldsymbol{f}^0,\boldsymbol{f}^1)\in \ell_{q_0}^{s_0}(L_p)\times\ell_{q_1}^{s_1}(L_p)$ satisfying $I(f)=\boldsymbol{f}^0+\boldsymbol{f}^1$.
Since $PI$ is an identity operator on $B_{p,q}^{\vec{\phi},\gamma}$, we have $f=P(\boldsymbol{f}^0)+P(\boldsymbol{f}^1)$.
By \eqref{25.03.04.14.25},
\begin{equation*}
K(t,f;B_{p,q_0}^{\vec{\phi},s_0},B_{p,q_1}^{\vec{\phi},s_1})\leq \|P(\boldsymbol{f}^0)\|_{B_{p,q_0}^{\vec{\phi},s_0}}+t\|P(\boldsymbol{f}^1)\|_{B_{p,q_1}^{\vec{\phi},s_1}}\lesssim \|\boldsymbol{f}^0\|_{\ell_{q_0}^{s_0}(L_p)}+t\|\boldsymbol{f}^1\|_{\ell_{q_1}^{s_1}(L_p)}.
\end{equation*}
Taking the infimum, we have
\begin{align}\label{eqn 04.28.18:17}
K(t,f;B_{p,q_0}^{\vec{\phi},s_0},B_{p,q_1}^{\vec{\phi},s_1})\lesssim K(t,I(f);\ell_{q_0}^{s_0}(L_{p}),\ell_{q_1}^{s_1}(L_{p})).
\end{align}
Using \eqref{25.03.04.14.26}, \eqref{25.03.14.15.38}, \eqref{eqn 04.28.18:17}, and the fact that $(\ell_{q_0}^{s_0}(L_p),\ell_{q_1}^{s_1}(L_p))_{\psi,q}=\ell_q^{\psi(s_0,s_1)}(L_p)$ (see \cite[Proposition A.4]{CLSW2023trace}), we have
\begin{align*}
\|f\|_{( B^{\vec{\phi},s_{0}}_{p,q_{0}}, B^{\vec{\phi},s_{1}}_{p,q_{1}} )_{\psi,q}} &= \int_{0}^{\infty} \left( \psi(t^{-1}) K(t,f;B^{\vec{\phi},s_{0}}_{p,q_{0}},B^{\vec{\phi},s_{1}}_{p,q_{1}}) \right)^{q} \frac{\mathrm{d}t}{t}
\\
& \simeq \int_{0}^{\infty} \left( \psi(t^{-1}) K(t,I(f);\ell_{q_0}^{s_0}(L_{p}),\ell_{q_1}^{s_1}(L_{p})) \right)^{q} \frac{\mathrm{d}t}{t}
\\
& = \| If \|_{\ell^{\psi(s_{0},s_{1})}_{q}(L_{p})} = \|f\|_{B^{\vec{\phi},s}_{p,q}}. 
\end{align*}
This certainly implies the desired result.
    
(ii) By Proposition \ref{25.02.07.16.36} and Minkowski's inequality, we can check that
\begin{equation*}
    \begin{gathered}
    B_{p,p}^{\vec{\phi},s}\subset H_p^{\vec{\phi},s}\subset B_{p,2}^{\vec{\phi},s}\quad \text{if}\quad 1<p\leq 2,\\
    B_{p,2}^{\vec{\phi},s}\subset H_p^{\vec{\phi},s}\subset B_{p,p}^{\vec{\phi},s}\quad \text{if}\quad p\geq 2.
    \end{gathered}
\end{equation*}
By the definition of generalized interpolation and (i), we have
$$
B_{p,q}^{\vec{\phi},\psi(s_0,s_1)}=(B_{p,p}^{\vec{\phi},s_0},B_{p,p}^{\vec{\phi},s_1})_{\psi,q}\subseteq (H_p^{\vec{\phi},s_0},H_p^{\vec{\phi},s_1})_{\psi,q}\subseteq (B_{p,2}^{\vec{\phi},s_0},B_{p,2}^{\vec{\phi},s_1})_{\psi,q}=B_{p,q}^{\vec{\phi},\psi(s_0,s_1)},\quad \text{if}\quad 1<p\leq 2,
$$
and
$$
B_{p,q}^{\vec{\phi},\psi(s_0,s_1)}=(B_{p,2}^{\vec{\phi},s_0},B_{p,2}^{\vec{\phi},s_1})_{\psi,q}\subseteq (H_p^{\vec{\phi},s_0},H_p^{\vec{\phi},s_1})_{\psi,q}\subseteq (B_{p,p}^{\vec{\phi},s_0},B_{p,p}^{\vec{\phi},s_1})_{\psi,q}=B_{p,q}^{\vec{\phi},\psi(s_0,s_1)},\quad \text{if}\quad p\geq 2.
$$
The proposition is proved.
\end{proof}

\begin{corollary}\label{thm 02.13.13:30}
Let $\alpha \in (0,1]$, $p,p_{0},p_{1} \in [1,\infty]$, $q_{0},q_{1},q\in[1,\infty]$, and $s,s_{0},s_{1} \in \mathbb{R}$. Suppose that $\alpha q >1$ and let $\psi(t) = t^{1/\alpha q}$.
\begin{enumerate}[(i)]
    \item If $s_0\neq s_1$, then
        \begin{align*}
        (B_{p, q_0}^{\vec{\phi},s_0}, B_{p, q_1}^{\vec{\phi},s_1})_{\psi, q} &= B_{p,q}^{\vec{\phi},(s_{1}-s_{0})/\alpha q+s_{0}},\\
        (\mathring{B}_{p, q_0}^{\vec{\phi},s_0}, \mathring{B}_{p, q_1}^{\vec{\phi},s_1})_{\psi, q} &= \mathring{B}_{p,q}^{\vec{\phi},(s_{1}-s_{0})/\alpha q+s_{0}}.
    \end{align*}
    \item If $s_0\neq s_1$, then
        \begin{align*}
    (H_p^{\vec{\phi},s_0}, H_p^{\vec{\phi},s_1})_{\psi, q} &= B_{p,q}^{\vec{\phi},(s_{1}-s_{0})/\alpha q+s_{0}},\\
    (\mathring{H}_p^{\vec{\phi},s_0}, \mathring{H}_p^{\vec{\phi},s_1})_{\psi, q} &= \mathring{B}_{p,q}^{\vec{\phi},(s_{1}-s_{0})/\alpha q+s_{0}}.
\end{align*}
    \item In particular, for $\gamma\in \mathbb{R}$, we have
 \begin{align*}
        (B_{p, q_0}^{\vec{\phi},\gamma+2}, B_{p, q_1}^{\vec{\phi},\gamma})_{\psi, q} &= B_{p,q}^{\vec{\phi},\gamma+2-2/(\alpha q)},\\
        (\mathring{B}_{p, q_0}^{\vec{\phi},\gamma+2}, \mathring{B}_{p, q_1}^{\vec{\phi},\gamma})_{\psi, q} &= \mathring{B}_{p,q}^{\vec{\phi},\gamma+2-2/(\alpha q)},
    \end{align*}
and
\begin{align*}
    (H_p^{\vec{\phi},\gamma+2}, H_p^{\vec{\phi},\gamma})_{\psi, q} &= B_{p,q}^{\vec{\phi},\gamma+2-2/(\alpha q)},\\
    (\mathring{H}_p^{\vec{\phi},\gamma+2}, \mathring{H}_p^{\vec{\phi},\gamma})_{\psi, q} &= \mathring{B}_{p,q}^{\vec{\phi},\gamma+2-2/(\alpha q)}.
\end{align*}
\end{enumerate}
\end{corollary}
\begin{proof}
We only need to observe that $\psi(t) = t^{1/\alpha q} \in \mathcal{I}_{o}(0,1)$, which can be easily checked by a direct computation.
\end{proof}

We conclude this section with the proof of Theorem \ref{25.03.13.20.52}.

\begin{proof}[Proof of Theorem \ref{25.03.13.20.52}]
It suffices to adapt the framework provided in \cite{CLSW2023trace}.
We first consider the time non-local (\textit{i.e.} $\alpha \in (0,1)$) case.
By setting  $W(t) = t$, $\kappa(t) = t^{-\alpha}/\Gamma(1-\alpha)$ and $\kappa^{\ast}(t) = \kappa^{-1}(t^{-1}) = (\Gamma(1-\alpha)t)^{1/\alpha}$ we obtain $(W\circ\kappa^{\ast})^{1/q}(t) = (\Gamma(1-\alpha)t)^{1/\alpha q}$. 
Applying Corollary \ref{thm 02.13.13:30}, we have
\begin{align*}
(H^{\vec{\phi},\gamma+2}_{p},H^{\vec{\phi},\gamma}_{p})_{(W\circ\kappa^{\ast}),q} = B^{\vec{\phi},\gamma+2-2/\alpha q}_{p,q}.
\end{align*}
Then, statement (i) follows directly from \cite[Theorem 5.3]{CLSW2023trace}.

Moreover, according to \cite[Theorem 1.6]{CLSW2023trace}, for each \(u_0\in B^{\vec{\phi},\gamma+2-2/(\alpha q)}_{p,q}\), there exist \(u \in L_{q}(\mathbb{R}_{+};H^{\vec{\phi},\gamma+2}_{p})\) and \(f \in L_{q}(\mathbb{R}_{+};H^{\vec{\phi},\gamma}_{p})\) satisfying \(\partial_t^{\alpha}(u-u_0)=f\) along with the estimate
\[
\|u\|_{L_{q}(\mathbb{R}_{+};H^{\vec{\phi},\gamma+2}_{p})} + \|f\|_{L_{q}(\mathbb{R}_{+};H^{\vec{\phi},\gamma}_{p})} \leq C \|u_{0}\|_{B^{\vec{\phi},\gamma+2-2/(\alpha q)}_{p,q}},
\]
where the constant $C$ is independent of $u_{0},u,f$.
Since
$$
\|u\|_{\mathbb{H}^{\alpha,\vec{\phi},\gamma}_{q,p}(T)} + \|u\|_{H^{\vec{\phi},\gamma+2}_{q,p}(T)} \leq \|u\|_{L_{q}(\mathbb{R}_{+};H^{\vec{\phi},\gamma+2}_{p})} + \|f\|_{L_{q}(\mathbb{R}_{+};H^{\vec{\phi},\gamma}_{p})}+\|u_0\|_{B^{\vec{\phi},\gamma+2-2/\alpha q}_{p,q}},
$$
statement (ii) immediately follows.
For time local (\textit{i.e.} $\alpha=1$) case, by following the above argument with \cite[Corollary 5.1]{CLSW2023trace} (for (i)) and \cite[Theorem 1.5]{CLSW2023trace} (for (ii)), we prove the theorem. 
\end{proof}

\mysection{Proof of Theorem \ref{main theorem}}

In this section, we prove Theorem \ref{main theorem}. 
Note that due to Proposition \ref{prop 03.21.16:12} (iii), it suffices to prove case $\gamma=0$.

\vspace{1mm}

\noindent{\textbf{Step 1}} (Existence and estimation of solution). 

\textbf{Step 1-1} We consider the case $u_{0}=0$.
For time local case (\textit{i.e.} $\alpha =1$), the theorem is a direct consequence of \cite[Theorem 2.8]{CKP23} with $\vec{a} = \vec{1}$ and $\vec{b} = \vec{b}_{0}$ therein. Hence, we only consider the case $\alpha <1$.
First, assume $f\in C_c^\infty(\bR^{d+1}_+)$.
Then by Theorem \ref{lem 06.24.15:35} a function $u(t,x)$ defined in \eqref{eqn 07.22.16:51} is a solution to equation \eqref{25.03.14.10.36}.
Moreover, $u(t,x)$ is infinitely differentiable in $(t,x)$ and hence $\partial^{\alpha}_{t}u$ exists as a function. Those facts and \eqref{eqn 06.24.16:34} imply  $u\in H^{\vec{\phi},\gamma+2}_{q,p}(T)$ and $\partial^{\alpha}_{t}u = \vec{\phi}\cdot \Delta_{\vec{d}} \, u \in H^{\vec{\phi},\gamma}_{q,p}(T)$, and hence $u\in \mathbb{H}^{\alpha,\vec{\phi},\gamma+2}_{q,p,0}(T)$.
 
Now, we show estimations \eqref{mainestimate1} and \eqref{mainestimate-111}. Take $\eta_{k}=\eta_{k}(t)\in C^{\infty}(\bR)$ such that $0\leq \eta_{k}\leq 1$, $\eta_{k}(t)=1$ for $t\leq T+1/k$ and $\eta_{k}(t)=0$ for $t\geq T+2/k$. Since $f\eta_{k}\in L_{q}(\bR;L_{p}(\bR^{d}))$, and $f(t)=f\eta_{k}(t)$ for $t\leq T$, By Theorem \ref{thm 07.22.11:31}, we have
\begin{equation*}
\begin{aligned}
\|\vec{\phi}\cdot\Delta_{\vec{d}} \, u\|_{\mathbb{L}_{q,p}(T)} &= \|\mathcal{G} f \|_{\mathbb{L}_{q,p}(T)} = \|\mathcal{G}(f\eta_{k})\|_{\mathbb{L}_{q,p}(T)} 
\\
&\leq \|\mathcal{G}(f\eta_{k})\|_{L_{q}(\bR;L_{p}(\bR^{d}))}\leq C \|f\eta_{k}\|_{L_{q}(\bR;L_{p}(\bR^{d}))}.
\end{aligned}
\end{equation*}
Hence, by the dominated convergence theorem, taking $k\to\infty$, we have
$$
\|\vec{\phi}\cdot\Delta_{\vec{d}} \, u\|_{\mathbb{L}_{q,p}(T)} \leq C \|f\|_{\mathbb{L}_{q,p}(T)}.
$$
Also, by Lemma \ref{lem 06.08.14:52} and Minkowski's inequality, we can easily check that
$$
\|u\|_{\mathbb{L}_{q,p}(T)} \leq C(T) \|f\|_{\mathbb{L}_{q,p}(T)}.
$$
Therefore, using the above inequalities and Lemma \ref{eqn 03.25.15:03}, we prove estimations \eqref{mainestimate1} and \eqref{mainestimate-111}.
For general $f$,  we take a sequence of functions   $f_{n}\in \Ccinf(\R^{d+1}_{+})$  such that $f_n \to f$ in $\bL_{q,p}(T)$. Let $u_{n}$  denote the solution with representation  \eqref{eqn 07.22.16:51} with $f_{n}$ in place of $f$. 
Then \eqref{mainestimate1} applied to $u_m-u_n$ shows that $u_{n}$ is a Cauchy sequence in $\bH^{\alpha,\vec{\phi},0}_{q,p,0}(T) \cap H^{\vec{\phi},2}_{q,p}(T)$. 
By taking $u$ as the limit of $u_{n}$ in $\bH^{\alpha,\vec{\phi},2}_{q,p,0}(T) \cap H^{\vec{\phi},2}_{q,p}(T)$, we find that $u$ satisfies \eqref{mainequation1}. Also, the estimations  \eqref{mainestimate1} and \eqref{mainestimate-111} directly follows. 

\textbf{Step 1-2} Now we consider non-trivial initial condition (\textit{i.e.} $u_{0} \neq 0$).

Recall that we consider non-trivial initial condition only when $\alpha q >1$.
Hence, we apply Theorem \ref{25.03.13.20.52} (ii), to obtain $v\in \mathbb{H}^{\alpha,\vec{\phi},0}_{q,p}(T) \cap H^{\vec{\phi},2}_{q,p}(T)$ satisfying
\begin{align*}
\partial^{\alpha}_{t}v = g, \quad t>0, \quad v(0,x) = u_{0},
\end{align*}
with estimation
\begin{align*}
\|v\|_{\mathbb{H}^{\alpha,\vec{\phi},0}_{q,p}(T)} + \|v\|_{H^{\vec{\phi},2}_{q,p}(T)} \leq C \|u_{0}\|_{B^{\alpha,\vec{\phi},2-2/(\alpha q)}_{p,q}}.
\end{align*}
By Step 1-1, there exists a solution $\tilde{v} \in \mathbb{H}^{\alpha,\vec{\phi},0}_{q,p,0}(T) \cap H^{\vec{\phi},2}_{q,p}(T)$  to
\begin{align*}
\partial^{\alpha}_{t}\tilde{v} = \vec{\phi}\cdot\Delta_{\vec{d}}\, \tilde{v} + f - g +\vec{\phi}\cdot \Delta_{\vec{d}}\, v, \quad 0<t<T, \quad \tilde{v}(0,x) = 0.
\end{align*}
One can check that $u=\tilde{v} + v \in \mathbb{H}^{\alpha,\vec{\phi},0}_{q,p}(T) \cap H^{\vec{\phi},2}_{q,p}(T)$ satisfies \eqref{mainequation1} and the desired estimations.

\vspace{1mm}

\noindent\textbf{Step 2} (Uniqueness of solution).  

Let $u,v\in \mathbb{H}^{\alpha,\vec{\phi},0}_{q,p}(T) \cap H^{\vec{\phi},2}_{q,p}(T)$ be solutions to \eqref{mainequation1} with $f\in L_{q,p}(T)$ and $u_{0}\in U^{\alpha.\vec{\phi},2}_{q,p}$. Then $w:=u-v \in \mathbb{H}^{\alpha,\vec{\phi},0}_{q,p,0}(T) \cap H^{\vec{\phi},2}_{q,p}(T)$ satisfies \eqref{mainequation1} with $f=0$ and $u_{0}=0$. 
By Proposition \ref{lem 01.13.15:58} (viii), there exists a sequence $w_{n}\in C^{\infty}_{c}(\bR^{d+1}_{+})$ which converges to $w$ in $\mathbb{H}^{\alpha,\vec{\phi},0}_{q,p}(T) \cap H^{\vec{\phi},2}_{q,p}(T)$. Now define
$$
f_{n} = \partial^{\alpha}_{t}w_{n}- \vec{\phi}\cdot\Delta_{\vec{d}} \, w_{n}.
$$
Making use of  Theorem \ref{lem 06.24.15:35}, we have the representation \eqref{eqn 07.22.16:51} with $f_{n}$. Therefore, Step 1 yields that $w_{n}$ satisfies estimation \eqref{mainestimate1} with $f_{n}$, which converges to $0$ in $L_{q,p}(T)$ due to its definition. Therefore, by taking $n\to\infty$, we deduce that $w=0$, and hence $u=v$ in $\mathbb{H}^{\alpha,\vec{\phi},0}_{q,p}(T) \cap H^{\vec{\phi},2}_{q,p}(T)$. The theorem is proved. \qed

\mysection{Proofs of Proposition \ref{lem 01.13.15:58} and Proposition \ref{prop 03.21.16:12}}
\label{23.03.08.12.24}
In this section, we provide the proof of Proposition \ref{lem 01.13.15:58}.
\begin{proof}[Proof of Proposition \ref{lem 01.13.15:58}]
(i) By the definition of $H^{\vec{\phi},\gamma}_{p}$, there exists a sequence $u_{0n}\in \mathcal{S}(\mathbb{R}^{d})$ which converges to $u_{0}$ in $H^{\vec{\phi},\gamma}_{p}$.
Then we can check 
\begin{align}\label{eqn 01.07.14:41}
I^{1-\alpha}_{t}u_{0n} = \frac{t^{1-\alpha}}{(1-\alpha)\Gamma(1-\alpha)}u_{0n}, \quad \partial_{t}I^{1-\alpha}_{t}u_{0n}= \frac{t^{-\alpha}}{\Gamma(1-\alpha)}u_{0n}.
\end{align}
Since $0<\alpha q<1$, a direct computation to \eqref{eqn 01.07.14:41} implies
\begin{gather}\label{eqn 01.06.15:53}
\|I^{1-\alpha}_{t}u_{0n}\|_{H^{\vec{\phi},\gamma}_{q,p}(T)} \leq C(\alpha,q)T^{(1-\alpha) +1/q} \|u_{0n}\|_{H^{\vec{\phi},\gamma}_{p}},
\\
\|\partial_{t}I^{1-\alpha}_{t}u_{0n}\|_{H^{\vec{\phi},\gamma}_{q,p}(T)} \leq C(\alpha,q)T^{1/q-\alpha} \|u_{0n}\|_{H^{\vec{\phi},\gamma}_{p}},  \nonumber
\end{gather}
and 
\begin{equation*}
\begin{gathered}\|I^{1-\alpha}_{t}(u_{0n}-u_{0m})\|_{H^{\vec{\phi},\gamma}_{q,p}(T)} \leq C(\alpha,q,T) \|u_{0n}-u_{0m}\|_{H^{\vec{\phi},\gamma}_{p}},\\
\|\partial_tI_t^{1-\alpha}(u_{0n}-u_{0m})\|_{H_{q,p}^{\vec{\phi},\gamma}(T)}\leq C(\alpha,q,T)\|u_{0n}-u_{0m}\|_{H_p^{\vec{\phi},\gamma}}.
\end{gathered}
\end{equation*}
These mean that both $I^{1-\alpha}_{t}u_{0n}$ and $\partial_{t}I^{1-\alpha}_{t}u_{0n}$ are Cauchy sequences in $H^{\vec{\phi},\gamma}_{q,p}(T)$. 
Taking $I^{1-\alpha}u_{0}$ and $\partial_{t}I^{1-\alpha}u_{0}$ as the limit of $I^{1-\alpha}_{t}u_{0n}$ and $\partial_{t}I^{1-\alpha}_{t}u_{0n}$ in $H^{\vec{\phi},\gamma}_{q,p}(T)$, we prove that $u_{0} \in \mathbb{H}^{\alpha,\vec{\phi},\gamma}_{q,p,0}(T)$ and \eqref{eqn 01.07.14:11} follows.

(ii) Clearly, $\mathbb{H}^{\alpha,\vec{\phi},\gamma}_{q,p,0}(T) \subset \mathbb{H}^{\alpha,\vec{\phi},\gamma}_{q,p}(T)$ by taking $u_{0}\equiv 0$ in Definition \ref{def 01.06.16:16}-(iii). 
Now suppose that $u\in \mathbb{H}^{\alpha,\vec{\phi},\gamma}_{q,p}(T)$, and let $u_{0}\in H^{\vec{\phi},\gamma}_{p}$ such that $u-u_{0} \in \mathbb{H}^{\alpha,\vec{\phi},\gamma}_{q,p,0}(T)$. Then by (i), $\partial_{t}I^{1-\alpha}_{t}u_{0} \in H^{\vec{\phi},\gamma}_{q,p}(T)$ exists. Hence, we deduce that $u\in \mathbb{H}^{\alpha,\vec{\phi},\gamma}_{q,p,0}(T)$ by taking
$$
f = \partial^{\alpha}_{t}u+\partial_{t}I^{1-\alpha}_{t}u_{0}\in H_{q,p}^{\vec{\phi},\gamma}(T)
$$
which fulfills \eqref{eqn 01.02.17:53}. 

(iii) Let $u_{0n}\in \mathcal{S}(\mathbb{R}^{d})$ which converges to $u_{0}$ in $H^{\vec{\phi},\gamma}_{p}$, then
\begin{equation}
\label{25.03.14.16.45}
\begin{aligned}
& \int_{0}^{T}\int_{\mathbb{R}^{d}} \left(I^{1-\alpha}_{t}(1-\vec{\phi}\cdot\Delta_{\vec{d}})^{\gamma/2}u_{0n}(t,x) \right) \partial_{t}\left((1-\vec{\phi}\cdot\Delta_{\vec{d}})^{-\gamma/2} \eta(t,x) \right) \mathrm{d}x\mathrm{d}t 
\\
=& - \int_{0}^{T}\int_{\mathbb{R}^{d}} \left((1-\vec{\phi}\cdot\Delta_{\vec{d}})^{\gamma/2}\partial_{t}I^{1-\alpha}_{t}u_{0n}(t,x)\right) \left((1-\vec{\phi}\cdot\Delta_{\vec{d}})^{-\gamma/2}\eta(t,x)\right) \mathrm{d}x\mathrm{d}t   
\end{aligned}
\end{equation}
for all $\eta\in C_c^{\infty}([0,T)\times\mathbb{R}^d)$.
By \eqref{eqn 01.06.15:53}, there exists $I_t^{1-\alpha}u_0\in H_{q,p}^{\vec{\phi},\gamma}(T)$, thus the limit of the first term of \eqref{25.03.14.16.45} also exists.
This certainly implies that the limit of the second term of \eqref{25.03.14.16.45} exists.
Since we assume that $\partial_tI_t^{1-\alpha}u_0$ exists in $H_{q,p}^{\vec{\phi},\gamma}(T)$,
\begin{equation*}
\begin{aligned}
&\lim_{n\to\infty}\int_{0}^{T}\int_{\mathbb{R}^{d}} \left((1-\vec{\phi}\cdot\Delta_{\vec{d}})^{\gamma/2}\partial_{t}I^{1-\alpha}_{t}u_{0n}(t,x)\right) \left((1-\vec{\phi}\cdot\Delta_{\vec{d}})^{-\gamma/2}\eta(t,x)\right) \mathrm{d}x\mathrm{d}t\\
=&\int_{0}^{T}\int_{\mathbb{R}^{d}} \left((1-\vec{\phi}\cdot\Delta_{\vec{d}})^{\gamma/2}\partial_{t}I^{1-\alpha}_{t}u_{0}(t,x)\right) \left((1-\vec{\phi}\cdot\Delta_{\vec{d}})^{-\gamma/2}\eta(t,x)\right) \mathrm{d}x\mathrm{d}t.
\end{aligned}
\end{equation*}
Hence, there exists $N\in\mathbb{N}$ such that
\begin{align*}
&\left|\int_{0}^{T}\int_{\mathbb{R}^{d}} \left((1-\vec{\phi}\cdot\Delta_{\vec{d}})^{\gamma/2}\partial_{t}I^{1-\alpha}_{t}u_{0n}(t,x)\right) \left((1-\vec{\phi}\cdot\Delta_{\vec{d}})^{-\gamma/2}\eta(t,x)\right) \mathrm{d}x\mathrm{d}t\right|\\
\leq& 2 \left|\int_{0}^{T}\int_{\mathbb{R}^{d}} \left((1-\vec{\phi}\cdot\Delta_{\vec{d}})^{\gamma/2}\partial_{t}I^{1-\alpha}_{t}u_{0}(t,x)\right) \left((1-\vec{\phi}\cdot\Delta_{\vec{d}})^{-\gamma/2}\eta(t,x)\right) \mathrm{d}x\mathrm{d}t\right|
\end{align*}
for all $n\geq N$.
According to the duality argument,
$$
\|\partial_tI_t^{1-\alpha}u_{0n}\|_{H_{q,p}^{\vec{\phi},\gamma}(T)}\leq 2 \|\partial_tI_t^{1-\alpha}u_{0}\|_{H_{q,p}^{\vec{\phi},\gamma}(T)} \quad \forall\, n \geq N.
$$
However, from \eqref{eqn 01.07.14:41}, $\partial_{t}I^{1-\alpha}_{t}u_{0n}$ fails to exist in $H^{\vec{\phi},\gamma}_{q,p}(T)$ unless $u_{0n}=0$ since $\alpha q  \geq 1$.
Therefore, $u_{0n}=0$, and thus $u_{0}=0$.

(iv) Suppose that $u\in \mathbb{H}^{\alpha,\vec{\phi},\gamma}_{q,p}(T)$ and let $u_{0},v_{0} \in U^{\vec{\phi},\gamma}_{p,q}$ such that $u-u_{0},u-v_{0} \in \mathbb{H}^{\alpha,\vec{\phi},\gamma}_{q,p,0}(T)$. Since $\alpha q \geq 1$, we have $U^{\alpha,\vec{\phi},\gamma}_{p,q} = B^{\vec{\phi},\gamma+2-2/(\alpha q)}_{p,q}  \subset H^{\vec{\phi},\gamma}_{p}$ by Corollary \ref{thm 02.13.13:30} (iii). Hence, $\partial_{t}I^{1-\alpha}_{t}(u_{0}-v_{0})$ exists in $H^{\vec{\phi},\gamma}_{q,p}(T)$,  $u_{0} = v_{0}$ follows due to (iii). The proposition is proved.
\end{proof}

\begin{proof}[Proof of Proposition \ref{prop 03.21.16:12}]
All of the assertions in the proposition for $\alpha=1$ are proved in \cite[Lemma 2.7]{CKP23}. Hence, we only consider the case $\alpha \in (0,1)$.
\\
(i) The definition of $H^{\vec{\phi},\gamma}_{q,p}(T)$ directly yields the statement for it. 
Thus we only consider the space $\mathbb{H}^{\alpha,\vec{\phi},\gamma}_{q,p}(T)$.
It suffices to prove only the completeness.
Suppose that $u_{n} \in \mathbb{H}^{\alpha,\vec{\phi},\gamma}_{q,p}(T)$ is a Cauchy sequence. 
We divide the proof into two cases.

\textbf{Case 1.} $\alpha q <1$. 
\\
In this case, by Proposition \ref{lem 01.13.15:58} (ii), $u_{n}$ is a Cauchy sequence in $\mathbb{H}^{\alpha,\vec{\phi},\gamma}_{q,p,0}(T)$. 
By the definition of $\mathbb{H}^{\alpha,\vec{\phi},\gamma+2}_{q,p,0}(T)$, $u_{n}$ and $\partial^{\alpha}_{t}u_{n}$ are both Cauchy sequences in $H^{\vec{\phi},\gamma}_{q,p}(T)$. Let $u$ and $f$ be the limits of $u_{n}$ and $\partial^{\alpha}_{t}u_{n}$ in $H^{\vec{\phi},\gamma}_{q,p}(T)$. Observe that
\begin{align}\label{eqn 01.13.16:26}
& \int_{0}^{T}\int_{\mathbb{R}^{d}} \left(I^{1-\alpha}_{t}(1-\vec{\phi}\cdot\Delta_{\vec{d}})^{\gamma/2}u_{n}(t,x) \right) \partial_{t}\left((1-\vec{\phi}\cdot\Delta_{\vec{d}})^{-\gamma/2} \eta(t,x) \right) \mathrm{d}x\mathrm{d}t     \nonumber
\\
&\quad = - \int_{0}^{T}\int_{\mathbb{R}^{d}} \left((1-\vec{\phi}\cdot\Delta_{\vec{d}})^{\gamma/2}\partial^{\alpha}_{t}u_{n}(t,x)\right) \left((1-\vec{\phi}\cdot\Delta_{\vec{d}})^{-\gamma/2}\eta(t,x)\right) \mathrm{d}x\mathrm{d}t.
\end{align}
Also by H\"older's inequality, one can check that
\begin{align}\label{eqn 01.13.18:53}
\| I^{1-\alpha}_{t}( 1-\vec{\phi}\cdot \Delta_{\vec{d}} )^{\gamma/2}u_{n}  \|_{L_{q,p}(T)} \leq C(\alpha,q,T) \|u_{n}\|_{H^{\vec{\phi},\gamma}_{q,p}(T)}.
\end{align}
Therefore, by taking limit $n\to\infty$ to both sides of \eqref{eqn 01.13.16:26}, we deduce that $\partial^{\alpha}_{t}u$ exists and equals $f$. This shows that $u_{n}$ converges to $u$ in $\mathbb{H}^{\alpha,\vec{\phi},\gamma}_{q,p}(T)$ by the definition of the norm $\|\cdot\|_{\mathbb{H}^{\alpha,\vec{\phi},\gamma}_{q,p}(T)}$.

\textbf{Case 2.} $\alpha q \geq 1$. 
\\
Let $u_{n0} \in U^{\alpha,\vec{\phi},\gamma}_{p,q}$($\subset H^{\vec{\phi},\gamma}_{p}$) such that $u_{n}-u_{n0} \in \mathbb{H}^{\alpha,\vec{\phi},\gamma}_{q,p,0}(T)$. Then due to the definition of the norm $\| \cdot \|_{\mathbb{H}^{\alpha,\vec{\phi},\gamma}_{q,p}(T)}$, we see that $(u_{n},\partial^{\alpha}_{t}u_{n},u_{n0})$ converge to $(u,f,u_{0})$ in $H_{q,p}^{\vec{\phi},\gamma}(T)\times H_{q,p}^{\vec{\phi},\gamma}(T)\times U_{p,q}^{\alpha,\vec{\phi},\gamma}$. 
Since $U^{\alpha,\vec{\phi},\gamma}_{p,q}$ is a closed subspace of $H^{\vec{\phi},\gamma}_{p}$, we deduce that $u_{0}\in H^{\vec{\phi},\gamma}_{p}$.  Then by following the argument in Case 1, we check that $\partial^{\alpha}_{t}u =f$.
Therefore, $u_{n}$ converges to $u$ in $\mathbb{H}^{\alpha,\vec{\phi},\gamma}_{q,p}(T)$.

(ii) It suffices to show that $u\in \mathbb{H}^{\alpha,\vec{\phi},\gamma}_{q,p,0}(T)$ given that there is a sequence $u_{n} \in \mathbb{H}^{\alpha,\vec{\phi},\gamma}_{q,p,0}(T)$ which converges to $u$ in $\mathbb{H}^{\alpha,\vec{\phi},\gamma}_{q,p}(T)$. Let $u_{0}$ be the element in $U^{\alpha,\vec{\phi},\gamma}_{p,q}$ such that $u-u_{0} \in \mathbb{H}^{\alpha,\vec{\phi},\gamma}_{q,p,0}(T)$.  Let $\varepsilon>0$ be given. Then there exists $n(\varepsilon)$ such that
$$
\|u_{0}\|_{U_{q,p}^{\alpha,\vec{\phi},\gamma}}\leq\| u - u_{n(\varepsilon)} \|_{\mathbb{H}^{\alpha,\vec{\phi},\gamma}_{q,p}(T)} \leq \varepsilon.
$$
Since $\varepsilon>0$ is arbitrary, $u \in \mathbb{H}^{\alpha,\vec{\phi},\gamma}_{q,p,0}(T)$.

(iii) Let $u \in \mathbb{H}^{\alpha,\vec{\phi},\gamma}_{q,p}(T) \cap H^{\vec{\phi},\gamma+2}_{q,p}(T)$ and let $u_{0} \in U^{\alpha,\vec{\phi},\gamma}_{q,p}$ such that $u-u_{0} \in \mathbb{H}^{\alpha,\vec{\phi},\gamma}_{q,p,0}(T)$ (if $u \in \mathbb{H}^{\alpha,\vec{\phi},\gamma}_{q,p,0}(T)$, then put $u_{0}=0$). For simplicity let $\partial^{\alpha}_{t}u=f$. Let $v= (1-\vec{\phi}\cdot \Delta_{\vec{d}})^{\nu/2}u$, $v_{0} = (1-\vec{\phi}\cdot\Delta_{\vec{d}})^{\nu/2}u_{0}$. Then $v \in H^{\vec{\phi},\gamma-\nu+2}_{q,p}(T)$ and $g:=(1-\vec{\phi}\cdot\Delta_{\vec{d}})^{\nu/2}f \in H^{\vec{\phi},\gamma-\nu}_{q,p}(T)$ due to Proposition \ref{H_p^phi,gamma space} (ii). Observe that if we set $\bar{\eta} = (1-\vec{\phi}\cdot\Delta_{\vec{d}})^{\nu/2}\eta$,
\begin{align*}
&\int_{0}^{T} \int_{\mathbb{R}^{d}} \left(I^{1-\alpha}(1-\vec{\phi}\cdot\Delta_{\vec{d}})^{(\gamma-\nu)/2}(v-v_{0})(t,x) \right) \partial_{t}\left( (1-\vec{\phi}\cdot\Delta_{\vec{d}})^{(-\gamma+\nu)/2} \eta(t,x)\right) \mathrm{d}x\mathrm{d}t
\\
& \quad  = \int_{0}^{T} \int_{\mathbb{R}^{d}} \left(I^{1-\alpha}(1-\vec{\phi}\cdot\Delta_{\vec{d}})^{\gamma/2}(u-u_{0})(t,x) \right) \partial_{t}\left( (1-\vec{\phi}\cdot\Delta_{\vec{d}})^{-\gamma/2} \bar{\eta}(t,x)\right) \mathrm{d}x\mathrm{d}t
\\
&\quad  = - \int_{0}^{T} \int_{\mathbb{R}^{d}} (1-\vec{\phi}\cdot\Delta_{\vec{d}})^{\gamma/2}f(t,x) \left( (1-\vec{\phi}\cdot\Delta_{\vec{d}})^{-\gamma/2}  \bar{\eta}(t,x) \right) \mathrm{d}x\mathrm{d}t
\\
&\quad = - \int_{0}^{T} \int_{\mathbb{R}^{d}} (1-\vec{\phi}\cdot\Delta_{\vec{d}})^{(\gamma-\nu)/2}g(t,x)  \left( (1-\vec{\phi}\cdot\Delta_{\vec{d}})^{-(\gamma-\nu)/2} \eta(t,x) \right)\mathrm{d}x\mathrm{d}t.
\end{align*}
This implies that 
$$
\partial^{\alpha}_{t}(1-\vec{\phi}\cdot\Delta_{\vec{d}})^{\nu/2}u=\partial^{\alpha}_{t}v = g = (1-\vec{\phi}\cdot\Delta_{\vec{d}})^{\nu/2}f = (1-\vec{\phi}\cdot\Delta_{\vec{d}})^{\nu/2} \partial^{\alpha}_{t}u.
$$
Also, since $(1-\vec{\phi}\cdot\Delta_{\vec{d}})^{\nu/2}$ is a isometry from $H^{\vec{\phi},s}_{q,p}(T)$ to $H^{\vec{\phi},s-\nu}_{q,p}(T)$ for any $s\in \mathbb{R}$, we see that $\|v\|_{\mathbb{H}^{\alpha,\vec{\phi},\gamma-\nu}_{q,p}(T)} = \|u\|_{\mathbb{H}^{\alpha,\vec{\phi},\gamma}_{q,p}(T)}$. Hence, again using due to Proposition \ref{H_p^phi,gamma space} (ii) we have
\begin{align*}
\|v\|_{\mathbb{H}^{\alpha,\vec{\phi},\gamma-\nu}_{q,p}(T)} + \|v\|_{H^{\vec{\phi},\gamma-\nu+2}_{q,p}(T)} = \|u\|_{\mathbb{H}^{\alpha,\vec{\phi},\gamma}_{q,p}(T)} + \|u\|_{H^{\vec{\phi},\gamma+2}_{q,p}(T)}.
\end{align*}
Thus we prove the assertion.

(iv) Let $u \in \mathbb{H}^{\alpha,\vec{\phi},\gamma}_{q,p,0}(T) \cap H^{\vec{\phi},\gamma+2}_{q,p}(T)$ and let $(1-\vec{\phi}\cdot \Delta_{\vec{d}})^{\gamma/2}u = v$. Extend $v(t,x)\equiv 0$ for $t\notin[0,T]$. Take nonnegative functions $\zeta_{1} \in C^{\infty}_{c}(\mathbb{R}^{d})$, $\eta \in C^{\infty}_{c}((1,2))$ with unit integrals. For $\varepsilon_{1}>0$, define
$$
v^{(\varepsilon_{1})}(t,x) = \int_{0}^{\infty} \int_{\bR^{d}} \eta_{\varepsilon_{1}}(t-s) \zeta_{1,\varepsilon_{1}}(x-y) v(s,y) \mathrm{d}y \mathrm{d}s, \quad \eta_{\varepsilon_{1}}(t) = \varepsilon_{1}^{-1}\eta(t/\varepsilon_{1}), \quad \zeta_{1,\varepsilon_{1}}(x) = \varepsilon_{1}^{-d}\zeta(x/\varepsilon_{1}).
$$
Then $v^{(\varepsilon_{1})} \in L_{q}([0,T];H^{2n}_{p})$ for any $n\in \bN$ (indeed, it is infinitely differentiable in $(t,x)$) and 
$$
v^{(\varepsilon_{1})}(0,x) = 0 \quad \text{for all} \quad t \notin [\varepsilon_{1},T+\varepsilon_{1}], \quad x\in \mathbb{R}^{d}.
$$
Hence, $\partial^{\alpha}_{t}v^{(\varepsilon_{1})}=f^{(\varepsilon_{1})}$ exists and satisfies \eqref{eqn 01.02.17:53}. Thus we can derive the following correspondence to \eqref{eqn 01.13.18:53}
\begin{align}\label{eqn 01.15.14:03}
\|I^{1-\alpha}_{t}v^{(\varepsilon_{1})}\|_{L_{q,p}(T)} \leq C(\alpha,q,T) \|v^{(\varepsilon_{1})}\|_{H^{\vec{\phi},2}_{q,p}(T)}.
\end{align}
Also, $v^{(\varepsilon_{1})} \to v$ in $H^{\vec{\phi},2}_{q,p}(T)$ as $\varepsilon_{1} \downarrow 0$. 
Using this and \eqref{eqn 01.02.17:53}, \eqref{eqn 01.15.14:03} we can check that $f^{(\varepsilon_{1})} \to \partial^{\alpha}_{t}v$ in $L_{q,p}(T)$. This implies that $v^{(\varepsilon_{1})}$ converges to $v$ in $\mathbb{H}^{\alpha,\vec{\phi},0}_{q,p}(T) \cap H^{\vec{\phi},2}_{q,p}(T)$ as $\varepsilon_{1} \downarrow 0$. 

Now take a nonnegative function $\zeta_{2} \in C^{\infty}_{c}(\mathbb{R}^{d})$ such that $\zeta_{2}(x) = 1$ for $|x| \leq 1$ and $\zeta_{2} = 0$ for $|x|>2$. For $\varepsilon_{1},\varepsilon_{2}>0$, define
$$
v^{(\varepsilon_{1},\varepsilon_{2})}(t,x) = \zeta_{2}(\varepsilon_{2} x)v^{(\varepsilon_{1})}(t,x).
$$
Then as $\varepsilon_{2} \downarrow 0$, $v^{(\varepsilon_{1},\varepsilon_{2})}$ converges to $v^{(\varepsilon_{1})}$ in $L_{q}([0,T];H^{2n}_{p})$ for any $n\in \bN$. This deduces $v^{(\varepsilon_{1},\varepsilon_{2})}$ converges to $v^{(\varepsilon_{1})}$ in $H^{\vec{\phi},2}_{q,p}(T)$ as $\varepsilon_{2} \downarrow 0$. Similarly, we also observe that $\partial^{\alpha}_{t}v^{(\varepsilon_{1},\varepsilon_{2})}$ converges to $\partial^{\alpha}_{t}v^{(\varepsilon_{1})}$ in $L_{q,p}(T)$ as $\varepsilon_{2} \downarrow 0$, and thus $v^{(\varepsilon_{1},\varepsilon_{2})}$ converges to $v$ in $\mathbb{H}^{\alpha,\vec{\phi},0}_{q,p}(T) \cap H^{\vec{\phi},2}_{q,p}(T)$ as $\varepsilon_{1},\varepsilon_{2} \downarrow 0$.
Therefore, by (iii), $u^{(\varepsilon_{1},\varepsilon_{2})} = (1-\vec{\phi}\cdot\Delta_{d})^{-\gamma/2}v^{(\varepsilon_{1},\varepsilon_{2})} \in \mathbb{H}^{\alpha,\vec{\phi},\gamma}_{q,p}(T) \cap H^{\vec{\phi},\gamma+2}_{q,p}(T)$ converges to $u$ in $\mathbb{H}^{\alpha,\vec{\phi},\gamma}_{q,p}(T) \cap H^{\vec{\phi},\gamma+2}_{q,p}(T)$ as $\varepsilon_{1},\varepsilon_{2} \downarrow 0$. Since $v^{(\varepsilon_{1},\varepsilon_{2})} \in C^{\infty}_{c}(\mathbb{R}^{d+1}_{+})$, $u^{(\varepsilon_{1},\varepsilon_{2})}$ is also infinitely differentiable in $(t,x)$ and belongs to any $L_{q}([0,T]; H^{2n}_{p})$. Thus if we define
$$
u^{(\varepsilon_{1},\varepsilon_{2},\varepsilon_{3})}(t,x) = \zeta_{2}(\varepsilon_{3}x)u^{(\varepsilon_{1},\varepsilon_{2})}(t,x) \quad \varepsilon_{1},\varepsilon_{2},\varepsilon_{3}>0,
$$
then $u^{(\varepsilon_{1},\varepsilon_{2},\varepsilon_{3})} \in C^{\infty}_{c}(\mathbb{R}^{d+1}_{+})$ and $u^{(\varepsilon_{1},\varepsilon_{2},\varepsilon_{3})}$ converges to $u$ in $\mathbb{H}^{\alpha,\vec{\phi},\gamma}_{q,p}(T) \cap H^{\vec{\phi},\gamma+2}_{q,p}(T)$ as $\varepsilon_{1},\varepsilon_{2},\varepsilon_{3} \downarrow 0$  since $u^{(\varepsilon_{1},\varepsilon_{2},\varepsilon_{3})}$ converges to $u^{(\varepsilon_{1},\varepsilon_{2})}$ in any $L_{q}([0,T];H^{2n}_{p})$ as $\varepsilon_{3} \downarrow 0$.
Therefore, for a given $u\in \mathbb{H}^{\alpha,\vec{\phi},\gamma+2}_{q,p,0}(T)$, by taking proper sequences $a_{n},b_{n},c_{n}>0$ which converges to $0$, we can define a sequence $u_{n}=u^{(a_{n},b_{n},c_{n})} \in C^{\infty}_{c}(\mathbb{R}^{d+1}_{+})$ which converges to $u$ in $\mathbb{H}^{\alpha,\vec{\phi},\gamma+2}_{q,p}(T)$. This proves the assertion.

(v) Let $u\in \mathbb{H}^{\alpha,\vec{\phi},\gamma}_{q,p}(T) \cap H^{\vec{\phi},\gamma+2}_{q,p}(T)$ and let $u_{0} \in U^{\alpha,\vec{\phi},\gamma}_{q,p}$ stand for $u$ to satisfy $u-u_0\in \mathbb{H}^{\alpha,\vec{\phi},\gamma}_{q,p,0}(T)$. By applying a standard mollification argument used in (iv) to $u$ and $u_{0}$ we have sequences $u_{n}$ and $u_{0n}$ such that 
\begin{align*}
\|\partial^{\alpha}_{t}u - \partial^{\alpha}_{t}u_{n}\|_{H^{\vec{\phi},\gamma}_{q,p}(T)} + \|u-u_{n}\|_{H^{\vec{\phi},\gamma+2}_{q,p}(T)} + \|u_{0} - u_{0n} \|_{U^{\alpha,\vec{\phi},\gamma}_{q,p}} \to 0
\end{align*}
as $n\to\infty$ and $u_{n}-u_{0n} \in \mathbb{H}^{\alpha,\vec{\phi},\gamma}_{q,p,0}(T) \cap H^{\vec{\phi},\gamma+2}_{q,p}(T)$. Then by (iv) there exists $v_{n,k} \in C^{\infty}_{c}(\mathbb{R}^{d+1}_{+})$ which converges to $u_{n}-u_{0n}$ in $\mathbb{H}^{\alpha,\vec{\phi},\gamma}_{q,p}(T) \cap H^{\vec{\phi},\gamma+2}_{q,p}(T)$ as $k \to \infty$. Hence if we define $w_{n,k} = v_{n,k} + u_{0n}$ and take a proper subsequence $k(n)$ of $k$, $w_{n,k(n)}$ converges to $u$ in $\mathbb{H}^{\alpha,\vec{\phi},\gamma}_{q,p}(T) \cap H^{\vec{\phi},\gamma+2}_{q,p}(T)$ as $n\to\infty$. The construnction of $w_{n,k(n)}$ directly shows it belongs to $C^{\infty}_{p}([0,T]\times\mathbb{R}^{d})$. 
The proposition is proved.
\end{proof}

\vspace{10mm}

\textbf{Declarations of interest.}
Declarations of interest: none

\textbf{Data Availability.}
Data sharing not applicable to this article as no datasets were generated or analyzed during the current study.

\textbf{Acknowledgements.}
J.-H. Choi has been supported by the National Research Foundation of Korea(NRF) grant funded by the Korea government(ME) (No.RS-2023-00237315).
D. Park has been supported by a 2025 Research Grant from Kangwon National University (grant no.202504300001).
J. Seo has been supported by a KIAS Individual Grant (MG095801) at Korea Institute for Advanced Study.


\begin{thebibliography}{10}

\bibitem{ALM2023}
A. Agresti, N. Lindemulder, M. Veraar,
\newblock{On the trace embedding and its applications to evolution equations},
\newblock{\em Math. Nachr.} \textbf{296} (2023), no.4, 1319--1350.

\bibitem{AL24abstract}
E. Alvarez, C. Lizama,
\newblock{A characterization of $L^{p}$-maximal regularity for time-fractional systems in $UMD$ spaces and applications,}
\newblock{\em J. Differ. Equ.} \textbf{389} (2024), 257--284.


\bibitem{C2007extension}
L.A. Caffarelli, L. Silvestre
\newblock{An extension problem related to the fractional
Laplacian,} 
\newblock{\em Comm. Partial Differential Equations.} \textbf{32} (2007), no. 7-9, 1245–1260.




\bibitem{CKP23}
J.-H. Choi, J. Kang, D. Park,
\newblock{A regularity theory for parabolic equations with anisotropic nonlocal operators in $L_{q}(L_{p})$spaces},
\newblock{\em SIAM Journal on Mathematical Analysis},
\textbf{56} (2024), no.1, 1264--1299.




\bibitem{CK2020}
J.-H. Choi, I. Kim,
\newblock{A maximal $L_p$-regularity theory to initial value problems with time measurable nonlocal operators generated by additive processes,}
\newblock{\em Stoch. Partial Differ. Equ. Anal. Comp.} 
\textbf{12} (2023), no.1, 352-415.



\bibitem{CLSW2023trace}
J.-H. Choi, J.B. Lee, J. Seo, K. Woo, 
\newblock{On the trace theorem for Volterra-type equations with local or non-local derivatives,}
\newblock{\em arXiv preprint,} arXiv:2309.00370v2.




\bibitem{DWZ20anisotropic}
W. Deng, X. Wang, P. Zhang,
\newblock{Anisotropic nonlocal diffusion operators for normal and anomalous dynamics,}
\newblock{\em Multiscale Model. Simul.} \textbf{18} (2020), no.1, 415--443.




\bibitem{de2011anisotropic}
S. De Santis,  A. Gabrielli, M. Bozzali, B. Maraviglia, E. Macaluso, S. Capuani,
\newblock{Anisotropic anomalous diffusion assessed in the human brain by scalar invariant indices,}
\newblock{\em Magnetic Resonance in Medicine.} \textbf{65} (2011), no.4, 1043-1052.




\bibitem{DK12elliptic}
H. Dong, D. Kim,
\newblock{On $L_{p}$-estimates for a class of non-local elliptic equations,}
\newblock{\em J. Funct. Anal.} \textbf{262} (2012), no.3, 1166-1199.





\bibitem{DK19fractional}
H. Dong, D. Kim,
\newblock{$L_{p}$-estimates for time fractional parabolic equations with coefficients measurable in time,}
\newblock{\em Adv. Math.} \textbf{345} (2019), 289--345.


\bibitem{DK20div}
H. Dong, D. Kim,
\newblock{$L_{p}$-estimates for time fractional parabolic equations in divergence form with measurable coefficients,}
\newblock{\em J. Funct. Anal.} \textbf{278} (2020), no.3, 108338.


\bibitem{DK21weighted}
H. Dong, D. Kim,
\newblock{An approach for weighted mixed-norm estimates for parabolic equations with local and non-local time derivatives,}
\newblock{\em Adv. Math.} \textbf{377} (2021) 107494.

\bibitem{DK2021}
H. Dong, D. Kim,
\newblock{Time fractional parabolic equations with measurable coefficients and embeddings for fractional parabolic Sobolev spaces,}
\newblock{\em Int. Math. Res. Not. IMRN} \textbf{2021} (2021), no.22, 17563--17610.


\bibitem{DY22fractional}
H. Dong, Y. Liu,
\newblock{Weighted mixed norm estimates for fractional wave equations with VMO coefficients,}
\newblock{\em J. Differ. Equ.} \textbf{337} (2022), no.15, 168--254.

\bibitem{DK2023}
H. Dong, D. Kim,
\newblock{Time fractional parabolic equations with partially SMO coefficients,}
\newblock{\em J. Differ. Equ.}
\textbf{377} (2023), no. 25, 759--808.

\bibitem{DY23nonlocal}
H. Dong, Y. Liu, 
\newblock{Sobolev estimates for fractional parabolic equations with space-time non-local operators,}
\newblock{\em Calc. Var.} \textbf{62} (2023), no.3, 96.

\bibitem{DR2024}
H. Dong, J. Ryu,
\newblock{Nonlocal elliptic and parabolic equations with general stable operators in weighted Sobolev spaces,}
\newblock{\em SIAM Journal on Mathematical Analysis}
\textbf{56} (2024), no. 4, 4623-4661.






\bibitem{SL21}
E.B. dos Santos, R. Leit\~ao,
\newblock{On the H\"older regularity for solutions of integro-differential equations like the anisotropic fractional Laplacian,}
\newblock{\em Partial Differ. Equ. Appl.} \textbf{2} (2021), no.25, 1--34.


\bibitem{farkaspsi}
W. Farkas, N. Jacob, R.L. Schilling,
\newblock{\em Function Spaces Related to Continuous Negative Definite Functions: $\psi$-Bessel Potential Spaces.}
\newblock{Polska Akademia Nauk, Instytut Mathematyczny} (2001).



\bibitem{grafakos2014classical}
L. Grafakos, 
\newblock{\em Classical Fourier Analysis,} 
\newblock{Springer,} (2014).



\bibitem{GK25}
R. Gupta, S. Kumar,
\newblock{A Space-Time Spectral Collocation Method for Two-Dimensional Variable-Order Space-Time Fractional Advection–Diffusion Equation,}
\newblock{\em Int. J. Appl. Comput. Math.} \textbf{11} (2025), no.11.

\bibitem{H08MRI}
M.G. Hall, T.R. Barrick,
\newblock{From diffusion-weighted MRI to anomalous diffusion imaging,}
\newblock{\em Magnetic Resonance in Medicine.} 
\textbf{59} (2008), no.3, 447--455.


\bibitem{HKP20weighted}
B.-S. Han, K. Kim, D. Park,
\newblock{Weighted $L_{q}(L_{p})$-estimate with Muckenhoupt weights for the diffusion-wave equations with time-fractional derivatives,}
\newblock{\em J. Differ. Equ.} \textbf{269} (2020), no.4, 3515--3550.

\bibitem{Hardy2006}
M. Hardy,
\newblock{Combinatorics of Partial Derivatives,}
\newblock{\em Electronic Journal of Combinatorics} \textbf{13} (2006),  no.1, R1.


\bibitem{KP23}
J. Kang, D. Park,
\newblock{An $L_{q}(L_{p})$-theory for space-time non-local equations generated by L\'evy processes with low intensity of small jumps,}
\newblock{\em Stoch. Partial Differ. Equ. Anal. Comp.} \textbf{12} (2024), no.3, 1439--1491.



\bibitem{sato1999}
S. Ken-Iti, 
\newblock{L\'evy processes and infinitely divisible distributions,}
\newblock{Cambridge University Press} (1999).



\bibitem{KW2025}
D. Kim, K. Woo,
\newblock{Sobolev spaces and trace theorems for time-fractional evolution equations,}
to appear in \newblock{\em Potential Anal.}.




\bibitem{KKL17}
I. Kim, K. Kim, S. Lim,
\newblock{An $L_{q}(L_{p})$-theory for the time fractional evolution equations with variable coefficients,}
\newblock{\em Adv. Math.} \textbf{306} (2017), 123--176.



\bibitem{KPR21nonlocal}
K. Kim, D. Park, J. Ryu,
\newblock {An $L_{q}(L_{p})$-theory for diffusion equations with space-time nonlocal operators}
\newblock {\em J. Differ. Equ.} \textbf{287} (2021), no.25, 376--427.




\bibitem{krylov2001caideron}
N.V. Krylov,
\newblock{On the Calder\'on-Zygmund theorem with applications to parabolic
  equations,}
\newblock {\em St. Petersburg Math. J.}, \textbf{13} (2002), no.4, 509-526.




\bibitem{Leit20}
R. Leit\~ao,
\newblock{Almgren's frequency formula for an extension problem related to the anisotropic fractional Laplacian,}
\newblock{\em Rev. Math. Iberoam.} \textbf{36} (2020), no.3, 641--660.

\bibitem{Leit23}
R. Leit\~ao,
\newblock{$L_{p}$-estimates for solutions governed by operators like the anisotropic fractional Laplacian,}
\newblock{\em Bull. Braz. Math. Soc.} \textbf{54} (2023), no.34.



\bibitem{mikulevivcius2017p}
R. Mikulevi{\v{c}}ius, C. Phonsom,
\newblock{On ${L}^{p}$-theory for parabolic and elliptic integro-differential equations with scalable operators in the whole space,}
\newblock{\em Stoch. Partial Differ. Equ.} \textbf{5} (2017), no.4, 472-519.

\bibitem{mikulevivcius2019cauchy}
R. Mikulevi{\v{c}}ius, C. Phonsom,
\newblock{On the Cauchy problem for integro-differential equations in the scale of spaces of generalized smoothness,}
\newblock{\em Potential Anal.} \textbf{50} (2019), no.3, 467-519.





\bibitem{P23weighted}
D. Park,
\newblock{Weighted maximal $L_{q}(L_{p})$-regularity theory for time-fractional diffusion-wave equations with variable coefficients,}
\newblock{\em J. Evol. Equ.} \textbf{23} (2023), no.12.



\bibitem{P15diffadv}
Y. Povstenko,
\newblock{Space-Time-Fractional Advection Diffusion Equation in a Plane}
\newblock{\em Advances in Modelling and Control of Non-integer-Order Systems: 6th Conference on Non-integer Order Calculus and Its Applications,}  Springer, (2015), 275--284.



\bibitem{P13}
J. Pr\"uss,
\newblock{\em Evolutionary Integral Equations and Applications,}
\newblock{Birkh\"auser,} 2013.



\bibitem{schbern}
R.L. Schilling,  R. Song, Z. Vondra\v{c}ek, 
\newblock{\em Bernstein Functions: Theory and Applications,} 
\newblock{De Gruyter,} 2012.



\bibitem{stein1993harmonic}
E.M. Stein, T.S. Murphy,
\newblock{\em Harmonic Analysis: Real-Variable Methods, Orthogonality, and Oscillatory integrals,}
\newblock{Princeton University Press,} 1993.



\bibitem{Z05abstract}
R. Zacher,
\newblock{Maximal regularity of type $L_{p}$ for abstract parabolic Volterra equations,}
\newblock{\em J. Evol. Equ.} \textbf{5} (2005), 79--103.

\bibitem{Z06abstract}
R. Zacher,
\newblock{Quasilinear parabolic integro-differential equations with nonlinear boundary conditions,}
\newblock{\em Differential Integral Equations,} \textbf{19} (2006), no.10, 1129--1156.




\bibitem{Z02}
G.M. Zaslavsky,
\newblock{Chaos, fractional kinetics, and anomalous transport,}
\newblock{\em Phys. Rep,} \textbf{371} (2002), no.6, 461--580.



\bibitem{ZEN97}
G.M. Zaslavsky, M. Edelman, B.A. Niyazov,
\newblock{Self-similarity, renormalization, and phase space nonuniformity of Hamiltonian chaotic dynamics,}
\newblock{\em Chaos,} \textbf{7} (1997), 159--181.



\end{thebibliography}
\end{document}